\input ./style/arxiv-general.cfg
\documentclass[aap,MSNbibl,dvips]{arximspdf}
\makeatletter
   \@ifpackageloaded{graphicx}{}{\usepackage{graphicx}}
\makeatother
\usepackage[all,cmtip,dvips]{xy}
\usepackage[ruled,vlined]{algorithm2e}
\usepackage{eucal}

\doi{10.1214/14-AAP1061} 
\volume{25}
\issue{5}
\pubyear{2015}
\firstpage{2809}
\lastpage{2866}
\docsubty{FLA}

\makeatletter
\let\my@algocf@latexcaption\algocf@latexcaption
\let\my@addcontentsline\addcontentsline
\long\def\algocf@latexcaption#1[#2]#3{%
\def\addcontentsline##1##2##3{}%
\my@algocf@latexcaption{#1}[#2]{#3}%
\global\let\addcontentsline\my@addcontentsline%
}
\def\xrightarrow{\longrightarrow}
\newtheorem{thmm}{Theorem}[section]
\newtheorem{cor}[thmm]{Corollary}
\newtheorem{lem}[thmm]{Lemma}
\newtheorem{prop}[thmm]{Proposition}

\newproclaim{defn}[thmm]{Definition}
\newproclaim{aspt}[thmm]{Assumption}
\newproclaim{rem}[thmm]{Remark}
\newproclaim{aspta}{Assumption A}
\newproclaim{saspt}{Standing Assumption}
\newproclaim{example}[thmm]{Example}

\newcommand{\osc}{\mathop{\operatorname{osc}}}
\newcommand{\card}{\mathop{\operatorname{card}}}
\newcommand{\diam}{\mathop{\operatorname{diam}}}
\newcommand{\corr}{\mathop{\operatorname{Corr}}}
\newcommand{\tcorr}{\mathop{\widetilde{\operatorname{Corr}}}}
\newcommand{\bbX}{\mathbb{X}}
\newcommand{\bbY}{\mathbb{Y}}
\newcommand{\bbS}{\mathbb{S}}

\makeatother

\begin{document}
\begin{frontmatter}

\title{Can local particle filters beat the
curse of~dimensionality?\thanksref{T1}}
\runtitle{Local particle filters}
\thankstext{T1}{Supported in part by NSF Grants DMS-10-05575
and CAREER-DMS-1148711, and by the ARO through PECASE award
W911NF-14-1-0094.}

\begin{aug}
\author[A]{\fnms{Patrick}~\snm{Rebeschini}\ead[label=e2]{prebesch@princeton.edu}}
\and
\author[A]{\fnms{Ramon}~\snm{van Handel}\corref{}\ead[label=e3]{rvan@princeton.edu}}
\runauthor{P. Rebeschini and R. van Handel}
\affiliation{Princeton University}
\address[A]{Sherrerd Hall\\
Princeton University\\
Princeton, New Jersey 08544\\
USA \\
\printead{e2}\\
\phantom{E-mail:\ }\printead*{e3}}
\end{aug}
%

\received{\smonth{3} \syear{2013}}
\revised{\smonth{5} \syear{2014}}

%
\begin{abstract}
The discovery of particle filtering methods has enabled the use of
nonlinear filtering in a wide array of applications. Unfortunately, the
approximation error of particle filters typically grows exponentially in
the dimension of the underlying model. This phenomenon has rendered
particle filters of limited use in complex data assimilation problems.
In this paper, we argue that it is often possible, at least in
principle, to develop local particle filtering algorithms whose
approximation error is dimension-free. The key to such developments is
the decay of correlations property, which is a spatial counterpart of
the much better understood stability property of nonlinear filters. For
the simplest possible algorithm of this type, our results provide under
suitable assumptions an approximation error bound that is uniform both
in time and in the model dimension. More broadly, our results provide a
framework for the investigation of filtering problems and algorithms in
high dimension.
\end{abstract}

%

%
\end{frontmatter}

\setattribute{tocline}{skip}{\space}.
\tableofcontents[alignleft,level=2]

\section{Introduction and background}
\label{sec:intro}

A fundamental problem in a broad range of applications is the combination
of observed data and dynamical models. Particularly in highly complex
systems with partial observations, the effective extraction and
utilization of the information contained in observed data can only be
accomplished by exploiting the availability of accurate predictive models
of the underlying dynamical phenomena of interest. Such problems arise in
applications that range from classical tracking problems in navigation and
robotics to extremely large-scale problems such as weather forecasting.
In the latter setting, and in other complex applications in the
geophysical, atmospheric and ocean sciences, incorporating observed data
into dynamical models is called \emph{data assimilation}.

From a probabilistic perspective,
it is in principle simple to formulate the optimal
solution to the data assimilation problem. We model the dynamics
and observations jointly as a bivariate Markov chain $(X_n,Y_n)_{n\ge0}$
taking values in a possibly high-dimensional state space $\bbX\times
\bbY$
(throughout this paper we will consider discrete time models for
simplicity; continuous time models may also be considered).
The process $(X_n)_{n\ge0}$
describes the underlying dynamics of interest, while the process
$(Y_n)_{n\ge0}$ denotes the observed data. To estimate the hidden state
$X_n$ based on the observation history $Y_1,\ldots,Y_n$ to date, we
introduce the \emph{nonlinear filter}
\[
\pi_n = \mathbf{P}[X_n\in\cdot|Y_1,
\ldots,Y_n].
\]
If the conditional distribution $\pi_n$ can be computed, it yields an
optimal (least mean square) estimate of $X_n$ as well as a complete
representation of the uncertainty in this estimate. Moreover, an
important property of the filter is that it is recursive: $\pi_n$
depends only on $\pi_{n-1}$ and the new observation $Y_n$. This is
crucial in practice, as it allows the filter to be implemented
on-line over a long time horizon.

In practice, however, the optimal filter is almost never directly
computable: it requires the propagation of an entire conditional
distribution, which generally does not admit any efficiently computable
sufficient statistics.
The practical implementation of nonlinear filtering was therefore long
considered to be intractable until the discovery of a class of surprisingly
efficient sequential Monte Carlo algorithms, known as \emph{particle
filters}, for approximating the filter. The simplest such algorithm simply
inserts a random sampling step in the filtering recursion and approximates
the filter $\pi_n$ by the resulting empirical measure $\hat\pi_n$ (cf.
Section~\ref{sec:filtering} below). It is not difficult to show that this
gives rise to a standard Monte Carlo error
\[
\sup_{|f|\le1}\mathbf{E}\bigl|\pi_n(f)-\hat
\pi_n(f)\bigr| \le\frac{C}{\sqrt{N}},
\]
where $N$ denotes the number of particles. Moreover, a crucial insight
is that the constant $C$ typically does not depend on time $n$ due to
the stability property of nonlinear filters \cite{DG01,CMR05}, so that
the particle filter can indeed function in an on-line fashion. Particle
filters have proved to perform extraordinarily well in many classical
applications and are widely used in practice. We refer to \cite{CMR05}
for a detailed overview of particle filtering algorithms and their
analysis.

Unfortunately, despite their widespread success, particle filters have\break 
nonetheless proved to be essentially useless in truly complex data
assimilation problems. The reason for this, long known to practitioners,
has only recently been subjected to mathematical analysis in the work of
Bickel et al. \cite{BLB08,SBBA08}. Roughly speaking, the constant
$C$ in the above bound, while independent of time $n$, must typically be
exponential in the dimension of the state space of the underlying model.
This \emph{curse of dimensionality} does not affect most classical
tracking problems, whose dimension is typically of order unity, but
becomes absolutely prohibitive in large-scale data assimilation problems
such as weather forecasting where model dimensions of order $10^7$ are
routinely encountered \cite{AJSV08}. While the curse of dimensionality
problem in particle filters is now fairly well understood, there exists no
rigorous approach to date for alleviating this problem
\cite{BCJW12,vL09,Sny11}. Practical data assimilation in high-dimensional
models is therefore generally performed by means of {ad-hoc}
algorithms, frequently based on (questionable) Gaussian approximations,
that possess limited theoretical justification \cite{LS12,MH12,AJSV08}.
The development of ideas that could enable the principled use of particle
filters in high-dimensional settings remains a fundamental open problem in
data assimilation and in numerous other complex filtering problems
(e.g., multitarget tracking, tracking the spread of epidemics, traffic
flow prediction in freeway networks, etc.).

At the same time, the mathematical theory of nonlinear filtering in high
dimension has remained essentially in its infancy. Despite that the
study of large-scale interacting systems is an important topic in
contemporary probability theory (frequently motivated by problems in
statistical mechanics, e.g., \cite{Geo11,Mar04}), almost nothing is known
about the emergence of high-dimensional phenomena in the setting of
conditional distributions. It is not even entirely clear how filtering
problems in high dimension can be fruitfully formulated, and what type
of models should be investigated in this setting. Moreover, most
mathematical tools used in nonlinear filtering theory (cf.
\cite{CMR05}) are ill-suited to the investigation of the much more
delicate problems that arise in high dimension. We have recently begun
to explore high-dimensional probabilistic phenomena in nonlinear
filtering \cite{RvH13,TvH12}. The present paper arose from the
realization that such phenomena are not only of interest in their own
right, but that they can provide mechanisms that enable the
development and analysis of particle filtering algorithms in high
dimension.

The central idea of this paper is that the \emph{decay of correlations}
property of high-dimensional filtering models, which is in essence a
spatial counterpart of the much better understood stability property of
nonlinear filters, can be exploited to develop \emph{local} particle
filters that avoid the curse of dimensionality. For the simplest possible
algorithm of this type, we will prove under suitable assumptions an
approximation error bound that is uniform both in time and in the model
dimension. While it is far from clear whether this simple algorithm is of
immediate practical utility in the most complex real-world applications (a
question far beyond the scope of this paper; cf. Section~\ref
{sec:discuss}), our results provide the first rigorous proof of
concept that it is in fact possible, at least in principle, to develop
particle filtering algorithms whose approximation error is dimension-free.
A broader goal of this paper is to introduce a natural foundation for the
investigation of filtering problems and algorithms in high dimension, as
well as some basic mathematical tools for this purpose.


In the remainder of this section, we provide some essential background on
nonlinear filtering, particle filtering algorithms and the curse of
dimensionality, as well as a brief overview of the general ideas and
contributions of this paper.

\subsection{Classical filtering models and particle filters}
\label{sec:filtering}

A \emph{hidden Markov model} is a Markov chain
$(X_n,Y_n)_{n\ge0}$ whose transition probability $P$ can be factored as
\[
P \bigl((x,y),A \bigr) = \int\mathbf{1}_A \bigl(x',y'
\bigr) p \bigl(x,x' \bigr) g \bigl(x',y'
\bigr) \psi \bigl(dx' \bigr) \varphi \bigl(dy'
\bigr).
\]
Thus, $(X_n)_{n\ge0}$ is itself a Markov chain in a Polish state space
$\bbX$ with transition density $p\dvtx\bbX\times\bbX\to\mathbb
{R}_+$ with respect
to a given reference measure $\psi$, while $(Y_n)_{n\ge0}$ are
conditionally independent given $(X_n)_{n\ge0}$ in a Polish state space
$\bbY$ with transition density $g\dvtx\bbX\times\bbY\to\mathbb
{R}_+$ with respect
to a reference measure $\varphi$. This dependence structure is illustrated
in Figure~\ref{fig:hmm}. We interpret $(X_n)_{n\ge0}$ as an underlying
dynamical process that is not directly observable, while the observable
process $(Y_n)_{n\ge0}$ consists of partial and noisy observations of
$(X_n)_{n\ge0}$.
%
%
\begin{figure}

\includegraphics{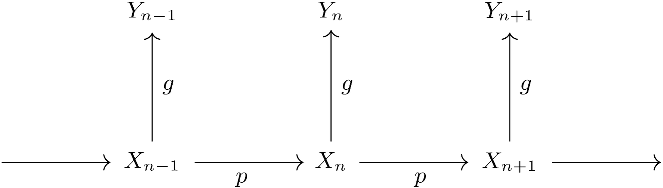}

%
\caption{Dependency graph of a hidden Markov model.}
\label{fig:hmm}
\end{figure}

In the following, we will assume that the process $(X_n,Y_n)_{n\ge0}$
is realized on its canonical probability space, and denote for any
probability measure $\mu$ on $\bbX$ by $\mathbf{P}^\mu$ the probability
measure under which $(X_n,Y_n)_{n\ge0}$ is a hidden Markov model with
transition probability $P$ as above and with initial condition
$X_0\sim\mu$. For $x\in\bbX$, we write for simplicity $\mathbf{P}^x:=
\mathbf{P}^{\delta_x}$. As the process $(X_n)_{n\ge0}$ is
unobservable, a central\vadjust{\goodbreak} problem in this setting is to track the
unobserved state $X_n$ given the observation history $Y_1,\ldots,Y_n$,
that is, we aim to compute the \emph{nonlinear filter}
\[
\pi_n^\mu:= \mathbf{P}^\mu[X_n
\in\cdot|Y_1,\ldots,Y_n].
\]
It is well known, and easily verified using the Bayes formula, that
the filter $\pi_n^\mu$ can be computed recursively, that is, we have
the recursion (see, e.g., \cite{CMR05})
\[
\pi_0^\mu=\mu, \qquad \pi_n^\mu=
\mathsf{F}_n\pi_{n-1}^\mu\qquad (n\ge1),
\]
where
\[
(\mathsf{F}_n\rho) (f):= \frac{\int f(x') g(x',Y_n) p(x,x')
\psi(dx') \rho(dx)}{\int g(x',Y_n) p(x,x') \psi(dx')
\rho(dx)}.
\]
It is instructive to write the recursion
$\mathsf{F}_n:=\mathsf{C}_n\mathsf{P}$ in two steps
\[
\pi_{n-1}^\mu\mathop{\xrightarrow}\limits^{\mathrm{prediction}}
\pi_{n-}^\mu= \mathsf{P}\pi_{n-1}^\mu
\mathop{\xrightarrow}\limits^{\mathrm{correction}} \pi_n^\mu=
\mathsf{C}_n\pi_{n-}^\mu,
\]
where
\begin{eqnarray*}
(\mathsf{P}\rho) (f) &:= &\int f \bigl(x' \bigr) p
\bigl(x,x' \bigr) \psi \bigl(dx' \bigr) \rho(dx),
\\
(\mathsf{C}_n\rho) (f) &:=& \frac{\int f(x) g(x,Y_n) \rho(dx)}{\int
g(x,Y_n) \rho(dx)}.
\end{eqnarray*}
In the prediction step, the filter $\pi_{n-1}^\mu$ is propagated forward
using the dynamics of the underlying unobserved process $(X_n)_{n\ge0}$
to compute the predictive distribution
$\pi_{n-}^\mu:=\mathbf{P}^\mu[X_n\in\cdot|Y_1,\ldots,Y_{n-1}]$. Then,
in the correction step, the predictive distribution is conditioned on
the new observation $Y_n$ to obtain the filter $\pi_n^\mu$.
%

The recursive structure of the nonlinear filter is of central
importance, as it allows the filter to be computed on-line over a
long time horizon. Nonetheless, the recursion is still at the level of
probability measures, and in general no finite-dimensional sufficient
statistics exist. Therefore, the practical implementation of nonlinear
filters typically proceeds by Monte Carlo approximation. The most common
algorithm of this type simply inserts a sampling step in the filtering
recursion: $\pi_n^\mu$ is approximated by the empirical
distribution $\hat\pi_n^\mu$ computed by the recursion
\[
\hat\pi_0^\mu=\mu,\qquad \hat\pi_n^\mu=
\mathsf{\hat F}_n\hat\pi_{n-1}^\mu\qquad (n\ge1),
\]
where $\mathsf{\hat F}_n:=\mathsf{C}_n\mathsf{S}^N\mathsf{P}$ consists
of three steps
\[
\hat\pi_{n-1}^\mu\mathop{\xrightarrow}\limits^{\mathrm{prediction}}
\mathsf{P}\hat\pi_{n-1}^\mu\mathop{\xrightarrow}\limits^{\mathrm{sampling}}
\hat\pi_{n-}^\mu= \mathsf{S}^N\mathsf{P}\hat
\pi_{n-1}^\mu\mathop{\xrightarrow}\limits^{\mathrm{correction}} \hat
\pi_n^\mu= \mathsf{C}_n\hat
\pi_{n-}^\mu.
\]
Here, $N\ge1$ is the number of particles used in the algorithm, and
$\mathsf{S}^N$ is the sampling operator that defines for a probability
measure $\rho$ the random measure
\[
\mathsf{S}^N\rho:= \frac{1}{N}\sum
_{i=1}^N \delta_{x(i)},\qquad \bigl(x(i)
\bigr)_{i=1,\ldots,N}\mbox{ are i.i.d. samples}\sim\rho
\]
[if $\rho$ is a random measure, then $(x(i))_{i=1,\ldots,N}$ are
drawn conditionally given $\rho$].
This yields the
bootstrap particle filtering algorithm described in Figure~\ref{alg:boot}.
This algorithm is exceedingly simple to implement, and
it is easily shown that the particle filter $\hat\pi_n^\mu$
converges to the exact filter $\pi_n^\mu$ as $N\to\infty$.
We refer to \cite{CMR05} for a detailed overview of particle filtering
algorithms and their analysis.

\begin{figure}
\begin{algorithm}[H]
Let $\hat\pi_0^\mu= \mu$;\\
\textbf{for} $k=1,\ldots,n$ \textbf{do}\\
\quad Sample i.i.d. $\hat x_{k-1}(i)$, $i=1,\ldots,N$ from the distribution
$\hat\pi_{k-1}^\mu$;\\
\quad Sample $x_k(i)\sim p(\hat x_{k-1}(i), \cdot) \,d\psi$, $i=1,\ldots
,N$;\\
\quad Compute $w_k(i)=g(x_k(i),Y_k)/\sum_{\ell=1}^N
g(x_k(\ell),Y_k)$, $i=1,\ldots,N$;\\
\quad Let $\hat\pi_k^\mu= \sum_{i=1}^Nw_k(i) \delta_{x_k(i)}$;\\
\textbf{end}\\
\caption{Bootstrap particle filter}
\end{algorithm}
\caption{The classical bootstrap particle filtering algorithm.}
\label{alg:boot}
\end{figure}

To gain some insight into the approximation properties of the particle
filter, let us perform the simplest possible error analysis. We
define the distance
\[
{\bigl|\!\bigl|\!\bigl| \rho-\rho' \bigr|\!\bigr|\!\bigr|} := \sup_{|f|\le1}
\mathbf{E} \bigl[\bigl|\rho(f)-\rho'(f)\bigr|^2
\bigr]^{1/2}
\]
between two random measures $\rho,\rho'$ on $\bbX$.
It is an easy exercise to show that
\[
{\bigl|\!\bigl|\!\bigl| \mathsf{P}\rho-\mathsf{P}\rho' \bigr|\!\bigr|\!\bigr|} \le{\bigl|\!\bigl|\!\bigl| \rho
- \rho' \bigr|\!\bigr|\!\bigr|}, \qquad{\bigl|\!\bigl|\!\bigl| \rho-\mathsf{S}^N\rho \bigr|\!\bigr|\!\bigr|}\le \frac{1}{\sqrt{N}}.
\]
Let us assume for simplicity that the observation density $g$ is bounded
away from zero and infinity, that is, $\kappa\le g(x,y)\le\kappa^{-1}$
for some $0<\kappa<1$. As
\begin{eqnarray*}
&&(\mathsf{C}_n\rho) (f)- \bigl(\mathsf{C}_n
\rho' \bigr) (f) \\
&&\qquad=
\frac{\kappa^{-1}}{\rho(g_n)} \bigl\{\rho(\kappa fg_n)-\rho'(
\kappa fg_n) \bigr\} + \frac{\rho'(fg_n)}{\rho'(g_n)} \frac{\kappa
^{-1}}{\rho(g_n)} \bigl\{
\rho'(\kappa g_n)-\rho(\kappa g_n) \bigr\}
\end{eqnarray*}
with $g_n(x):=g(x,Y_n)$, and as $|\kappa g_n|\le1$ and
$\rho(g_n)\ge\kappa$, we obtain
\[
{\bigl|\!\bigl|\!\bigl| \mathsf{C}_n\rho-\mathsf{C}_n
\rho' \bigr|\!\bigr|\!\bigr|} \le2\kappa^{-2}{\bigl|\!\bigl|\!\bigl| \rho-
\rho' \bigr|\!\bigr|\!\bigr|}.
\]
Putting these bounds together, we find that
\[
{\bigl|\!\bigl|\!\bigl| \pi_n^\mu-\hat\pi_n^\mu
\bigr|\!\bigr|\!\bigr|} \le2\kappa^{-2} \biggl\{\frac{1}{\sqrt{N}}+{\bigl|\!\bigl|\!\bigl|
\pi_{n-1}^\mu - \hat\pi_{n-1}^\mu \bigr|\!\bigr|\!\bigr|} \biggr\} \le\frac{\sum_{i=1}^n
(2\kappa^{-2})^i}{\sqrt{N}},
\]
where the second inequality is obtained by iterating the first
inequality $n$ times. We therefore find that the bootstrap particle
filter does indeed approximate the exact nonlinear filter with the
typical Monte Carlo $1/\sqrt{N}$-rate.

It should be noted that our crude error bound grows exponentially in
time $n$. If the error were in fact to grow exponentially in time,\vadjust{\goodbreak} this
would make the particle filter largely useless in practice as it could
not be run reliably for more than a few time steps (in particular, it
could not be run on-line over a long time horizon). Fortunately,
however, the exponential growth of the error is an artifact of our crude
bound and typically does not occur in practice. We have omitted to take
into account an essential phenomenon: ergodicity of the underlying model
will cause the filter to be \emph{stable}, that is, $\pi_n^\mu$ forgets
its initial condition $\mu$ as $n\to\infty$. The stability property
provides a dissipation mechanism that mitigates the accumulation of
approximation errors over time. A more sophisticated analysis that
exploits this idea yields a time-uniform error bound; see
Section~\ref{sec:errdecomp} below.

\subsection{The curse of dimensionality}
\label{sec:curse}

We have stated that particle filters suffer from the curse of
dimensionality. It is, however, far from obvious at this point why this
should be the case: no explicit notion of dimension appears in the above
error bound. To understand why the above bound is typically exponential
in the model dimension, we must consider a suitable class of models in
which the dependence on dimension can be explicitly investigated. In
Section~\ref{sec:main}, we will introduce a general class of
high-dimensional filtering models that is prototypical of many data
assimilation problems. In the present section, however, we consider a
much simpler class of \emph{trivial} models that is useless in any
application, but is nonetheless helpful for developing intuition for
dimensionality issues in particle filters.

In a $d$-dimensional model, $X_n,Y_n$ are each described by $d$
coordinates $X_n^i,Y_n^i$, $i=1,\ldots,d$. To construct a trivial
$d$-dimensional model, we simply start with a given one-dimensional
model and duplicate it $d$ times. That is, let $(\tilde X_n,\tilde
Y_n)_{n\ge0}$ be a hidden Markov model on $\tilde\bbX\times\tilde
\bbY$
with transition density $\tilde p$ and observation density $\tilde g$
with respect to reference measures $\tilde\psi$ and $\tilde\varphi$,
respectively. Then we set
\[
\bbX= \tilde\bbX^d, \qquad\bbY= \tilde\bbY^d,\qquad \psi= \tilde
\psi^{\otimes d},\qquad \varphi= \tilde\varphi^{\otimes d}
\]
and
\[
p(x,z) = \prod_{i=1}^d \tilde p
\bigl(x^i,z^i \bigr),\qquad  g(x,y) = \prod
_{i=1}^d \tilde g \bigl(x^i,y^i
\bigr),
\]
so that each coordinate $(X_n^i,Y_n^i)_{n\ge0}$ is an independent
copy of $(\tilde X_n,\tilde Y_n)_{n\ge0}$. Note that we have used the
term $d$-dimensional in the sense that our model has $d$ independent
degrees of freedom: each degree of freedom can itself in principle take
values in a high- or even infinite-dimensional state space
$\tilde\bbX\times\tilde\bbY$. This is, however, precisely the notion
of dimension that is relevant to the curse of dimensionality
(in~\cite{BLB08,SBBA08} this idea is sharpened by a notion of ``effective
dimension'').

In this trivial setting, it is now easily seen how the curse
of dimensionality arises in our error bound. Indeed, let us
assume again for simplicity that $\kappa\le\tilde g(\tilde x,\tilde y)
\le\kappa^{-1}$ for some $0<\kappa<1$. Then
$\kappa^d\le g(x,y)\le\kappa^{-d}$, so we obtain a bound that is
exponential in the dimension $d$ even after only one time step:
\[
{\bigl|\!\bigl|\!\bigl| \pi_1^\mu-\hat\pi_1^\mu
\bigr|\!\bigr|\!\bigr|} \le\frac{2\kappa^{-2d}}{\sqrt{N}}.
\]
An inspection of our bound clarifies the source of this exponential
growth: even though the Monte Carlo sampling itself is dimension-free
(${|\!|\!| \rho-\mathsf{S}^N\rho |\!|\!|}\le N^{-1/2}$ independent
of dimension),
the correction operator $\mathsf{C}_n$ blows
up the sampling error exponentially in high dimension (this is a
manifestation of the fact that the prior $\rho$ and
posterior $\mathsf{C}_n\rho$ measures are nearly singular in high
dimension, so that random samples drawn from $\rho$ have exponentially
small likelihood under $\mathsf{C}_n\rho$).
In particular, it
is evidently the dimension of the observations, rather than that of the
underlying model, that controls the exponential growth in our error bound.

Of course, the above analysis is far from convincing. First, we have only
proved a rather crude upper bound on the approximation error: could a more
sophisticated bound eliminate the exponential dependence on dimension as
was done using the filter stability property to eliminate the exponential
dependence on time? Second, one could argue that a good approximation of
$\pi_n(f)$ for \emph{any} function $f$ (as is implicit in the definition
of the ${|\!|\!| \cdot |\!|\!|}$-norm) is too much to ask for in high
dimension:
could a \emph{local} notion of approximation avoid the exponential
dependence on
dimension? Unfortunately, neither of these ideas can help us avoid the
curse of dimensionality of the bootstrap particle filter, which is a
genuine phenomenon and not a mathematical deficiency of our analysis. As
a simple illustration of this phenomenon, we note that even if $f(x)$
is a
function that depends on a single dimension $x^i$ only [any reasonable
approximation of $\pi_n(f)$ should work at least for such local functions]
and if $\mu=\delta_x$, the asymptotic variance $\sigma_f^2$ in the
central limit theorem
\[
\sqrt{N} \bigl\{\pi_1^\mu(f)-\hat\pi_1^\mu(f)
\bigr\} \Longrightarrow N \bigl(0,\sigma_f^2 \bigr)\qquad
\mbox{as } N\to\infty
\]
grows exponentially in the dimension $d$ (the computation of $\sigma
_f$ is
a simple exercise that is left to the interested reader), which suggests
that our crude upper bound is qualitatively correct. The more delicate
analysis of Bickel et al. \cite{BLB08,SBBA08}, which allows $d$ to
grow with $N$, demonstrates conclusively that the bootstrap particle
filter cannot approximate the filter unless the number of particles $N$
grows exponentially in the dimension $d$. Nonetheless, both the ideas
raised above to eliminate the exponential dependence on dimension will
play an important role in the remainder of this paper, as will be
explained in the next section.

%
\begin{rem}
\label{rem:mcmc}
The problem of sampling from a weighted measure of the form
$(\mathsf{C}\rho)(dx) := g(x)\rho(dx)/\int g(z)\rho(dz)$ appears in
numerous applications in statistics, computer science and physics. The
naive approximation $\mathsf{C}\rho\approx\mathsf{C}\mathsf
{S}^N\rho$ is
well known to be useless in large-scale problems: instead, Markov Chain
Monte Carlo (MCMC) methods are almost universally used for this purpose.
However, even if we were able to sample \emph{exactly} from the
weighted measure $\mathsf{C}\rho$, this would still not resolve our
problems in the filtering context. Indeed, if we implement
the ``optimal proposal'' (cf. \cite{Sny11}) particle filtering recursion
$\mathsf{\hat F}_n=\mathsf{S}^N\mathsf{C}_n\mathsf{P}$ rather than the
bootstrap filter $\mathsf{\hat F}_n=\mathsf{C}_n\mathsf{S}^N\mathsf{P}$,
then the error between $\hat\pi_1^\mu=\mathsf{\hat F}_1\mu$ and
$\pi_1^\mu=\mathsf{F}_1\mu$ would be dimension-free, but the error between
$\hat\pi_2^\mu=\mathsf{\hat F}_2\hat\pi_1^\mu$ and
$\pi_2^\mu=\mathsf{F}_2\pi_1^\mu$ would again exhibit exponential
dependence on the dimension due to the sampling performed in the first
time step. The curse of dimensionality would therefore still arise
essentially as above due to the recursive nature of the filtering
problem.

If, instead of computing the filter
$\mathbf{P}[X_n\in\cdot|Y_1,\ldots,Y_n]$, we wish to compute the
full conditional path distribution
$\mathbf{P}[X_0,\ldots,X_n\in\cdot|Y_1,\ldots,Y_n]$ (known as the
smoothing problem), MCMC methods can be successfully employed in high
dimension. However, this procedure requires the entire history of
observations and is not recursive, so that it cannot be
implemented on-line and is impractical over a long time horizon (cf.
\cite{BCJW12}). The crucial question to be addressed is therefore whether
it is possible to develop filtering algorithms that are both
recursive and that admit error bounds that are uniform in time and
in the model dimension.
\end{rem}

\subsection{Contributions of this paper}
\label{sec:contributions}

While the curse of dimensionality in particle filters is now fairly well
understood, it is far from clear how one could go about addressing this
problem. Several fundamental questions arise directly:
\begin{longlist}[1.]
\item[1.] What sort of filtering models are natural to investigate
in high dimension?
\item[2.] What sort of mechanism might allow to surmount the curse of
dimensionality? How can such a mechanism be exploited algorithmically?
\item[3.] What sort of mathematical tools are needed to address such
problems?
\end{longlist}
We aim to address each of these questions in the sequel. We will
presently provide an informal discussion of some basic ideas in this
paper; much of the remainder of the paper will be devoted to making these
ideas precise.

Some basic insight can be obtained by considering again the trivial model
of the previous section. Despite that the bootstrap particle filter
suffers from the curse of dimensionality when applied to the full model,
it is obvious in this case that one can surmount this problem in a trivial
fashion: as each of the coordinates is independent, one can simply run an
independent bootstrap filter in each coordinate. It is evident that the
local error of this algorithm (i.e., the error of the marginal of the
filter in each coordinate) is, by construction, independent of the model
dimension (i.e., the number of coordinates). Even though this approach
exploits a very special property of the trivial model---the
independence of
the coordinates---we will see that the same basic idea can be implemented
in a much more general setting.

In most data assimilation problems, the high-dimensional nature of the
model is essentially due to its spatial structure: the aim of the problem
is to track the dynamics of a random field (e.g., the atmospheric
pressure and temperature fields in the case of weather forecasting). In
this paper, we take as a starting point the notion that the coordinates
$X_n^v,Y_n^v$ $(v\in V)$ of our hidden Markov model are indexed by a large
graph $G=(V,E)$ that represents the spatial degrees of freedom of the
model, and that its interactions are local: the dynamics and observations
at a spatial location depends only on the states at locations in a
neighborhood, as is illustrated in Figure~\ref{fig:lfilt} below. While the
law of the model at each spatial location is no longer independent as in
the trivial model of the previous section, large-scale interacting systems
can nonetheless exhibit an approximate version of this property: this is
the \emph{decay of correlations} phenomenon that has been particularly
well studied in statistical mechanics \cite{Geo11}.
Informally speaking, while the states $X_n^v$ and $X_n^{w}$ at two sites
$v,w\in V$ are probably quite strongly correlated when $v$ and $w$ are
close together, one might expect that $X_n^v$ and $X_n^{w}$ are nearly
independent when $v$ and $w$ are far apart with respect to the
natural distance in the graph $G$.\setcounter{footnote}{1}\footnote{The precise formulation of the decay of correlations property
that will be used in our analysis is determined by the
mathematical machinery that will be used in the proofs;
cf. Sections~\ref{sec:dobrushin} and \ref{sec:lfstab}.}

%
%
\begin{figure}

\includegraphics{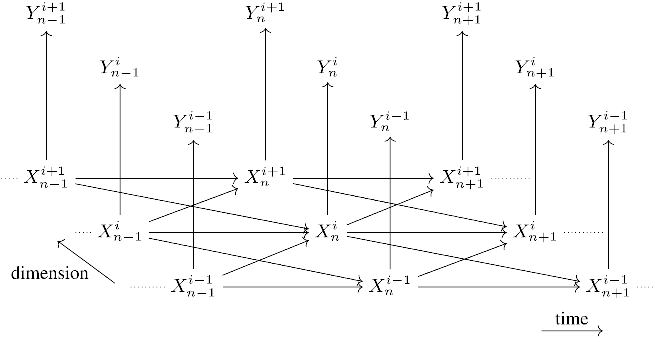}

\caption{Dependency graph of a high-dimensional filtering model
of the type considered in this paper.}
\label{fig:lfilt}
\end{figure}

The core idea of this paper is that the decay of correlations property can
provide a mechanism to surmount the curse of dimensionality. A
speculative
back-of-the-envelope computation explains how this might work. Due to the
decay of correlations, the conditional distribution of the site $X_n^v$
given the new observation $Y_n$ should not depend significantly on
observations $Y_n^w$ at sites $w$ distant from $v$. Suppose we can
develop a \emph{local} particle filtering algorithm that at each site $v$
only uses observations in a local neighborhood $K$ of $v$ to update the
filtering distribution.\vadjust{\goodbreak} As we have seen in the previous section, the
sampling error is controlled by the dimension of the observations: as
we have now restricted to observations in~$K$, the sampling error
at each site will be exponential only in $\card K$ rather than in the full
dimension $\card V$. On the other hand, the truncation to observations in
$K$ is only approximate: the decay of correlations property suggests that
the bias introduced by this truncation should decay exponentially in
$\diam K$. Therefore,
\[
\mathrm{error} = \mathrm{bias} + \mathrm{variance} \approx e^{-\diam
K} +
\frac{e^{\card K}}{\sqrt{N}}.
\]
If the size of the neighborhoods $K$ is chosen so as to optimize the
error, then the resulting algorithm is evidently consistent (with
a slower convergence rate than the standard $1/\sqrt{N}$ Monte
Carlo rate: this is likely unavoidable in high dimension) with an error
bound that is independent of the model dimension $\card V$.

The main result of this paper is that these speculative ideas can be made
precise at least for one particularly simple local filtering algorithm:
the block particle filter (Section~\ref{sec:blockfilt}). While the above
back-of-envelope computation provides a basic template for our approach,
the rigorous implementation of these ideas requires the introduction of
mathematical machinery that has not previously been applied in the study
of nonlinear filtering. Just as in the case of the filter stability
property (see~\cite{TvH12} and the references therein), it is far from
clear that any decay of correlations properties of the underlying model
are inherited by the filter as we have taken for granted above: in fact,
striking counterexamples show that such inheritance can fail in surprising
ways \cite{RvH13}. More generally, the development of machinery for the
\emph{local}
analysis of high-dimensional filtering problems forms an essential part of
our proofs. An outline of the main steps and ideas in the proof of our
main result will be given in Section~\ref{sec:outline}; detailed proofs
are given in Section~\ref{sec:proofs}.

It should be emphasized that our result, while providing a first rigorous
analysis of a local particle filtering algorithm in high dimension, is
essentially a proof of concept. The general idea to exploit decay of
correlations provides a promising approach to the curse of dimensionality
problem (such a possibility has also been occasionally mentioned in the
applied literature, e.g., \cite{vL09,SBBA08}); however, the block particle
filter that we analyze is the simplest possible algorithm of its type, and
possesses some inherent limitations that can potentially be addressed by
the development of more sophisticated local particle filters. In
Section~\ref{sec:discuss}, we will discuss some limitations of our
results and
potential directions for further investigation.\vadjust{\eject}

\section{Main result and discussion}
\label{sec:main}

\subsection{Filtering models in high dimension}
\label{sec:hdfiltmodels}

This paper is concerned with filtering problems in high dimension. In
order to investigate such problems systematically, we presently
introduce a class of high-dimensional filtering models that will provide
the basic framework to be investigated throughout this paper. In these
models, the state $(X_n,Y_n)$ at each time $n$ is a random field
$(X_n^v,Y_n^v)_{v\in V}$ indexed by a (finite) undirected graph
$G=(V,E)$. The graph $G$ describes the spatial degrees of freedom of the
model, and the underlying dynamics and observations are local with
respect to the graph structure in a sense to be made precise below. The
dimension of the model should be interpreted as the cardinality of the
vertex set $V$, which is typically assumed to be large. Our aim is to
develop quantitative results that are, under appropriate assumptions,
independent of the dimension $\card V$.

We now define the hidden Markov model $(X_n,Y_n)_{n\ge0}$
to be considered in the sequel (we will adopt throughout the basic
setting and notation introduced in Section~\ref{sec:filtering}).
The state spaces $\bbX$ and $\bbY$ of
$X_n$ and $Y_n$, and the reference measures $\psi$ and $\varphi$ of
the transition densities $p$ and $g$, respectively, are of product form
\[
\bbX= \prod_{v\in V}\bbX^v, \qquad\bbY= \prod
_{v\in V}\bbY^v,\qquad \psi= \bigotimes
_{v\in V}\psi^v, \qquad\varphi= \bigotimes
_{v\in V}\varphi^v,
\]
where $\psi^v$ and $\varphi^v$ are reference measures on the Polish spaces
$\bbX^v$ and $\bbY^v$, respectively. The transition densities $p$ and $g$
are given by
\[
p(x,z) = \prod_{v\in V}p^v
\bigl(x,z^v \bigr),\qquad g(x,y) = \prod_{v\in V}g^v
\bigl(x^v,y^v \bigr),
\]
where $p^v\dvtx\bbX\times\bbX^v\to\mathbb{R}_+$ and
$g^v\dvtx\bbX^v\times\bbY^v\to\mathbb{R}_+$
are transition densities with respect to the reference
measures $\psi^v$ and $\varphi^v$, respectively.

%

The spatial graph $G$ is endowed with
its natural distance $d$ [i.e., $d(v,v')$
is the length of the shortest path in $G$ between $v,v'\in V$].
Let us fix throughout a neighborhood size
$r\in\mathbb{N}$, and define for each vertex $v\in V$ the $r$-neighborhood
\[
N(v)= \bigl\{v'\in V\dvtx d \bigl(v,v' \bigr)\le r
\bigr\}.
\]
We will assume that the dynamics of the
underlying process $(X_n)_{n\ge0}$ is \emph{local} in the sense that
$p^v(x,z^v)$ depends on $x^{N(v)}$ only [we write $x^J=(x^j)_{j\in J}$ for
$J\subseteq V$]:
\[
p^v \bigl(x,z^v \bigr)=p^v \bigl(\tilde
x,z^v \bigr) \qquad\mbox{whenever } x^{N(v)}=\tilde
x^{N(v)}.
\]
That is, the conditional distribution of $X_n^v$ given
$X_0,\ldots,X_{n-1}$ depends on $X_{n-1}^{N(v)}$ only. Similarly, by
construction, the observations are local in that the conditional distribution
of $Y_n^v$ given $X_n$ depends on $X_n^v$ only. This dependence
structure is
illustrated in Figure~\ref{fig:lfilt} (in the simplest case of a linear
graph $G$ with $r=1$).

Markov models of the form introduced above appear in the literature under
various names, such as locally interacting Markov chains or probabilistic
cellular automata \cite{DKT90,LMS90}. Such models arise naturally in
numerous complex and large-scale applications, including percolation
models of disease spread or forest fires, freeway traffic flow models,
probabilistic models on networks and large-scale queueing systems, and
various biological, ecological and neural models. Moreover, local Markov
processes of this type arise naturally from finite-difference
approximation of stochastic partial differential equations, and are
therefore in principle applicable to a diverse set of data assimilation
problems that arise in areas such as weather forecasting, oceanography
and geophysics (cf. Section~\ref{sec:finitediff}). While more general
models are certainly of substantial interest, the model defined above is
prototypical of a broad range of high-dimensional data assimilation
problems and provides a basic setting for the investigation of filtering
problems in high dimension.

\subsection{Block particle filter: Dimension-free bounds}
\label{sec:blockfilt}

As was explained in Section~\ref{sec:curse}, the bootstrap particle filter
is not well suited to high-dimensional models: the approximation error
generally grows exponentially in the model dimension $\card V$. To
surmount this problem, we aim to develop \emph{local} particle filtering
algorithms that can exploit decay of correlations properties of the
underlying filtering model. In this paper, we will investigate in detail
the simplest possible algorithm of this type, the \emph{block particle
filter}, that will be introduced presently. While this algorithm
possesses some inherent limitations (see below), it is the simplest local
algorithm both mathematically and computationally and, therefore, provides
an ideal starting point for the investigation of particle filters in high
dimension.

%
\begin{figure}
\begin{algorithm}[H]
Let $\hat\pi_0^\mu= \mu$;\\
\textbf{for} $k=1,\ldots,n$ \textbf{do}\\
\quad Sample i.i.d. $\hat x_{k-1}(i)$, $i=1,\ldots,N$ from the distribution
$\hat\pi_{k-1}^\mu$;\\
\quad Sample $x_k^v(i)\sim p^v(\hat x_{k-1}(i), \cdot) \,d\psi^v$,
$i=1,\ldots,N$, $v\in V$;\\
\quad Compute $w_k^K(i)=\frac{\prod_{v\in K}g^v(x_k^v(i),Y_k^v)}{\sum_{\ell=1}^N
\prod_{v\in K}g^v(x_k^v(\ell),Y_k^v)}$, $i=1,\ldots,N$, $K\in
\mathcal
{K}$;\\
\quad Let $\hat\pi_k^\mu= \bigotimes_{K\in\mathcal{K}}
\sum_{i=1}^Nw_k^K(i) \delta_{x_k^K(i)}$;\\
\textbf{end}
\caption{Block particle filter}
\end{algorithm}
\caption{The block particle filtering algorithm considered in this
paper. Note that sampling $\hat x$ from a product distribution
$\bigotimes_{K\in\mathcal{K}}\rho^K$ is implemented by
sampling independently $\hat x^K\sim\rho^K$,
$K\in\mathcal{K}$.}
\label{alg:block}
\end{figure}

To define the block particle filtering algorithm, we begin by introducing
a partition $\mathcal{K}$ of the vertex set $V$ into nonoverlapping blocks,
that is, we have
\[
V = \bigcup_{K\in\mathcal{K}}K,\qquad  K\cap K'=
\varnothing\mbox{ for }K\ne K', K,K'\in\mathcal{K}.
\]
We now define the blocking operator
\[
\mathsf{B}\rho:= \bigotimes_{K\in\mathcal{K}}
\mathsf{B}^K\rho,
\]
where for any measure $\rho$ on $\bbX=\prod_{v\in V}\bbX^v$ and
$J\subseteq V$ we denote by $\mathsf{B}^J\rho$ the marginal of $\rho
$ on
$\prod_{v\in J}\bbX^v$. The random field described by the measure
$\mathsf{B}\rho$ on $\bbX$ is independent across different blocks defined
by the partition $\mathcal{K}$, while the marginal on each block agrees
with the original measure $\rho$. The block particle filter inserts
an additional blocking step into the bootstrap particle filter recursion,
that is,
\[
\hat\pi_0^\mu=\mu, \qquad\hat\pi_n^\mu=
\mathsf{\hat F}_n\hat\pi_{n-1}^\mu\qquad(n\ge1),
\]
where $\mathsf{\hat F}_n:=\mathsf{C}_n\mathsf{B}\mathsf{S}^N\mathsf{P}$
consists of four steps
\[
\hat\pi_{n-1}^\mu\mathop{\xrightarrow}\limits^{\mathrm{prediction/sampling}}
\hat\pi_{n-}^\mu= \mathsf{S}^N\mathsf{P}\hat
\pi_{n-1}^\mu\mathop{\xrightarrow}\limits^{\mathrm{blocking/correction}} \hat
\pi_n^\mu= \mathsf{C}_n\mathsf{B}\hat
\pi_{n-}^\mu.
\]
The resulting algorithm is given in Figure~\ref{alg:block}.
In the special case $\mathcal{K}=\{V\}$, the block particle filter
reduces to the bootstrap particle filter, so that the former is a strict
generalization of the latter (we have therefore not introduced a
separate notation for the bootstrap particle filter: in the sequel, the
notation $\hat\pi_n^\mu$ \emph{always refers to the block particle
filter}). The introduction of independent blocks allows to localize the
algorithm, however, which will be crucial in the high-dimensional
setting.

It is immediately evident from inspection of the block particle filtering
algorithm that only observations in block $K$ are used by the algorithm to
update the filtering distribution in block $K$. Therefore, following
the heuristic ideas of Section~\ref{sec:contributions}, we expect that the
sampling error of the algorithm is exponential in $\card K$ rather than in
the model dimension $\card V$. To control the bias introduced by the
blocking step, note that the blocking operator $\mathsf{B}\rho$ decouples
the distribution $\rho$ at the boundaries of the blocks. The decay of
correlations property (if it can be established) should cause the
influence of such a perturbation on the marginal distribution at a vertex
$v\in K$ to decay exponentially in the distance from $v$ to the boundary
of the block $K$. Thus, the back-of-the-envelope computation in
Section~\ref{sec:contributions} applies to the local error at ``most''
vertices,
as the boundaries of the blocks only constitute a small fraction of the
total number of vertices. On the other hand, the error will necessarily be
larger for vertices closer to the block boundaries. This spatial
inhomogeneity of the local error is an inherent limitation of the block
particle filter that one might hope to alleviate by the development of
more sophisticated local particle filters. We postpone further discussion
of this point to Section~\ref{sec:localalg}.

Having introduced the block particle filtering algorithm, we now proceed
to formulate the main result of this paper (Theorem~\ref{thmm:main} below).

Recall that we have introduced the neighborhoods
\[
N(v) := \bigl\{v'\in V\dvtx d \bigl(v,v' \bigr)\le r
\bigr\}
\]
above, where the neighborhood size $r$ is fixed throughout this
paper [in our model, the state of vertex $v$ depends only on the states of
vertices in $N(v)$ in the previous time step]. Given a set $J\subseteq
V$, we denote the $r$-inner boundary of $J$ as
\[
\partial J := \bigl\{v\in J\dvtx N(v)\nsubseteq J \bigr\}
\]
(i.e., $\partial J$ is the subset of vertices in $J$
that can interact with vertices outside $J$ in one step of the dynamics).
We also define the following quantities:
\begin{eqnarray*}
|\mathcal{K}|_\infty&:=& \max_{K\in\mathcal{K}}\card K,
\\
\Delta&:=& \max_{v\in V}\card \bigl\{v'\in V\dvtx d
\bigl(v,v' \bigr)\le r \bigr\},
\\
\Delta_\mathcal{K} &:=& \max_{K\in\mathcal{K}}\card \bigl
\{K'\in\mathcal{K}\dvtx d \bigl(K,K' \bigr)\le r \bigr
\},
\end{eqnarray*}
where we define as usual
$d(J,J') := \min_{v\in J}\min_{v'\in J'}d(v,v')$ for $J,J'\subseteq V$.
Thus, $|\mathcal{K}|_\infty$ is the maximal size of a block
in $\mathcal{K}$, while $\Delta$ ($\Delta_\mathcal{K}$) is
the maximal number of vertices (blocks) that interact with a single
vertex (block) in one step of the dynamics. It should be emphasized
that $r$, $\Delta$ and $\Delta_{\mathcal{K}}$ are \emph{local}
quantities that depend on the geometry but not on the size of the
spatial graph $G$.

Finally, we introduce for $J\subseteq V$ the local distance
\[
{\bigl|\!\bigl|\!\bigl| \rho-\rho' \bigr|\!\bigr|\!\bigr|}_J := \sup
_{f\in\mathcal{X}^J\dvtx
|f|\le1} \mathbf{E} \bigl[\bigl|\rho(f)-\rho'(f)\bigr|^2
\bigr]^{1/2}
\]
between random measures $\rho,\rho'$ on $\bbX$, where $\mathcal{X}^J$
denotes the class of measurable functions $f\dvtx\bbX\to\mathbb{R}$ such
that $f(x)=f(\tilde x)$ whenever $x^J=\tilde x^J$.

%
\begin{thmm}
\label{thmm:main}
There exists a constant $0<\varepsilon_0<1$, depending only on the local
quantities $\Delta$ and $\Delta_{\mathcal{K}}$, such that the
following holds.

Suppose there exist $\varepsilon_0<\varepsilon<1$ and $0<\kappa<1$ such
that
\[
\varepsilon\le p^v \bigl(x,z^v \bigr) \le
\varepsilon^{-1}, \qquad\kappa\le g^v \bigl(x^v,y^v
\bigr) \le\kappa^{-1} \qquad\forall v\in V, x,z\in\bbX, y\in\bbY.
\]
Then for every $n\ge0$, $x\in\bbX$,
$K\in\mathcal{K}$ and $J\subseteq K$ we have
\[
{\bigl|\!\bigl|\!\bigl| \pi_n^x-\hat\pi_n^x \bigr|\!\bigr|\!\bigr|}_J \le\alpha\card J \biggl[e^{-\beta_1 d(J,\partial K)} +
\frac
{e^{\beta_2|\mathcal{K}|_\infty}}{\sqrt{N}} \biggr],
\]
where the constants
$0<\alpha,\beta_1,\beta_2<\infty$ depend only on
$\varepsilon$, $\kappa$, $r$, $\Delta$ and $\Delta_{\mathcal{K}}$.
\end{thmm}

The key point of this result is that both the assumptions and the
resulting error bound depend only on \emph{local} quantities. In
particular, the assumptions and error bound depend neither on time $n$
nor on the model dimension $\card V$.

%
\begin{rem}
A threshold requirement of the form $\varepsilon>\varepsilon_0$ is
essential in order to obtain the decay of correlations property, which
can fail if $\varepsilon>0$ is too small (a~phenomenon
known as \emph{phase transition} in statistical mechanics). Otherwise,
the assumptions of Theorem~\ref{thmm:main} are comparable to assumptions
commonly imposed in the literature to obtain error bounds for the
bootstrap particle filter \cite{CMR05,DG01} and possess similar
limitations. We postpone discussion of these issues to Section~\ref
{sec:mixing}.
\end{rem}

%
\begin{rem}
In Theorem~\ref{thmm:main}, we have considered $\pi_n^x:=\pi
_n^{\delta_x}$
and $\hat\pi_n^x:=\hat\pi_n^{\delta_x}$
with a nonrandom initial condition $x\in\bbX$. This is a choice of
convenience: the proof of Theorem~\ref{thmm:main} yields the same
conclusion for more general initial conditions that satisfy a suitable
decay of correlations property. On the other hand, the stability property
of the filter (Corollary~\ref{cor:fstab} below) ensures that
$\pi_n^\mu$ forgets its initial condition $\mu$ exponentially fast
uniformly in the dimension, so there is little loss of generality in
choosing a computationally convenient initial condition.
\end{rem}

%
\begin{rem}
The particle filter $\hat\pi_n^\mu$ depends both on the random samples
that are drawn in the algorithm and on the random sequence of the
observations. However, the randomness of the observations plays no role
in our proofs. One can therefore interpret the expectation in the
definition of ${|\!|\!| \cdot |\!|\!|}_J$ as being taken only with
respect to the
random sampling mechanism in the block particle filter, and the bound of
Theorem~\ref{thmm:main} as holding uniformly with respect to the
observation sequence.
\end{rem}

To provide a concrete illustration of Theorem~\ref{thmm:main}, we consider
in the remainder of this section the example where the spatial graph $G$
is a square lattice, that is,
\[
V = \{-d,\ldots,d\}^q\qquad (d,q\in\mathbb{N})
\]
endowed with its natural edge structure. Note that in this case,
the graph distance $d(v,v')$ is simply the $\ell_1$-distance between
the corresponding vectors of integers. To define the partition
$\mathcal{K}$, we cover $V$ by blocks of radius $b\in\mathbb{N}$,
that is,
\[
\mathcal{K} = \bigl\{ \bigl(x+\{-b,\ldots,b\}^q \bigr)\cap V\dvtx x
\in(2b+1)\mathbb{Z}^q \bigr\}.
\]
We assume for simplicity in the sequel that $b\ge r$, and that
$(2d+1)/(2b+1)\in\mathbb{N}$ is integer so that all $K\in\mathcal
{K}$ are
translates
of $\{-b,\ldots,b\}^q$ (this slightly simplifies our arguments below but
is not essential to our results). We can easily compute
\[
|\mathcal{K}|_\infty= (2b+1)^q,\qquad \Delta
\le(2r+1)^q, \qquad\Delta_{\mathcal{K}} \le3^q.
\]
Note that these local quantities do not depend on the
size $d$ of our lattice. In a data assimilation application
one might have, for example, $q=2$, $r=1$, $d\sim10^3$.

Consider the block $K=\{-b,\ldots,b\}^q$. Note that for $u=0,\ldots,b-r$
\[
\bigl\{v\in K\dvtx d(v,\partial K)>u \bigr\} = \bigl\{-(b-r-u),\ldots ,b-r-u
\bigr\}^q.
\]
Fix $0<\delta<1$ and choose $u= \lfloor\delta(2b+1)/2q-r\rfloor$. Then
\[
\frac{\card\{v\in K\dvtx d(v,\partial K)>u\}}{\card K} = \biggl(\frac{2(b-r-u)+1}{2b+1} \biggr)^q \ge1-
\delta,
\]
where we have used $1-(1-\delta)^{1/q} \ge\delta/q$. The same conclusion
evidently holds for every block $K\in\mathcal{K}$. Thus, Theorem~\ref
{thmm:main} gives the following corollary.

%
\begin{cor}
\label{cor:squarelattice}
In the square lattice setting $V=\{-d,\ldots,d\}^q$, there exists a
constant $0<\varepsilon_0<1$, depending only on $r$ and
$q$, such that the following holds.

Suppose there exist $\varepsilon_0<\varepsilon<1$ and $0<\kappa<1$ such
that
\[
\varepsilon\le p^v \bigl(x,z^v \bigr) \le
\varepsilon^{-1},\qquad \kappa\le g^v \bigl(x^v,y^v
\bigr) \le\kappa^{-1}\qquad \forall v\in V, x,z\in\bbX, y\in\bbY.
\]
Then for every $x\in\bbX$, $n\ge0$ and $0<\delta<1$ we have
\[
\card \biggl\{v\in V\dvtx{\bigl|\!\bigl|\!\bigl| \pi_n^x-\hat
\pi_n^x \bigr|\!\bigr|\!\bigr|}_v \le\alpha'
e^{-\beta_1'\delta(2b+1)} + \alpha' \frac{e^{\beta
_2'(2b+1)^q}}{\sqrt{N}} \biggr\}\ge(1-\delta)
\card V,
\]
where the constants
$0<\alpha',\beta_1',\beta_2'<\infty$ depend only on
$\varepsilon$, $\kappa$, $r$ and $q$.

In particular, if we choose the block size
$b=\lfloor\frac{1}{2}(4\beta_2')^{-1/q}\log^{1/q}N-\frac
{1}{2}\rfloor$,
then
\[
\card \bigl\{v\in V\dvtx{\bigl|\!\bigl|\!\bigl| \pi_n^x-\hat
\pi_n^x \bigr|\!\bigr|\!\bigr|}_v \le c_1
e^{-c_2\delta\log^{1/q}N} \bigr\}\ge(1-\delta)\card V
\]
and (using $\mathbf{E}|Z|=\int_0^\infty\mathbf{P}[|Z|\ge t] \,dt$)
\[
\frac{1}{\card V}\sum_{v\in V} {\bigl|\!\bigl|\!\bigl|
\pi_n^x-\hat\pi_n^x \bigr|\!\bigr|\!\bigr|}_v \le\frac{c_3}{\log^{1/q}N},
\]
where the constants $0<c_1,c_2,c_3<\infty$
depend only on $\varepsilon$, $\kappa$, $r$ and $q$.
\end{cor}

Corollary~\ref{cor:squarelattice} makes precise the notion that a properly
tuned block particle filter can avoid the curse of dimensionality:
choosing the block size $b\sim\log^{1/q}N$, we obtain a local error that
can be made arbitrarily small, uniformly both in time $n$ and in the
lattice size $d$, by choosing a sufficiently large sample size $N$. More
precisely, we see that the local error at \emph{most} locations is of order
$e^{-c\log^{1/q}N}$, which is polynomial for $q=1$ and subpolynomial
otherwise, while the average local error is similarly uniform in
$n$ and $d$ albeit with a very slow convergence rate.
It appears that these results are chiefly limited by the spatial
inhomogeneity that is inherent in the block particle filtering algorithm,
as will be discussed in Section~\ref{sec:localalg} below.

%
\begin{rem}
Theorem~\ref{thmm:main} and Corollary~\ref{cor:squarelattice} should be
viewed as a theoretical proof of concept that it is possible, in
principle, to design particle filters that avoid the curse of
dimensionality. In practice, the slow rate $b\sim\log^{1/q}N$ suggests
that the block size must typically be quite small (of order unity) for
realistic values of the sample size $N$, which yields a large bias term in
our bounds. We have nonetheless observed in simple simulations that the
algorithm can work quite well even with the choice $b=0$, so that the
practical utility of the algorithm may not be fully captured by our
mathematical results. Moreover, specific features of certain data
assimilation applications, such as sparsity of observations, could make it
possible to choose substantially larger blocks. A systematic
investigation of the empirical performance of local particle filtering
algorithms in applications is beyond the scope of this paper, however. The
practical implementation of local particle filters for data assimilation
will likely require further advances in all mathematical, methodological
and applied aspects of high-dimensional filtering.
\end{rem}

\subsection{Discussion}
\label{sec:discuss}

\subsubsection{Mixing assumptions and the ergodicity threshold}
\label{sec:mixing}

The basic assumption of Theorem~\ref{thmm:main} is that the
local transition densities are bounded above and below:
\[
\varepsilon\le p^v \bigl(x,z^v \bigr) \le
\varepsilon^{-1}, \qquad\kappa\le g^v \bigl(x^v,y^v
\bigr) \le\kappa^{-1}.
\]
This is a local counterpart of the mixing assumptions that are routinely
employed in the analysis of particle filters \cite{CMR05,DG01}.
The global mixing assumption $\varepsilon\le p(x,z)
\le\varepsilon^{-1}$ would imply that the underlying Markov chain
is strongly ergodic (in the sense that its transition kernel is a strict
contraction with respect to the total variation distance) and is often
used to establish the stability property of the filter; this is
essential to
obtain a time-uniform bound on the particle filter error. See
Section~\ref{sec:errdecomp} below. The local mixing assumption
$\varepsilon\le p^v(x,z^v) \le\varepsilon^{-1}$
employed here should similarly be viewed as a local ergodicity assumption
on the model.

It is well known that strong mixing assumptions impose some constraints on
the underlying model. In particular, they typically hold only in a compact
state space: in a noncompact state space the likelihood ratio
$p(x,z)/p(x',z)$ is typically unbounded as $|z|\to\infty$, while
$\varepsilon\le p(x,z)\le\varepsilon^{-1}$ would imply that
$p(x,z)/p(x',z)$ is uniformly bounded. While qualitative results in this
area have been obtained in much more general settings (cf. \cite{TvH12}
and the references therein), it has proved to be more difficult to obtain
quantitative results under assumptions weaker than strong mixing
conditions. These technical issues are however unrelated to the problems
that arise in high dimension, and we do not address them here.

On the other hand, there is a crucial assumption in Theorem~\ref{thmm:main}
that does not arise in finite dimension. In classical results on particle
filters, it is assumed that $\varepsilon\le p(x,z)\le\varepsilon^{-1}$
with $\varepsilon>0$. For the local assumption $\varepsilon\le
p^v(x,z^v) \le\varepsilon^{-1}$, however, it is not sufficient to assume
that $\varepsilon>0$; we must assume that $\varepsilon>\varepsilon
_0$ for
some strictly positive threshold $\varepsilon_0>0$. Some assumption of
this form is absolutely essential in the high-dimensional setting.
Unlike the global mixing assumption, the local mixing assumption is not in
itself sufficient to ensure that the underlying model will remain ergodic
as the dimension $\card V\to\infty$: the cumulative effect of the
interactions can create long-range correlations that break both ergodicity
and any decay of correlations properties. Typically, the model is ergodic
when the mixing constant $\varepsilon$ is sufficiently large, but
ergodicity breaks abruptly as $\varepsilon$ drops below a threshold value
$\varepsilon_0$. Such phenomena, called \emph{phase transitions} in
statistical mechanics, are very common in large-scale interacting systems;
see \cite{LMS90,DKT90} for a number of examples. When the underlying
model fails to exhibit ergodicity and decay of correlations, we lack the
mechanism that we aim to exploit by developing local particle filters.
Therefore, some assumption of the form $\varepsilon>\varepsilon_0$ is
essential in Theorem~\ref{thmm:main} in order to ensure the presence of
decay of correlations.

Unfortunately, the actual constant $\varepsilon_0$ that arises in the
proof of Theorem~\ref{thmm:main} is almost certainly far from optimal.
The Dobrushin machinery \cite{Geo11}, Chapter~8, that forms the basis of
our proof already does not yield sharp estimates of the phase transition
point even in the simplest classical models of statistical mechanics. It
is also far from clear whether the block particle filter should
necessarily possess the same phase transition point as the underlying
model: it may be that the algorithm only works in a strict subset of the
regime in which the underlying model possesses the decay of correlations
property. The mathematical tools used in this paper are not sufficiently
powerful to address questions of this type. The practical relevance of
Theorem~\ref{thmm:main} is therefore of a qualitative nature---we show that
the block particle filter can beat the curse of dimensionality above a
certain phase transition point---but should not be relied upon to provide
quantitative guidance in specific situations. The development of sharper
quantitative results will require new probabilistic tools for the
investigation of filtering problems in high dimension.

One drawback of the assumptions of Theorem~\ref{thmm:main} is that mixing
in space and time are treated on the same footing: as $\varepsilon\to1$,
both the spatial and temporal interactions disappear. To ensure that
ergodicity and decay of correlations hold, it should suffice to assume
only that the spatial interactions are weak. Such an improvement can be
obtained using more refined mathematical tools that make it possible to
separate the temporal and spatial ergodicity assumptions \cite{RvH13b}.

\subsubsection{Local algorithms and spatial homogeneity}
\label{sec:localalg}

The major drawback of the block particle filtering algorithm is the
spatial inhomogeneity of the bias.
The consequences of this inhomogeneity are manifested quantitatively in
Corollary~\ref{cor:squarelattice}. Near the block boundaries,
Theorem~\ref{thmm:main} gives a bound of order unity. By excluding a small
fraction of spatial locations, however, we eliminate the block
boundaries to retain an error of order
$e^{-c\log^{1/q}N}$ at ``most'' spatial locations:
\[
\card \bigl\{v\in V\dvtx{\bigl|\!\bigl|\!\bigl| \pi_n^x-\hat
\pi_n^x \bigr|\!\bigr|\!\bigr|}_v \lesssim e^{-c\delta\log^{1/q}N}
\bigr\} \ge(1-\delta)\card V.
\]
If, on the other hand, we compute the spatial average of the error, we
obtain an exceedingly slow convergence rate that is much worse
than the ``typical'' rate:
\[
\frac{1}{\card V}\sum_{v\in V} {\bigl|\!\bigl|\!\bigl|
\pi_n^x-\hat\pi_n^x \bigr|\!\bigr|\!\bigr|}_v \lesssim\frac{1}{\log^{1/q}N}.
\]
Note that the block boundaries constitute a fraction
$\sim1/b$ of spatial locations, where $b$ is the block size; therefore,
as $b\sim\log^{1/q}N$ in Corollary~\ref{cor:squarelattice}, we see that
the error at the block boundaries dominates our bound on the average error.

The behavior of the errors described above seems to be an inherent
limitation of the block particle filtering algorithm. It is therefore of
significant interest to explore the possibility that one could develop
alternative local particle filtering algorithms that are spatially
homogeneous. Conceptually, as explained in Section~\ref
{sec:contributions}, such an algorithm should update the filtering
distribution at each site $v$ using sites in a centered neighborhood
$N_b(v):=\{v'\in V\dvtx d(v,v')\le b\}$; the decay of correlations
should then
yield a bias that decays exponentially in $b$. In this case, we would
expect to obtain a spatially uniform error bound of the form
\[
\sup_{v\in V} {\bigl|\!\bigl|\!\bigl| \pi_n^x-\hat
\pi_n^x \bigr|\!\bigr|\!\bigr|}_v \lesssim e^{-c\log^{1/q}N}
\]
for the optimized neighborhood size $b\sim\log^{1/q}N$.
Whether it is in fact possible to design a local particle filtering
algorithm that attains such a uniform error bound is perhaps the most
immediate open question that arises from our results.

It is, of course, not at all obvious how one might go about developing a
spatially homogeneous algorithm. We will presently discuss one possible
idea that could be of interest in this setting. It should be emphasized
the following discussion is intended to be heuristic, as we do not
know how to analyze algorithms of the type that we will discuss. However,
our aim is to illustrate that the general idea of local particle filters
could be much broader than is suggested by the block particle filtering
algorithm---and that the mathematical analysis developed in this paper
could in itself provide inspiration for further methodological developments.

At the heart of our results lies the decay of correlations. In our proofs,
we will use an intuitive notion of decay of correlations of essentially
the following form: a probability measure $\rho$ on $\bbX$ possesses the
decay of correlations property if the effect on the conditional
distribution $\rho(X^v\in\cdot|X^{V\setminus\{v\}}=x^{V\setminus\{
v\}})$
of a perturbation to $x^{v'}$ decays exponentially in the distance $d(v,v')$
(cf. Sections~\ref{sec:dobrushin} and \ref{sec:lfstab}). The blocking
operation evidently replaces these conditional distributions by
\[
(\mathsf{B}\rho) \bigl(X^v\in A|X^{V\setminus\{v\}}=x^{V\setminus\{
v\}}
\bigr) = \rho \bigl(X^v\in A|X^{K\setminus\{v\}}=x^{K\setminus\{v\}
} \bigr)
\]
for every $K\in\mathcal{K}$ and $v\in K$. Therefore, if $\rho$ possesses
the decay of correlations property, then the bias at site $v\in K$
incurred by the blocking operation decays exponentially in the distance
between $v$ and the boundary of $K$. From this perspective, an approach
to spatially homogeneous algorithms readily suggests itself:
we should aim to replace $\mathsf{B}$ with another operator $\mathsf{M}$
that satisfies
\[
(\mathsf{M}\rho) \bigl(X^v\in A|X^{V\setminus\{v\}}=x^{V\setminus\{
v\}}
\bigr) = \rho \bigl(X^v\in A|X^{N_b(v)\setminus\{v\}
}=x^{N_b(v)\setminus\{v\}} \bigr)
\]
for every $v\in V$. The bias incurred by this operation decays
exponentially in $b$ uniformly for all $v$ (it is therefore
spatially homogeneous). On the other hand, as
\begin{eqnarray*}
&&(\mathsf{C}_n\mathsf{M}\mathsf{P}\rho) \bigl(X^v\in
A|X^{V\setminus\{v\}}= x^{V\setminus\{v\}} \bigr)\\
&&\qquad =
 \frac{\int\mathbf{1}_A(x^v) g^v(x^v,Y_n^v)
\prod_{w\in N_b(v)}p^{w}(z,x^w) \rho(dz) \psi^v(dx^v)}{
\int g^v(x^v,Y_n^v)
\prod_{w\in N_b(v)}p^{w}(z,x^w) \rho(dz) \psi^v(dx^v)},
\end{eqnarray*}
the sampling error incurred if we replace $\rho$ by
$\mathsf{S}^N\rho$ in this expression should only be exponential in
$\card N_b(v)$ (which is $\sim b^q$ for the square lattice) rather than in
the model dimension $\card V$. This suggests that the local particle
filter defined by the recursion $\mathsf{\hat F}_n=
\mathsf{S}^N\mathsf{C}_n\mathsf{M}\mathsf{P}$ should yield a
spatially homogeneous algorithm in accordance with our
intuition. To implement this algorithm, one needs to sample from the
measure $\mathsf{C}_n\mathsf{M}\mathsf{P}\rho$, which we have defined
only implicitly in terms of its conditional distributions. This is
however precisely the task to which MCMC methods such as the Gibbs
sampler are well suited. One would therefore ostensibly obtain a spatially
homogeneous local particle filtering algorithm that is recursive in time
and that uses MCMC to sample the spatial degrees of freedom
(regularization using $\mathsf{M}$ is still key to avoiding the curse of
dimensionality; cf. Remark~\ref{rem:mcmc}).

Conceptually, the idea introduced here is quite natural. The general idea
of local particle filters is that one should introduce a spatial
regularization step into the filtering recursion that enables local
sampling. In the block particle filter, this regularization is provided
by the blocking operation $\mathsf{B}$ that projects a probability measure
on the class of measures that are independent across blocks. In the above
algorithm, we aim to regularize instead by the operation $\mathsf{M}$ that
projects a probability measure on the class of Markov random fields of
order $b$. The fatal flaw in our reasoning is that the operator
$\mathsf{M}$ that we have defined implicitly above does not exist:
the truncated conditional distributions $\rho(X^v\in\cdot
|X^{N_b(v)\setminus\{v\}}=x^{N_b(v)\setminus\{v\}})$ are typically not
consistent, so there exists no single probability measure that satisfies
our definition of $\mathsf{M}\rho$. Nonetheless, the basic idea
introduced here could be fruitful if one can develop a practical approach
to approximating random fields by Markov random fields [e.g., one
could attempt to substitute the above expression for
$(\mathsf{C}_n\mathsf{M}\mathsf{P}\rho)(X^v\in\cdot|X^{V\setminus
\{
v\}})$
in a Gibbs sampler regardless of its inconsistency].
The development of such ideas evidently presents some interesting
mathematical as well as methodological challenges that should
be investigated further.

Let us finally observe that, by their nature, local particle filtering
algorithms are well suited to distributed computation: as the particles
are updated locally in the spatial graph, this opens the possibility of
implementing each local neighborhood on a separate processor. While this
was not the original intention of the algorithms we propose, such
properties could prove to be advantageous in their own right for the
practical implementation of filtering algorithms in very large-scale
systems.

\subsubsection{High-dimensional models in data assimilation}
\label{sec:finitediff}

The basic model that we have introduced in Section~\ref{sec:hdfiltmodels}
is prototypical of many data assimilation problems and provides a
particularly convenient mathematical setting for the investigation of
filtering problems in high dimension. While such models could be directly
relevant to many high-dimensional applications, there remains a
substantial gap between relatively simple models of this form and
realistic models used in the most complex applications, particularly in
the geophysical, atmospheric and ocean sciences, that frequently consist
of coupled systems of partial differential equations. The investigation
of such complex problems, and the associated numerical, physical and
practical issues, is far beyond the scope of this paper. We
therefore restrict our discussion of such problems to a few brief
comments.

In principle, discrete models as defined in Section~\ref{sec:hdfiltmodels}
arise naturally as finite-difference approximations of stochastic partial
differential equations with space--time white noise forcing. As the
resulting state spaces $\bbX^v$ are not compact, such systems cannot
satisfy strong mixing assumptions (cf. Section~\ref{sec:mixing}), but
this is likely a mathematical rather than a practical problem. More
importantly, it is not clear whether the discretized models will be in the
regime of decay of correlations (i.e., above the phase transition
point) even if the original continuum model possesses such properties.
It is possible that this requirement would place constraints on the
spatial and temporal discretization steps, in the spirit of the von
Neumann stability criterion in numerical analysis. The physics of such
problems could also impose constraints on the design of local particle
filters; for example, it is suggested in \cite{vL09}, page~4107, that
discontinuities (such as might be introduced at the block boundaries in
the block particle filtering algorithm) could generate spurious gravity
waves in ocean models. Such numerical and practical issues are distinct
from the fundamental problems in high dimension that we aim to address in
this paper, but can ultimately play an equally important role in complex
applications.

Let us also note that models considered in the data assimilation
literature are often deterministic partial differential equations without
stochastic forcing; the only randomness in such models comes from the
initial condition (cf. \cite{LS12,AJSV08}). In deterministic chaotic
dynamical systems, it is impossible to obtain time-uniform approximations
using classical particle filters as there is no dissipation mechanism for
approximation errors (the filter cannot be stable in this case; cf.
Section~\ref{sec:errdecomp}). This issue is not directly related to
dimensionality issues in particle filters: such problems arise in every
deterministic filtering problem. It is natural to regularize
deterministic systems by adding dynamical noise to the model
(there is an extensive literature on random perturbations of chaotic
dynamics; see, e.g., \cite{Bla97}); a~similar observation has been
made by practitioners in the context of ad-hoc filtering algorithms;
cf.
\cite{LS12}, Section~5. To our knowledge, a rigorous analysis of such
ideas in the setting of particle filters has yet to be performed.


\section{Outline of the proof}
\label{sec:outline}

\subsection{Error decomposition}
\label{sec:errdecomp}

The goal of Theorem~\ref{thmm:main} is to bound the error between the
filter $\pi_n^\mu$ and the block particle filter $\hat\pi_n^\mu$.
Recall that both the filter (Section~\ref{sec:filtering}) and block
particle filter (Section~\ref{sec:blockfilt}) are defined recursively:
\[
\pi_n^\mu=\mathsf{F}_n\cdots
\mathsf{F}_1\mu, \qquad\hat\pi_n^\mu=\mathsf{\hat
F_n}\cdots\mathsf{\hat F}_1\mu,
\]
where $\mathsf{F}_n:=\mathsf{C}_n\mathsf{P}$ and
$\mathsf{\hat F}_n:=\mathsf{C}_n\mathsf{B}\mathsf{S}^N\mathsf{P}$.
We introduce also the \emph{block filter}
\[
\tilde\pi_n^\mu= \mathsf{\tilde F_n}\cdots
\mathsf{\tilde F}_1\mu
\]
with $\mathsf{\tilde F}_n:=\mathsf{C}_n\mathsf{B}\mathsf{P}$. By the
triangle inequality, we have
\[
{\bigl|\!\bigl|\!\bigl| \pi_n^\mu-\hat\pi_n^\mu
\bigr|\!\bigr|\!\bigr|}_J \le{\bigl|\!\bigl|\!\bigl| \pi_n^\mu-\tilde
\pi_n^\mu \bigr|\!\bigr|\!\bigr|}_J + {\bigl|\!\bigl|\!\bigl| \tilde
\pi_n^\mu-\hat\pi_n^\mu \bigr|\!\bigr|\!\bigr|}_J.
\]
The first term on the right-hand side quantifies the bias introduced by
the projection on independent blocks, while the second term quantifies the
error due to the variance of the random sampling in the algorithm. Each
term will be bounded separately to obtain the two terms in the error bound
of Theorem~\ref{thmm:main}.

The challenges encountered in bounding the bias term (cf. Section~\ref
{sec:bias}) and the variance term (cf. Section~\ref{sec:variance})
are quite different in nature. Nonetheless, both bounds are based on
a basic scheme of proof that was invented in order to prove time-uniform
bounds for the bootstrap particle filter \cite{DG01,CMR05}. We therefore
begin by reviewing this general idea, which is based on a simple error
decomposition.

Suppose for sake of illustration that we aim to bound directly the error
between $\pi_n^\mu$ and $\hat\pi_n^\mu$. The basic idea is to write
$\pi_n^\mu-\hat\pi_n^\mu$ as a telescoping sum:
\[
\pi_n^\mu- \hat\pi_n^\mu= \sum
_{s=1}^n\{ \mathsf{F}_n\cdots
\mathsf{F}_{s+1}\mathsf{F}_s \mathsf{\hat
F}_{s-1}\cdots\mathsf{\hat F}_1\mu- \mathsf{F}_n
\cdots\mathsf{F}_{s+1}\mathsf{\hat F}_s \mathsf{\hat
F}_{s-1}\cdots\mathsf{\hat F}_1\mu\}.
\]
By the triangle inequality,
\[
{\bigl|\!\bigl|\!\bigl| \pi_n^\mu- \hat\pi_n^\mu
\bigr|\!\bigr|\!\bigr|} \le\sum_{s=1}^n{\bigl|\!\bigl|\!\bigl|
\mathsf{F}_n\cdots\mathsf{F}_{s+1}\mathsf
{F}_s \hat\pi_{s-1}^\mu-
\mathsf{F}_n\cdots \mathsf {F}_{s+1}\mathsf{\hat
F}_s \hat \pi_{s-1}^\mu \bigr|\!\bigr|\!\bigr|}.
\]
The term $s$ in this sum could be interpreted as the contribution to the
total error at time $n$ due to the filter approximation made in
time step $s$.

The key insight is now that one can employ the \emph{filter stability}
property to control this sum uniformly in time. In its simplest form,
this property can be proved in the following form: if
$\varepsilon\le p(x,z)\le\varepsilon^{-1}$
for all $x,z\in\bbX$, then \cite{DG01,CMR05}
\[
{\bigl|\!\bigl|\!\bigl| \mathsf{F}_n\cdots\mathsf{F}_{s+1}\rho- \mathsf
{F}_n\cdots\mathsf{F}_{s+1}\rho' \bigr|\!\bigr|\!\bigr|} \le
\varepsilon^{-2} \bigl(1-\varepsilon^2
\bigr)^{n-s} {\bigl|\!\bigl|\!\bigl| \rho-\rho' \bigr|\!\bigr|\!\bigr|}.
\]
Thus, the filter forgets its initial condition at an exponential
rate. However, this also means that past approximation errors are
forgotten at an exponential rate: if we substitute the stability property
in the above error decomposition, we obtain
\[
{\bigl|\!\bigl|\!\bigl| \pi_n^\mu- \hat\pi_n^\mu
\bigr|\!\bigr|\!\bigr|} \le\sum_{s=1}^n
\varepsilon^{-2} \bigl(1-\varepsilon^2 \bigr)^{n-s}
{\bigl|\!\bigl|\!\bigl| \mathsf{F}_s\hat\pi _{s-1}^\mu-
\mathsf{\hat F}_s\hat\pi_{s-1}^\mu \bigr|\!\bigr|\!\bigr|}
\le \varepsilon^{-4} \sup_{n,\rho}{\bigl|\!\bigl|\!\bigl|
\mathsf{F}_n\rho-\mathsf {\hat F}_n\rho \bigr|\!\bigr|\!\bigr|}.
\]
Thus, if we can control the error ${|\!|\!| \mathsf{F}_n\rho-\mathsf
{\hat F}_n\rho |\!|\!|}$ in a single time step, we obtain a
time-uniform bound of the
same order. In the case of the bootstrap particle filter, if $\kappa
\le
g(x,y)\le\kappa^{-1}$, we proved
that ${|\!|\!| \mathsf{F}_n\rho-\mathsf{\hat F}_n\rho |\!|\!|}\le
2\kappa
^{-2}/\sqrt{N}$
in Section~\ref{sec:filtering},
and we obtain a time-uniform version of the crude error bound given
there.

The basic error decomposition discussed above allows us to separate the
problem of obtaining time-uniform bounds into two parts: the one-step
approximation error and the stability property. It is important to note,
however, that both parts become problematic in high dimension. We have
already seen (Section~\ref{sec:curse}) that the one-step approximation
error of the bootstrap particle filter is exponential in the model
dimension; we will surmount this problem by working with the block
particle filtering algorithm and performing a local analysis of
the one-step error using the decay of correlations property (which must
itself be established). On the other hand,
the filter stability bound used above also becomes exponentially
worse in high dimension: a local bound of the form $\varepsilon\le
p^v(x,z^v)\le\varepsilon^{-1}$ only yields $\varepsilon^{\card V}\le
p(x,z)\le\varepsilon^{-\card V}$, which is exponential in the model
dimension $\card V$. To surmount this problem, we must develop a much
more precise understanding of the filter stability property in high
dimension, which proves to be closely related to the decay of
correlations property. The development of these ingredients
constitutes the bulk of the proof of Theorem~\ref{thmm:main}.

\subsection{Dobrushin comparison method}
\label{sec:dobrushin}

How can one control the approximation error of high-dimensional
distributions? The basic idea that we aim to exploit, both
algorithmically and mathematically, is that the decay of correlations
property leads to a form of localization: the effect on the distribution
in some spatial set $J$ of a perturbation made in another set $J'$ decays
rapidly in the distance $d(J,J')$. Therefore, as long as we measure the
error locally (in ${|\!|\!| \cdot |\!|\!|}_J$ rather than ${|\!|\!|
\cdot |\!|\!|}$), one
would hope to control the spatial accumulation of approximation errors
much as we controlled the accumulation of approximation errors in time
using the filter stability property. We will presently introduce a
powerful (albeit blunt) tool---the Dobrushin comparison theorem---that
makes this idea precise in a very general setting. This fundamental
result, which plays an important role in statistical mechanics
\cite{Geo11}, Chapter~8, is the main workhorse that will be used
repeatedly in our proofs.

Let $I$ be a finite set, and let $\mathbb{S}=\prod_{i\in I}\mathbb{S}^i$
where $\mathbb{S}^i$ is a Polish space for each $i\in I$. Define the
coordinate projections $X^i\dvtx x\mapsto x^i$ for $x\in\mathbb{S}$
an $i\in I$.
For any probability $\rho$ on $\mathbb{S}$, we fix a version
$\rho^i_\cdot$ of the regular conditional probability
\[
\rho^i_x(A) = \rho \bigl(X^i\in
A|X^{I\setminus\{i\}}=x^{I\setminus\{i\}} \bigr).
\]
We also define for $J\subseteq I$ the local total variation distance
\[
\bigl\|\rho-\rho'\bigr\|_J := \sup_{f\in\mathcal{S}^J\dvtx|f|\le1} \bigl|
\rho(f)-\rho'(f)\bigr|,
\]
where $\mathcal{S}^J$ is the class of measurable functions
$f\dvtx\mathbb{S}\to\mathbb{R}$ such that $f(x)=f(z)$ whenever
$x^J=z^J$. For $J=I$, we write $\|\rho-\rho'\|$ for simplicity.

We can now state the Dobrushin comparison theorem \cite{Geo11},
Theorem~8.20.\footnote{Note that our definition of $\|\cdot\|_J$
differs by a factor $2$ from that in \cite{Geo11}.}

%
\begin{thmm}[(Dobrushin)]
\label{thmm:dobrushin}
Let $\rho,\tilde\rho$ be probability measures on
$\mathbb{S}$. Define
\[
C_{ij} = \frac{1}{2} \sup_{x,z\in\mathbb{S}\dvtx x^{I\setminus\{j\}}=
z^{I\setminus\{j\}}} \bigl\|
\rho^i_x-\rho^i_z\bigr\|,\qquad
b_j = \sup_{x\in\mathbb{S}}\bigl\|\rho^j_x-
\tilde\rho^j_x\bigr\|.
\]
Suppose that the Dobrushin condition holds:
\[
\max_{i\in I}\sum_{j\in I}C_{ij}<1.
\]
Then the matrix sum
$D := \sum_{n\ge0}C^n$ is convergent, and we have
for every $J\subseteq I$
\[
\|\rho-\tilde\rho\|_J \le\sum_{i\in J}
\sum_{j\in I} D_{ij} b_j.
\]
\end{thmm}

This result could be informally interpreted as follows. $C_{ij}$ measures
the degree to which a perturbation of site $j$ directly affects site $i$
under the distribution $\rho$. However, perturbing site $j$ might also
indirectly affect $i$: it could affect another site $k$ which in turn
affects $i$, etc. The aggregate effect of a perturbation of site $j$ on
site $i$ is captured by the quantity $D_{ij}$. In this setting, a useful
manifestation of the decay of correlations property is that $D_{ij}$
decays exponentially in the distance $d(i,j)$. If this is in fact the
case, then Theorem~\ref{thmm:dobrushin} yields, for example,
$\|\rho-\tilde\rho\|_i \lesssim\sum_j e^{-d(i,j)}b_j$, where $b_j$
measures the local error at site $j$ between $\rho$ and~$\tilde\rho$ (in
terms of the conditional distributions $\rho^j_\cdot$ and
$\tilde\rho^j_\cdot$). The decay of correlations property therefore
controls the accumulation of local errors much as one might expect.

Let us now explain how Theorem~\ref{thmm:dobrushin} will be applied in the
filtering setting. For sake of illustration, consider the problem of
obtaining a local filter stability bound: that is, we would like to bound
$\|\pi_n^x-\pi_n^{\tilde x}\|_J$ for $x,\tilde x\in\bbX$ and
$J\subseteq
V$. It would seem natural to apply Theorem~\ref{thmm:dobrushin} directly
with $I=V$, $\mathbb{S}=\bbX$, and $\rho=\pi_n^x$,
$\tilde\rho=\pi_n^{\tilde x}$. This is not useful, however, as we do not
know how to control the corresponding local quantities such as $\rho^v_z=
\mathbf{P}^x[X_n^v\in\cdot|Y_1,\ldots,Y_n,X_n^{V\setminus\{v\}}=
z^{V\setminus\{v\}}]$.

Instead, define $I=\{0,\ldots,n\}\times V$ and $\mathbb{S}=\bbX^{n+1}$,
and let
\begin{eqnarray*}
\rho&=&\mathbf{P}^x \bigl[(X_0,\ldots,X_n)
\in\cdot|Y_1,\ldots,Y_n \bigr],
\\
\tilde\rho&=&\mathbf{P}^{\tilde x} \bigl[(X_0,
\ldots,X_n)\in\cdot|Y_1,\ldots,Y_n \bigr].
\end{eqnarray*}
As
\[
\bigl\|\pi_n^x-\pi_n^{\tilde x}
\bigr\|_J=\|\rho-\tilde\rho\|_{\{n\}\times J},
\]
we can now apply Theorem~\ref{thmm:dobrushin} to the \emph{smoothing}
distributions $\rho,\tilde
\rho$.
Unlike the filters $\pi_n^x,\pi_n^{\tilde x}$, however, $\rho$ and
$\rho'$
are Markov random fields on $I$ (cf. Figure~\ref{fig:lfilt}), so that the
conditional distributions $\rho^{k,v}_z$ and $\tilde\rho^{k,v}_z$
can be
easily computed and controlled in terms of the local densities
$p^v(x,z^v)$ and $g^v(x^v,y^v)$. For example, as
\[
\rho(A) \propto\int\mathbf{1}_A(x,x_1,
\ldots,x_n) \prod_{k=1}^n\prod
_{v\in V} p^v \bigl(x_{k-1},x_k^v
\bigr) g^v \bigl(x_k^v,Y_k^v
\bigr) \psi^v \bigl(dx_k^v \bigr),
\]
and as $p^v(x_{k-1},x_k^v)$ depends only on $x_{k-1}^w$ for $d(w,v)\le r$,
we obtain
\[
\rho^{k,v}_z(B) \propto\int\mathbf{1}_B
\bigl(z_k^v \bigr) p^v \bigl(z_{k-1},z_k^v
\bigr) g^v \bigl(z_k^v,Y_k^v
\bigr)\prod_{w\in N(v)} p^w
\bigl(z_k,z_{k+1}^w \bigr) \psi^v
\bigl(dz_k^v \bigr)
\]
for $0<k<n$ and $v\in V$ (the proportionality is up to a
normalization factor). We will repeatedly exploit expressions of this
type to obtain explicit bounds on the quantities $C_{ij}$ and $b_j$ that
appear in Theorem~\ref{thmm:dobrushin}. It should be emphasized
that $\rho^{k,v}_z$ is a genuinely local quantity: the product inside
the integral contains at most $\card N(v)\le\Delta$ factors. We will
consequently be able to use Theorem~\ref{thmm:dobrushin} to obtain bounds
that do not depend on the model dimension $\card V$.

\subsection{Bounding the bias: Decay of correlations}
\label{sec:bias}

To bound the bias $\|\pi_n^x-\tilde\pi_n^x\|_J$, we follow
the basic error decomposition scheme described above, that is,
\[
\bigl\|\pi_n^x - \tilde\pi_n^x
\bigr\|_J \le\sum_{s=1}^n \bigl\|
\mathsf{F}_n\cdots\mathsf{F}_{s+1}\mathsf{F}_s
\tilde\pi_{s-1}^x- \mathsf{F}_n\cdots
\mathsf{F}_{s+1}\mathsf{\tilde F}_s \tilde
\pi_{s-1}^x\bigr\|_J.
\]
To implement our program, we must now obtain suitable local bounds on
the stability of the filter and on the one-step approximation error.
Both these problems will be approached by application of the
Dobrushin comparison theorem.

In its most basic form, one can prove a filter stability property of the
following type: provided $\varepsilon>\varepsilon_0$, there exists
$\beta>0$ (depending only on $\Delta$ and $r$) such that
\[
\|\mathsf{F}_n\cdots\mathsf{F}_{s+1}\mu-
\mathsf{F}_n\cdots\mathsf{F}_{s+1}\nu\|_J \le4
\card J e^{-\beta(n-s)}
\]
for any probability measures $\mu,\nu$ on $\bbX$ and
$J\subseteq V$, $n\ge0$ (cf. Corollary~\ref{cor:fstab}).
This bound is evidently dimension-free, unlike the crude
filter stability bound described in Section~\ref{sec:errdecomp}.
Nonetheless, this filter stability bound would yield a trivial result
when substituted in the error decomposition, as it does not provide any
control in terms of the distance between $\mu$ and $\nu$ (and, therefore,
in terms of the one-step error). Instead, we will prove in
Section~\ref{sec:lfstab} the local stability bound
\[
\|\mathsf{F}_n\cdots\mathsf{F}_{s+1}\mu-
\mathsf{F}_n\cdots\mathsf{F}_{s+1}\nu\|_J \le
2e^{-\beta(n-s)} \sum_{v\in J}\max
_{v'\in V} e^{-\beta d(v,v')} D_{v'}(\mu,\nu),
\]
where $D_{v'}(\mu,\nu)$ is a suitable measure of the local
error between $\mu$ and $\nu$ at site $v'$ that arises naturally
from the Dobrushin comparison theorem (see Proposition~\ref{prop:lfstab}
for precise expressions). This filter stability bound is genuinely
local: the stability on the spatial set $J\subseteq V$ depends
predominantly on the local distance of the initial conditions near $J$
(i.e., the spatial accumulation of errors is mitigated). This
localization comes at a price, however; the local filter stability
bound holds only if the initial condition $\mu$ satisfies {a priori}
a decay of correlations property.

Once the local filter stability bound is substituted in the error
decomposition, it remains to prove a bound on the one-step error
$D_v(\mathsf{F}_s\tilde\pi_{s-1}^x,\mathsf{\tilde F}_s\tilde\pi_{s-1}^x)$
with respect to the local distance prescribed by the filter stability
bound. This will be done in Section~\ref{sec:1step}: we will show
that for a constant $C$ that depends only on $\Delta,r,\varepsilon$,
\[
D_v(\mathsf{F}_s\mu,\mathsf{\tilde F}_s
\mu) \le C e^{-\beta d(v,\partial K)}
\]
for every $K\in\mathcal{K}$ and $v\in K$, provided again that
$\mu$ satisfies {a priori} a decay of correlations property.
This is precisely what we expect: as $\mathsf{B}$ only
introduces errors at the block boundaries, the decay of correlations
should ensure that the error at site $v$ decays exponentially in the
distance to the nearest block boundary. The Dobrushin comparison theorem
allows to make this intuition precise.

The decay of correlations property evidently plays a dual role in our
setting: it controls the approximation error of the block filter, which is
the basic principle behind the block particle filtering algorithm; at the
same time, it mitigates the spatial accumulation of approximation errors,
which is essential for proving dimension-free bounds. In order to apply
the above bounds, the key step that remains is to prove that the
appropriate decay of correlations property does in fact hold, uniformly
in time, for the block filter $\tilde\pi_n^x$. The latter will be
shown in Section~\ref{sec:decay} by iterating a one-step decay of
correlations bound that is obtained once again using the Dobrushin
comparison theorem. We conclude by putting together all these ingredients
in Section~\ref{sec:thmbias} to obtain a bound on the bias of the form
\[
\bigl\|\pi_n^x-\tilde\pi_n^x
\bigr\|_J \le C\card J e^{-\beta d(J,\partial K)}
\]
for $J\subseteq K$ (Theorem~\ref{thmm:bias}). This proves the first
half of Theorem~\ref{thmm:main} (note that, as the bias does not
depend on the random sampling in the block particle filtering algorithm,
we can trivially replace $\|\pi_n^x-\tilde\pi_n^x\|_J$ by
${|\!|\!| \pi_n^x-\tilde\pi_n^x |\!|\!|}_J$ in this bound).

\subsection{Bounding the variance: The computation tree}
\label{sec:variance}

To bound the variance term ${|\!|\!| \tilde\pi_n^x-\hat\pi_n^x |\!
|\!|}_J$, we
once again start from the basic error decomposition
\[
{\bigl|\!\bigl|\!\bigl| \tilde\pi_n^x - \hat\pi_n^x
\bigr|\!\bigr|\!\bigr|}_J \le\sum_{s=1}^n
{\bigl|\!\bigl|\!\bigl| \mathsf{\tilde F}_n\cdots\mathsf{\tilde F}_{s+1}
\mathsf{\tilde F}_s \hat\pi_{s-1}^x- \mathsf{
\tilde F}_n\cdots\mathsf{\tilde F}_{s+1}\mathsf{\hat
F}_s \hat\pi _{s-1}^x \bigr|\!\bigr|\!\bigr|}_J.
\]
The difficulties encountered in controlling this expression are quite
different in nature, however, than what was needed to control the bias
term.

Dimension-free bounds on the bias exploit decay of correlations: the core
difficulty is to obtain local control of the error inside the blocks.
The variance term, on the other hand, will already grow exponentially in
the size of the blocks due to the exponential dependence of the sampling
error on the dimension of the observations. There is therefore no need
bound the error on a finer scale than a single block. This makes the
analysis of the variance much less delicate than controlling the bias, and
it is indeed not difficult to obtain a variance bound of the right
order on a finite time horizon (but growing exponentially in time $n$).

The chief difficulty in controlling the variance is to obtain a
time-uniform bound. Note that, in the error decomposition for the
variance term, it is not stability of the filter $\pi_n^\mu$ that enters
the picture but rather stability of the block filter~$\tilde\pi_n^\mu$.
Unlike the filter, however, which has by construction an interpretation as
the marginal of a smoothing distribution, the block filter
is defined by a recursive algorithm and not as a conditional expectation.
It is therefore not entirely obvious how one could adapt the approach
outlined in Section~\ref{sec:dobrushin} to this setting.
%
%
\begin{figure}

\includegraphics{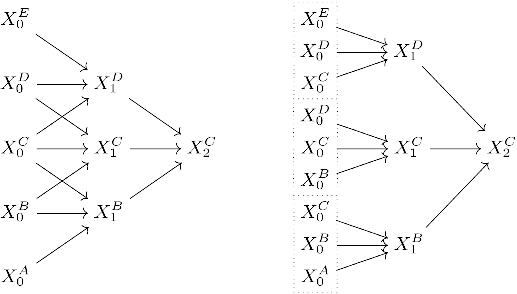}

\caption{For a linear spatial graph $G$ partitioned into blocks $A$--$E$
(with $r=1$), the dependencies between the
blocks at subsequent times are illustrated here. The left dependency
graph represents $\mathsf{B}^C\mathsf{P}^2\mu$, the right
graph represents $\mathsf{B}^C\mathsf{P}\mathsf{B}\mathsf{P}\mu$.
The blocking operation unravels the original graph into a tree by
introducing independent duplicates (dotted boxes) of blocks in the
previous time step.}
\label{fig:comptree}
\end{figure}


The key idea that will be used to establish stability is that the block
filter can nonetheless
be viewed as the marginal of a suitably defined Markov random field, just
like the filter can be viewed as the marginal of a smoothing distribution.
This random field, however, lives on a much larger index set than the
original model. The basic idea behind the construction is illustrated in
Figure~\ref{fig:comptree} (disregarding the observations for
simplicity of
exposition). When we apply the transition operator~$\mathsf{P}$, each
block interacts with its $\Delta_{\mathcal{K}}$ neighbors in the previous
time step. However, if we subsequently apply the blocking
operator~$\mathsf{B}$, then each block is replaced by an independent copy.
This
could be modeled equivalently by introducing independent duplicates of
the blocks in the previous time step, and having each block interact with
its own set of duplicates. This unravels the original dependency graph
into a tree. By iterating this process, we can express the block filter
as the marginal of a Markov random field defined on a tree that contains
many independent duplicates of each block. We call this construction the
\emph{computation tree} in analogy with a similar notion that arises in
the analysis of belief propagation algorithms \cite{TJ02}.

With this construction in place, we can now obtain a stability bound for
the block filter by applying the Dobrushin comparison theorem to the
computation tree. This will be done in Section~\ref{sec:blockfstab}
to obtain a bound of the following form: provided
$\varepsilon>\varepsilon_0$, there exist $\beta,\beta'>0$
(depending only on $\Delta,\Delta_{\mathcal{K}},r$) such that
\[
\max_{K\in\mathcal{K}}\| \mathsf{\tilde F}_n\cdots\mathsf{
\tilde F}_{s+1}\mu- \mathsf{\tilde F}_n\cdots\mathsf{\tilde
F}_{s+1}\nu\|_K \le e^{\beta'|\mathcal{K}|_\infty}e^{-\beta(n-s)} \max
_{K\in\mathcal{K}}\bigl\|\mu^K-\nu^K\bigr\|
\]
for any pair of initial conditions of product form
$\mu=\bigotimes_{K\in\mathcal{K}}\mu^K$,
$\nu=\bigotimes_{K\in\mathcal{K}}\nu^K$ (cf. Corollary~\ref
{cor:blockfstab}). Combining this bound with the
error decomposition, we obtain in Section~\ref{sec:thmvariance} a
time-uniform bound on the variance term
of the form
\[
\max_{K\in\mathcal{K}}{\bigl|\!\bigl|\!\bigl| \tilde\pi_n^x-\hat
\pi_n^x \bigr|\!\bigr|\! \bigr|}_K \le C
\frac{e^{\beta'|\mathcal{K}|_\infty}}{\sqrt{N}},
\]
where we bound the one-step error in the same spirit as the computation
for the bootstrap particle filter in Section~\ref{sec:filtering} (however,
a more involved argument is needed here to surmount the fact that the
block filter stability bound is given in a total variation norm
rather than the weaker norm ${|\!|\!| \cdot |\!|\!|}_K$).
Thus, Theorem~\ref{thmm:main} is proved.

\section{Proof of Theorem \texorpdfstring{\protect\ref{thmm:main}}{2.1}}
\label{sec:proofs}

Theorem~\ref{thmm:main} yields a bound on
${|\!|\!| \pi_n^\mu-\hat\pi_n^\mu |\!|\!|}_J$. As
\[
{\bigl|\!\bigl|\!\bigl| \pi_n^\mu-\hat\pi_n^\mu
\bigr|\!\bigr|\!\bigr|}_J \le{\bigl|\!\bigl|\!\bigl| \pi_n^\mu-\tilde
\pi_n^\mu \bigr|\!\bigr|\!\bigr|}_J + {\bigl|\!\bigl|\!\bigl| \tilde
\pi_n^\mu-\hat\pi_n^\mu \bigr|\!\bigr|\!\bigr|}_J,
\]
it suffices to bound each term in this inequality. As was
explained in Section~\ref{sec:errdecomp}, the first term quantifies the
bias of the block particle filter, while the second term quantifies the
variance of the random sampling. The bias term will be bounded in
Theorem~\ref{thmm:bias} below, while the variance will be bounded in
Theorem~\ref{thmm:variance}. The combination of these two results immediately
yields Theorem~\ref{thmm:main}.

\subsection{Preliminary lemmas}
\label{sec:prelim}

The Dobrushin comparison method introduced in Section~\ref{sec:dobrushin}
is the main workhorse of our proof. To use this method, we must be able
to bound the quantities $C_{ij}$, $b_j$ and $D_{ij}$ that appear in
Theorem~\ref{thmm:dobrushin}. The goal of this preliminary section is to
collect some elementary lemmas for this purpose.

We start with a rather trivial lemma that will be used to bound $C_{ij}$.

%
\begin{lem}
\label{lem:minorize}
Let probability measures $\nu,\nu',\gamma,\gamma'$ and $\varepsilon
>0$ be
such that
$\nu(A)\ge\varepsilon\gamma(A)$ and $\nu'(A)\ge\varepsilon\gamma
'(A)$ for
every measurable set $A$. Then
\[
\bigl\|\nu-\nu'\bigr\| \le2(1-\varepsilon) + \varepsilon\bigl\|\gamma-
\gamma'\bigr\|.
\]
In particular, if $\gamma=\gamma'$, then
$\|\nu-\nu'\|\le2(1-\varepsilon)$.
\end{lem}

\begin{pf}
As
$\mu=(1-\varepsilon)^{-1}(\nu-\varepsilon\gamma)$ and
$\mu'=(1-\varepsilon)^{-1}(\nu'-\varepsilon\gamma')$ are probability
measures and $\nu-\nu'=(1-\varepsilon)(\mu-\mu')+\varepsilon
(\gamma-\gamma')$, the result follows readily.
\end{pf}

Next, we state a simple lemma on the distance between weighted measures.
We have already used this result in Section~\ref{sec:filtering} to bound
${|\!|\!| \mathsf{C}_n\rho-\mathsf{C}_n\rho' |\!|\!|}$.

%
\begin{lem}
\label{lem:weight}
Let $\mu,\nu$ be (possibly random) probability measures and let
$\Lambda$
be a bounded and strictly positive measurable function. Define
\[
\mu_\Lambda(A) := \frac{\int\mathbf{1}_A(x)\Lambda(x)\mu(dx)}{
\int\Lambda(x)\mu(dx)},\qquad \nu_\Lambda(A) :=
\frac{\int\mathbf{1}_A(x)\Lambda(x)\nu(dx)}{
\int\Lambda(x)\nu(dx)}.
\]
Then
\[
\|\mu_\Lambda-\nu_\Lambda\| \le2 \frac{\sup_x\Lambda(x)}{\inf_x\Lambda(x)} \|\mu-\nu
\|.
\]
The same conclusion holds if the $\|\cdot\|$-norm is replaced
by the ${|\!|\!| \cdot |\!|\!|}$-norm.
\end{lem}

\begin{pf}
The result follows readily from the identity
\[
\mu_\Lambda(f)-\nu_\Lambda(f) = \frac{1}{\mu(\Lambda)} \biggl[ \bigl
\{\mu(f\Lambda)-\nu(f\Lambda) \bigr\} + \frac{\nu(f\Lambda)}{\nu
(\Lambda)} \bigl\{\nu(\Lambda)-
\mu( \Lambda) \bigr\} \biggr]
\]
using the definition of the norms $\|\cdot\|$ or ${|\!|\!| \cdot |\!
|\!|}$.
\end{pf}

Finally, we give a lemma that will be essential for bounding $D_{ij}$.
In essence, the lemma states that if $C_{ij}$ decays exponentially in the
distance between $i$ and $j$ at a sufficiently rapid rate, then $D_{ij}$
will also decay exponentially in the distance between $i$ and $j$. This is
essential in order to establish the decay of correlations property using
only bounds on $C_{ij}$, which can be obtained in explicit form. While the
lemma should be interpreted in the spirit of decay of correlations, it is
essentially a simple lemma about matrices and will be stated as such.

%
\begin{lem}
\label{lem:expdecay}
Let $I$ be a finite set and let $m$ be a pseudometric on $I$. Let
$C=(C_{ij})_{i,j\in I}$ be a matrix with nonnegative entries.
Suppose that
\[
\max_{i\in I}\sum_{j\in I}e^{m(i,j)}C_{ij}
\le c<1.
\]
Then the matrix $D=\sum_{n\ge0}C^n$ satisfies
\[
\max_{i\in I}\sum_{j\in I}e^{m(i,j)}D_{ij}
\le\frac{1}{1-c}.
\]
In particular, this implies that
\[
\sum_{j\in J}D_{ij}\le\frac{e^{-m(i,J)}}{1-c}
\]
for every $J\subseteq I$.
\end{lem}

\begin{pf}
Define for any matrix $A$ with nonnegative entries the norm
\[
\|A\|_m := \max_{i\in I}\sum
_{j\in I} e^{m(i,j)}A_{ij}.
\]
Using $m(i,j)\le m(i,k)+m(k,j)$, we compute
\begin{eqnarray*}
\|AB\|_m &=& \max_{i\in I}\sum
_{j\in I} e^{m(i,j)}\sum_{k\in I}
A_{ik}B_{kj}
\\
&\le&\max_{i\in I}\sum_{k\in I}
e^{m(i,k)}A_{ik} \sum_{j\in I}
e^{m(k,j)}B_{kj}
\\
&\le&\|A\|_m\|B\|_m,
\end{eqnarray*}
so $\|A\|_m$ is a matrix norm. Therefore,
\[
\|D\|_m \le\sum_{n\ge0}\|C
\|_m^n \le\sum_{n\ge0}c^n=
\frac{1}{1-c}.
\]
As
\[
e^{m(i,J)}\sum_{j\in J}D_{ij} \le
\sum_{j\in J}e^{m(i,j)}D_{ij} \le\|D
\|_m,
\]
the last statement of the lemma follows immediately.
\end{pf}

\subsection{Local stability of the filter}
\label{sec:lfstab}

The main goal of this section is to prove a local stability
bound for the nonlinear filter. We begin, however, by introducing
a number of objects that will appear several times in the sequel.\vadjust{\goodbreak}

For any probability measure $\mu$ on $\bbX$ and $x,z\in\bbX$, $v\in
V$, we
define
\begin{eqnarray*}
\mu_{x,z}^v(A)& :=  &\mathbf{P}^\mu
\bigl[X_0^v\in A| X_0^{V\setminus\{v\}}=x^{V\setminus\{v\}},
X_1=z \bigr]
\\
&=  &\frac{
\int\mathbf{1}_A(x^v)\prod_{w\in N(v)}p^w(x,z^w) \mu^v_x(dx^v)
}{\int\prod_{w\in N(v)}p^w(x,z^w) \mu^v_x(dx^v)}
\end{eqnarray*}
(recall the notation $\mu_x^v:=\mathbf{P}^\mu[X_0^v\in\cdot|
X_0^{V\setminus\{v\}}=x^{V\setminus\{v\}}]$ in Section~\ref
{sec:dobrushin}). Let
\[
C^\mu_{vv'} := \frac{1}{2} \sup
_{z\in\bbX} \sup_{x,\tilde x\in\bbX:x^{V\setminus\{v'\}}=
\tilde x^{V\setminus\{v'\}}} \bigl\|\mu^v_{x,z}-
\mu^v_{\tilde x,z}\bigr\|
\]
for $v,v'\in V$. The quantity
\[
\corr(\mu,\beta) := \max_{v\in V}\sum
_{v'\in V} e^{\beta d(v,v')}C^\mu_{vv'}
\]
could be viewed as a measure of the degree of correlation decay of the
measure $\mu$ at rate $\beta>0$. It will turn out that this (not
entirely obvious) measure of decay of correlations is precisely tuned to
the needs of the proof of Theorem~\ref{thmm:main}. This is due to the
fact that the measures $\mu^v_{x,z}$ arise naturally when applying
the Dobrushin comparison method to the smoothing distributions as
discussed in Section~\ref{sec:dobrushin}.

We recall once and for all that the interaction radius $r$ and
neighborhood size $\Delta$ that will appear repeatedly in the following
results are defined in Section~\ref{sec:blockfilt}.

%
\begin{prop}[(Local filter stability)]
\label{prop:lfstab}
Suppose there exists $\varepsilon>0$ such that
\[
\varepsilon\le p^v \bigl(x,z^v \bigr)\le
\varepsilon^{-1} \qquad\mbox{for all }v\in V, x,z\in\bbX.
\]
Let $\mu,\nu$ be probability measures on $\bbX$, and suppose that
\[
\corr(\mu,\beta) + 3 \bigl(1-\varepsilon^{2\Delta} \bigr)e^{2\beta r}
\Delta^2 \le\tfrac{1}{2}
\]
for a sufficiently small constant $\beta>0$. Then we have
\begin{eqnarray*}
&&\|\mathsf{F}_n\cdots\mathsf{F}_{s+1}\mu-
\mathsf{F}_n\cdots\mathsf{F}_{s+1}\nu\|_J
\\
& &\qquad\le2e^{-\beta(n-s)} \sum_{v\in J}\max
_{v'\in V} e^{-\beta d(v,v')} \sup_{x,z\in\bbX}\bigl\|
\mu^{v'}_{x,z}-\nu^{v'}_{x,z}\bigr\|
\end{eqnarray*}
for every $J\subseteq V$ and $s<n$.
\end{prop}

%
\begin{rem}
There is nothing magical about the constant $1/2$ in the decay of
correlations assumption; any constant $c<1$ would work at the expense
of a
constant $1/(1-c)$ rather than $2$ in the filter stability bound. As our
methods are not expected to yield tight quantitative bounds, we have taken
the liberty to fix various constants of this sort throughout the following
sections for aesthetic purposes.
\end{rem}

%
\begin{rem}
Note that by Lemma~\ref{lem:weight}
\[
\bigl\|\mu^{v'}_{x,z}-\nu^{v'}_{x,z}\bigr\| \le
\frac{2}{\varepsilon^{2\Delta}} \bigl\|\mu^{v'}_{x}-\nu^{v'}_{x}
\bigr\|.
\]
This yields a slightly cleaner bound in Proposition~\ref{prop:lfstab}
with a worse constant. For our purposes, however, it will be just
as easy to bound $\|\mu^{v'}_{x,z}-\nu^{v'}_{x,z}\|$ directly.
\end{rem}

\begin{pf*}{Proof of Proposition \ref{prop:lfstab}}
Define the smoothing distributions
\begin{eqnarray*}
\rho&=&\mathbf{P}^\mu[X_0,\ldots,X_n\in
\cdot|Y_1,\ldots,Y_n],
\\
\tilde\rho&=&\mathbf{P}^\nu[X_0,\ldots,X_n
\in\cdot|Y_1,\ldots,Y_n].
\end{eqnarray*}
We will apply Theorem~\ref{thmm:dobrushin} to $\rho,\tilde\rho$ with
$I=\{0,\ldots,n\}\times V$ and $\bbS=\bbX^{n+1}$ as discussed in
Section~\ref{sec:dobrushin}. To this end, we must bound the quantities $C_{ij}$
and $b_j$. We begin by bounding $C_{ij}$ with $i=(k,v)$ and $j=(k',v')$.
We distinguish three cases.

\textit{Case} $k=0$. The key observation in this case is that
$\rho^i_x = \mu^v_{x_0,x_1}$ by the Markov property (or by direct
computation). Note that as $\card N(v)\le\Delta$, we have
\[
\mu_{x,z}^v(A) = \frac{
\int\mathbf{1}_A(x^v)\prod_{w\in N(v)}p^w(x,z^w) \mu^v_x(dx^v)
}{\int\prod_{w\in N(v)}p^w(x,z^w) \mu^v_x(dx^v)} \ge
\varepsilon^{2\Delta} \mu^v_x(A),
\]
so $\|\mu_{x,z}^v-\mu_{x,z'}^v\|\le2(1-\varepsilon^{2\Delta})$
for any $z,z'\in\bbX$ by Lemma~\ref{lem:minorize}.
Therefore,
\[
C_{ij} \le\cases{ C^\mu_{vv'}, &\quad $\mbox{if
$k'=0$}$, \vspace*{2pt}
\cr
1-\varepsilon^{2\Delta}, &\quad $
\mbox{if $k'=1$ and $v'\in N(v)$}$,\vspace*{2pt}
\cr
0,
&\quad $ \mbox{otherwise}$.}
\]
This evidently implies that
\[
\sum_{(k',v')\in I} e^{\beta k'}e^{\beta d(v,v')}C_{(0,v)(k',v')}
\le\corr(\mu,\beta) + \bigl(1-\varepsilon^{2\Delta} \bigr)e^{\beta(r+1)}
\Delta.
\]

\textit{Case} $0<k<n$. Now we have (cf. Section~\ref{sec:dobrushin})
\[
\rho^i_x(A) = \frac{
\int\mathbf{1}_A(x_k^v)
p^v(x_{k-1},x_k^v) g^v(x_k^v,Y_k^v)\prod_{w\in N(v)}
p^w(x_k,x_{k+1}^w)
\psi^v(dx_k^v)
}{
\int
p^v(x_{k-1},x_k^v) g^v(x_k^v,Y_k^v)\prod_{w\in N(v)}
p^w(x_k,x_{k+1}^w)
\psi^v(dx_k^v)
}.
\]
By inspection, $\rho^i_x$ does not depend on $x_{k'}^{v'}$ except in the
following cases: $k'=k-1$ and $v'\in N(v)$; $k'=k+1$ and $v'\in N(v)$;
$k'=k$ and $v'\in\bigcup_{w\in N(v)}N(w)$. As
\[
\rho_x^i(A) \ge\varepsilon^{2\Delta}
\frac{
\int\mathbf{1}_A(x_k^v)
p^v(x_{k-1},x_k^v) g^v(x_k^v,Y_k^v)
\psi^v(dx_k^v)
}{
\int
p^v(x_{k-1},x_k^v) g^v(x_k^v,Y_k^v)
\psi^v(dx_k^v)
}
\]
as well as
\[
\rho^i_x(A) \ge\varepsilon^2
\frac{
\int\mathbf{1}_A(x_k^v)
g^v(x_k^v,Y_k^v)\prod_{w\in N(v)}
p^w(x_k,x_{k+1}^w)
\psi^v(dx_k^v)
}{
\int
g^v(x_k^v,Y_k^v)\prod_{w\in N(v)}
p^w(x_k,x_{k+1}^w)
\psi^v(dx_k^v)
},
\]
we can use Lemma~\ref{lem:minorize} to estimate
\[
C_{ij} \le%
\cases{ 1-\varepsilon^2, &\quad $\mbox{if
$k'=k-1$ and $v'\in N(v)$}$,\vspace*{2pt}
\cr
1-
\varepsilon^{2\Delta}, &\quad $\mbox{if $k'=k+1$ and
$v'\in N(v)$}$,\vspace*{2pt}
\cr
1-\varepsilon^{2\Delta}, &\quad $
\mbox{if $k'=k$ and $\displaystyle v'\in\bigcup
_{w\in N(v)}N(w)$}$,\vspace*{2pt}
\cr
0, & \quad $\mbox{otherwise}$. }
\]
This yields
\begin{eqnarray*}
\sum_{(k',v')\in I} e^{\beta|k-k'|}e^{\beta d(v,v')}C_{(k,v)(k',v')}
&\le& \bigl(1-\varepsilon^{2\Delta} \bigr) \bigl\{e^{2\beta r}
\Delta^2+ 2e^{\beta(r+1)}\Delta \bigr\}
\\
&\le&3 \bigl(1-\varepsilon^{2\Delta} \bigr)e^{2\beta r}
\Delta^2,
\end{eqnarray*}
where we have used that $r\ge1$ and $\Delta\ge1$ in the last inequality.

\textit{Case} $k=n$. Now we have
\begin{eqnarray*}
\rho^i_x(A) &=& \frac{
\int\mathbf{1}_A(x_n^v)
p^v(x_{n-1},x_n^v) g^v(x_n^v,Y_n^v)
\psi^v(dx_n^v)
}{
\int
p^v(x_{n-1},x_n^v) g^v(x_n^v,Y_n^v)
\psi^v(dx_n^v)
}
\\
&\ge&\varepsilon^2 \frac{
\int\mathbf{1}_A(x_n^v)
g^v(x_n^v,Y_n^v)
\psi^v(dx_n^v)
}{
\int
g^v(x_n^v,Y_n^v)
\psi^v(dx_n^v)
},
\end{eqnarray*}
and we obtain precisely as above
\[
C_{ij} \le%
\cases{ 1-\varepsilon^2, &\quad $\mbox{if
$k'=n-1$ and $v'\in N(v)$}$,\vspace*{2pt}
\cr
0, &\quad $
\mbox{otherwise}$.} %
\]
We therefore find
\[
\sum_{(k',v')\in I} e^{\beta|k-k'|}e^{\beta d(v,v')}C_{(n,v)(k',v')}
\le \bigl(1-\varepsilon^2 \bigr) e^{\beta(r+1)}\Delta.
\]

Combining the above three cases and the assumption of the
proposition yields
\[
\max_{(k,v)\in I}\sum_{(k',v')\in I}e^{\beta\{|k-k'|+d(v,v')\}}
C_{(k,v)(k',v')}\le\frac{1}{2}.
\]
Thus, Lemma~\ref{lem:expdecay} gives
\[
\max_{(k,v)\in I}\sum_{(k',v')\in I}e^{\beta\{|k-k'|+d(v,v')\}}
D_{(k,v)(k',v')}\le2.
\]
Now consider the quantities $b_j$ in Theorem~\ref{thmm:dobrushin}.
By the Markov property, it is evident that $\rho^i_x=\tilde\rho^i_x$
whenever $i=(k,v)$ with $k\ge1$. On the other hand, for $k=0$ we
obtain $\rho^i_x=\mu^v_{x_0,x_1}$ and $\tilde\rho^i_x=\nu^v_{x_0,x_1}$.
Applying Theorem~\ref{thmm:dobrushin} therefore yields
\[
\bigl\|\pi_n^\mu-\pi_n^\nu
\bigr\|_J = \|\rho-\tilde\rho\|_{\{n\}\times J} \le\sum
_{v\in J}\sum_{v'\in V}D_{(n,v)(0,v')}
\sup_{x,z\in\bbX}\bigl\|\mu^{v'}_{x,z}-
\nu^{v'}_{x,z}\bigr\|.
\]
However, note that
\begin{eqnarray*}
&&\sum_{v'\in V}D_{(n,v)(0,v')} \sup
_{x,z\in\bbX}\bigl\|\mu^{v'}_{x,z}-
\nu^{v'}_{x,z}\bigr\|
\\
&&\qquad = e^{-\beta n} \sum_{v'\in V}
e^{\beta\{n+d(v,v')\}}D_{(n,v)(0,v')} e^{-\beta d(v,v')} \sup_{x,z\in
\bbX}
\bigl\|\mu^{v'}_{x,z}-\nu^{v'}_{x,z}\bigr\|
\\
&&\qquad \le2e^{-\beta n}\max_{v'\in V} e^{-\beta d(v,v')} \sup
_{x,z\in\bbX}\bigl\|\mu^{v'}_{x,z}-
\nu^{v'}_{x,z}\bigr\|,
\end{eqnarray*}
using the above estimate on the matrix $D$. Substituting this into
the bound for $\|\pi_n^\mu-\pi_n^\nu\|_J$ yields the statement of
the proposition for the special case $s=0$.

To obtain the result for any $s<n$, note that
$\mathsf{F}_n\cdots\mathsf{F}_{s+1}\mu$ and $\pi_{n-s}^\mu$
differ only in
that a different sequence of observations ($Y_{s+1},\ldots,Y_n$ versus
$Y_1,\ldots,Y_{n-s}$) is used in the computation of these quantities. As
our bound holds uniformly in the observation sequence, however, the
general result follows immediately.
\end{pf*}

As a corollary of Proposition~\ref{prop:lfstab}, let us derive a simple
filter stability statement that illustrates the role of decay of
correlations (this will not be used elsewhere).

%
\begin{cor}[(Filter stability)]
\label{cor:fstab}
Suppose there exists $\varepsilon>0$ such that
\[
\varepsilon\le p^v \bigl(x,z^v \bigr)\le
\varepsilon^{-1} \qquad\mbox{for all }v\in V, x,z\in\bbX,
\]
and such that
\[
\varepsilon>\varepsilon_0= \biggl(1-\frac{1}{6\Delta^2}
\biggr)^{1/2\Delta}.
\]
Then for any probability measures $\mu,\nu$ on $\bbX$ and
$J\subseteq V$, $n\ge0$, we have
\[
\bigl\|\pi_n^\mu-\pi_n^\nu
\bigr\|_J \le4\card J \gamma^{n/2r},
\]
where $\gamma=6\Delta^2(1-\varepsilon^{2\Delta})<1$.
\end{cor}

\begin{pf}
We first apply Proposition~\ref{prop:lfstab} with
$\mu=\delta_x$. Then $\corr(\mu,\beta)=0$ for any $\beta>0$.
Choosing $\beta= -(2r)^{-1}\log\gamma>0$, we find that
\[
\corr(\mu,\beta) + 3 \bigl(1-\varepsilon^{2\Delta} \bigr)e^{2\beta r}
\Delta^2 = \tfrac{1}{2},
\]
so that the assumption of Proposition~\ref{prop:lfstab} is satisfied.
Therefore,
\[
\bigl\|\pi_n^x-\pi_n^\nu
\bigr\|_J \le4\card J e^{-\beta n} = 4\card J \gamma^{n/2r}.
\]
To obtain the result for arbitrary $\mu$, note that
\begin{eqnarray*}
\pi_n^\mu(A) &=& \mathbf{P}^\mu[X_n
\in A|Y_1,\ldots,Y_n]
\\
&=& \mathbf{E}^\mu \bigl[\mathbf{P}^\mu[X_n
\in A|X_0,Y_1,\ldots,Y_n]|
Y_1, \ldots,Y_n \bigr]
\\
&=& \mathbf{E}^\mu \bigl[\pi_n^{\delta_{X_0}}(A)|Y_1,
\ldots,Y_n \bigr].
\end{eqnarray*}
Therefore, by Jensen's inequality,
\[
\bigl\|\pi_n^\mu-\pi_n^\nu
\bigr\|_J \le\mathbf{E}^\mu \bigl[\bigl\|\pi_n^{\delta_{X_0}}-
\pi_n^\nu\bigr\|_J|Y_1,
\ldots,Y_n \bigr] \le\sup_{x\in\bbX} \bigl\|
\pi_n^x-\pi_n^\nu
\bigr\|_J,
\]
which yields the result.
\end{pf}

While Proposition~\ref{prop:lfstab} requires a decay of correlations
assumption on the initial condition [$\corr(\mu,\beta)$ must be
sufficiently small], Corollary~\ref{cor:fstab} works for any initial
condition provided that $\varepsilon>\varepsilon_0$ is sufficiently large
(which is necessary in general, see Section~\ref{sec:mixing}). Thus, no
assumption is needed on the initial condition if we want to show only that
the filter is stable in time. On the other hand, Proposition~\ref
{prop:lfstab} controls not only the stability in time, but also the
spatial accumulation of error between $\mu$ and $\nu$ by virtue of the
damping factor $e^{-\beta d(v,v')}$: the decay of correlations property of
the initial condition is essential to obtain this type of local control.
The latter is of central importance if we wish to obtain local error
bounds for filter approximations that are uniform in time and in the model
dimension.

\subsection{The block projection error}
\label{sec:1step}

The proof of a time-uniform error bound between $\pi_n^\mu$ and
$\tilde\pi_n^\mu$ requires two ingredients: we need the filter stability
property of $\pi_n^\mu$, developed in the previous section, in order to
mitigate the accumulation of approximation errors over time; and we need
to control the approximation error between $\pi_n^\mu$ and
$\tilde\pi_n^\mu$ in one time step. The latter is the purpose of this
section.

We will in fact consider two separate cases. To control the total error
$\|\pi_n^\mu-\tilde\pi_n^\mu\|_J$, we need to consider the
one-step error
made in each time step $s=1,\ldots,n$. For time steps $s<n$ (for which
the error is dissipated by the stability of the filter), the error must be
measured in terms of the quantities that appear in Proposition~\ref
{prop:lfstab}: that is, we must control
$\|(\mathsf{F}_s\nu)^v_{x,z}-(\mathsf{\tilde F}_s\nu)^v_{x,z}\|$.
On the
other hand, in the last time step $s=n$, we must control directly
$\|\mathsf{F}_n\nu-\mathsf{\tilde F}_n\nu\|_J$. While the proofs of these
cases are quite similar, each must be considered separately in the
following.

We begin by bounding the error in time steps $s<n$.

%
\begin{prop}[(Block error, $s<n$)]
\label{prop:1steps}
Suppose there exists $\varepsilon>0$ such that
\[
\varepsilon\le p^v \bigl(x,z^v \bigr)\le
\varepsilon^{-1} \qquad\mbox{for all }v\in V, x,z\in\bbX.
\]
Let $\nu$ be a probability measure on $\bbX$, and suppose that
\[
\corr(\nu,\beta) + \bigl(1-\varepsilon^2 \bigr)e^{\beta
(r+1)}
\Delta \le\tfrac{1}{2}
\]
for a sufficiently small constant $\beta>0$. Then we have
\[
\sup_{x,z\in\bbX} \bigl\|(\mathsf{F}_s\nu)^v_{x,z}-(
\mathsf{\tilde F}_s\nu)^v_{x,z}\bigr\|
\le4e^{-\beta} \bigl(1-\varepsilon^{2\Delta} \bigr) e^{-\beta
d(v,\partial K)}
\]
for every $s\in\mathbb{N}$, $K\in\mathcal{K}$ and $v\in K$.
\end{prop}

This result makes precise the idea that was heuristically expressed in
Section~\ref{sec:blockfilt}: if the measure $\nu$ possesses the decay of
correlations property, then the error at site $v$ incurred by applying the
block filter rather than the true filter decays exponentially in the
distance between $v$ and the boundary of the block that it is in.

\begin{pf*}{Proof of Proposition \ref{prop:1steps}}
We begin by writing out the definitions
\begin{eqnarray*}
(\mathsf{F}_s\nu) (A) &=& \frac{
\int\mathbf{1}_A(x) \prod_{w\in V}p^w(x_0,x^w) g^w(x^w,Y_s^w)
\nu(dx_0) \psi(dx)
}{
\int\prod_{w\in V}p^w(x_0,x^w) g^w(x^w,Y_s^w)
\nu(dx_0) \psi(dx)
},
\\
(\mathsf{\tilde F}_s\nu) (A) &=& \frac{
\int\mathbf{1}_A(x)
\prod_{K'\in\mathcal{K}} [\int\prod_{w\in K'}p^w(x_0,x^w)
g^w(x^w,Y_s^w) \nu(dx_0) ]
\psi(dx)
}{
\int
\prod_{K'\in\mathcal{K}} [\int\prod_{w\in K'}p^w(x_0,x^w)
g^w(x^w,Y_s^w) \nu(dx_0) ]
\psi(dx)
}.
\end{eqnarray*}
Let us fix $K\in\mathcal{K}$, $v\in K$ throughout the proof. Then
\begin{eqnarray*}
(\mathsf{F}_s\nu)^v_x(A) &=&
\frac{
\int\mathbf{1}_A(x^v) g^v(x^v,Y_s^v)
\prod_{w\in V}p^w(x_0,x^w)
\nu(dx_0) \psi^v(dx^v)
}{
\int g^v(x^v,Y_s^v)
\prod_{w\in V}p^w(x_0,x^w)
\nu(dx_0) \psi^v(dx^v)
},
\\
(\mathsf{\tilde F}_s\nu)^v_x(A) &=&
\frac{
\int\mathbf{1}_A(x^v) g^v(x^v,Y_s^v)
\prod_{w\in K}p^w(x_0,x^w) \nu(dx_0)
\psi^v(dx^v)
}{
\int g^v(x^v,Y_s^v)
\prod_{w\in K}p^w(x_0,x^w) \nu(dx_0)
\psi^v(dx^v)
}.
\end{eqnarray*}
Define $I=(\{0\}\times V)\cup(1,v)$ and $\bbS=\bbX\times\bbX^v$, and
the probability measures on $\bbS$
{\fontsize{10.5}{12.5}\selectfont
\begin{eqnarray*}
&&\hspace*{-4pt}\rho(A)
\\
&&\hspace*{-4pt}\qquad =\frac{
\int\mathbf{1}_A(x_0,x^v) g^v(x^v,Y_s^v)
\prod_{w\in V}p^w(x_0,x^w)
\prod_{u\in N(v)}p^u(x,z^u)
\nu(dx_0) \psi^v(dx^v)
}{
\int g^v(x^v,Y_s^v)
\prod_{w\in V}p^w(x_0,x^w)
\prod_{u\in N(v)}p^u(x,z^u)
\nu(dx_0) \psi^v(dx^v)
},
\\
&&\hspace*{-4pt}\tilde\rho(A)
\\
&&\qquad\hspace*{-4pt}= \frac{
\int\mathbf{1}_A(x_0,x^v) g^v(x^v,Y_s^v)
\prod_{w\in K}p^w(x_0,x^w)
\prod_{u\in N(v)}p^u(x,z^u)
\nu(dx_0)
\psi^v(dx^v)
}{
\int g^v(x^v,Y_s^v)
\prod_{w\in K}p^w(x_0,x^w)
\prod_{u\in N(v)}p^u(x,z^u)
\nu(dx_0)
\psi^v(dx^v)
}.
\end{eqnarray*}}
\hspace*{-3pt}Then we have by construction
\[
\bigl\|(\mathsf{F}_s\nu)^v_{x,z}- (\mathsf{\tilde
F}_s\nu)^v_{x,z}\bigr\|= \|\rho-\tilde\rho
\|_{(1,v)}.
\]
We will apply Theorem~\ref{thmm:dobrushin} to bound
$\|\rho-\tilde\rho\|_{(1,v)}$. To this end, we must bound $C_{ij}$
and $b_i$ with $i=(k',v')$ and $j=(k'',v'')$. We distinguish two cases.

\textit{Case} $k'=0$. In this case, we have
\begin{eqnarray*}
\rho^i_{(x_0,x^v)}(A) &= &\frac{
\int\mathbf{1}_A(x^{v'}_0)\prod_{w\in N(v')}p^w(x_0,x^w)
\nu^{v'}_{x_0}(dx^{v'}_0)
}{\int\prod_{w\in N(v')}p^w(x_0,x^w) \nu^{v'}_{x_0}(dx^{v'}_0)},
\\
\tilde\rho^i_{(x_0,x^v)}(A) &=& \frac{
\int\mathbf{1}_A(x^{v'}_0)\prod_{w\in N(v')\cap K}p^w(x_0,x^w)
\nu^{v'}_{x_0}(dx^{v'}_0)
}{\int\prod_{w\in N(v')\cap K}p^w(x_0,x^w) \nu^{v'}_{x_0}(dx^{v'}_0)}.
\end{eqnarray*}
In particular, $\rho^i_{(x_0,x^v)}=\nu^{v'}_{x_0,x}$, so
$C_{ij}\le C_{v'v''}^\nu$ if $k''=0$. Moreover, as
\[
\rho^i_{(x_0,x^v)}(A) \ge\varepsilon^2
\frac{
\int\mathbf{1}_A(x^{v'}_0)\prod_{w\in
N(v')\setminus\{v\}}p^w(x_0,x^w)
\nu^{v'}_{x_0}(dx^{v'}_0)
}{\int\prod_{w\in N(v')\setminus\{v\}
}p^w(x_0,x^w) \nu^{v'}_{x_0}(dx^{v'}_0)},
\]
we have $C_{ij}\le1-\varepsilon^2$ if $k''=1$ (so $v''=v$)
and $v\in N(v')$ by Lemma~\ref{lem:minorize}, and $C_{ij}=0$ otherwise.
We therefore immediately obtain the estimate
\[
\sum_{(k'',v'')\in I}e^{\beta k''}e^{\beta d(v',v'')}
C_{(0,v')(k'',v'')} \le\corr(\nu,\beta) + \bigl(1-\varepsilon^2
\bigr)e^{\beta(r+1)}.
\]
On the other hand, note that $\rho^i_{(x_0,x^v)}=\tilde\rho^i_{(x_0,x^v)}$
if $N(v')\subseteq K$, and that we have
$\rho^i_{(x_0,x^v)}\ge\varepsilon^{2\Delta}\nu^{v'}_{x_0}$ and
$\tilde\rho^i_{(x_0,x^v)}\ge\varepsilon^{2\Delta}\nu^{v'}_{x_0}$.
Therefore, by Lemma~\ref{lem:minorize}
\[
b_i = \sup_{(x_0,x^v)\in\bbS}\bigl \|\rho^i_{(x_0,x^v)}-
\tilde\rho^i_{(x_0,x^v)}\bigr\| \le%
\cases{ 0, &\quad $\mbox{for
$v'\in K\setminus\partial K$}$,\vspace*{2pt}
\cr
2 \bigl(1-
\varepsilon^{2\Delta} \bigr), &\quad  $\mbox{otherwise}$.} %
\]

\textit{Case} $k'=1$. In this case, we have
\begin{eqnarray*}
\rho^i_{(x_0,x^v)}(A)& =& \tilde\rho^i_{(x_0,x^v)}(A)\\
&=&
\frac{\int\mathbf{1}_A(x^v) g^v(x^v,Y_s^v)
p^v(x_0,x^v)\prod_{u\in N(v)}p^u(x,z^u)
\psi^v(dx^v)
}{
\int g^v(x^v,Y_s^v)
p^v(x_0,x^v)\prod_{u\in N(v)}p^u(x,z^u)
\psi^v(dx^v)
}.
\end{eqnarray*}
Thus, $b_i=0$, and estimating as above we obtain $C_{ij}\le
1-\varepsilon^2$ whenever $k''=0$ and $v''\in N(v)$, and $C_{ij}=0$
otherwise. In particular, we obtain
\[
\sum_{(k'',v'')\in I}e^{\beta|1-k''|}e^{\beta d(v,v'')}
C_{(1,v)(k'',v'')} \le \bigl(1-\varepsilon^2 \bigr)e^{\beta(r+1)}
\Delta.
\]

Combining the above two cases and the assumption of the proposition yields
\[
\max_{(k',v')\in I}\sum_{(k'',v'')\in I}
e^{\beta\{|k'-k''|+d(v',v'')\}} C_{(k',v')(k'',v'')}\le\frac{1}{2}.
\]
Applying Theorem~\ref{thmm:dobrushin} and
Lemma~\ref{lem:expdecay} gives
\begin{eqnarray*}
\bigl\|(\mathsf{F}_s\nu)^v_{x,z}- (\mathsf{\tilde
F}_s\nu)^v_{x,z}\bigr\| &=& \|\rho-\tilde\rho
\|_{(1,v)}
\\
&\le&2 \bigl(1-\varepsilon^{2\Delta} \bigr) \sum
_{v'\in V\setminus(K\setminus\partial K)} D_{(1,v)(0,v')}
\\
&\le&4e^{-\beta} \bigl(1-\varepsilon^{2\Delta} \bigr)
e^{-\beta
d(v,\partial K)}.
\end{eqnarray*}
As the choice of $x,z\in\bbX$ was arbitrary, the proof is complete.
\end{pf*}

We now use a similar argument to bound the error in time step $n$.

%
\begin{prop}[(Block error, $s=n$)]
\label{prop:1stepn}
Suppose there exists $\varepsilon>0$ such that
\[
\varepsilon\le p^v \bigl(x,z^v \bigr)\le
\varepsilon^{-1}\qquad \mbox{for all }v\in V, x,z\in\bbX.
\]
Let $\nu$ be a probability measure on $\bbX$, and suppose that
\[
\corr(\nu,\beta) + \bigl(1-\varepsilon^2 \bigr)e^{\beta
(r+1)}
\Delta \le\tfrac{1}{2}
\]
for a sufficiently small constant $\beta>0$. Then we have
\[
\|\mathsf{F}_n\nu-\mathsf{\tilde F}_n\nu
\|_{J} \le4e^{-\beta} \bigl(1-\varepsilon^{2\Delta} \bigr)
e^{-\beta d(J,\partial K)}\card J
\]
for every $K\in\mathcal{K}$ and $J\subseteq K$.
\end{prop}

\begin{pf}
Define $I=\{0,1\}\times V$ and $\bbS=\bbX^2$.
Fix $K\in\mathcal{K}$, and let
\begin{eqnarray*}
\rho(A) &=& \frac{
\int\mathbf{1}_A(x_0,x_1)
\prod_{v\in V}p^v(x_0,x_1^v) g^v(x^v_1,Y_n^v)
\nu(dx_0) \psi(dx_1)
}{
\int
\prod_{v\in V}p^v(x_0,x_1^v) g^v(x^v_1,Y_n^v)
\nu(dx_0) \psi(dx_1)
},
\\
\tilde\rho(A) &=& \frac{
\int\mathbf{1}_A(x_0,x_1)
\prod_{v\in K}p^v(x_0,x_1^v)
\prod_{w\in V}g^w(x^w_1,Y_n^w)
\nu(dx_0) \psi(dx_1)
}{
\int
\prod_{v\in K}p^v(x_0,x_1^v)
\prod_{w\in V}g^w(x^w_1,Y_n^w)
\nu(dx_0) \psi(dx_1)
}.
\end{eqnarray*}
Then for any $J\subseteq K$, we have
\[
\|\mathsf{F}_n\nu-\mathsf{\tilde F}_n\nu
\|_{J} = \|\rho-\tilde\rho\|_{\{1\}\times J}.
\]
We will apply Theorem~\ref{thmm:dobrushin} to bound
$\|\rho-\tilde\rho\|_{\{1\}\times J}$. To this end, we must bound $C_{ij}$
and $b_i$ with $i=(k,v)$ and $j=(k',v')$. We distinguish two cases.

\textit{Case} $k=0$. In this case, we have
\begin{eqnarray*}
\rho^i_x(A) &=& \frac{
\int\mathbf{1}_A(x^v_0)\prod_{w\in N(v)}p^w(x_0,x^w_1)
\nu^{v}_{x_0}(dx^{v}_0)
}{
\int\prod_{w\in N(v)}p^w(x_0,x^w_1)
\nu^{v}_{x_0}(dx^{v}_0)
},
\\
\tilde\rho^i_x(A) &=& \frac{
\int\mathbf{1}_A(x^v_0)\prod_{w\in N(v)\cap K}p^w(x_0,x^w_1)
\nu^{v}_{x_0}(dx^{v}_0)
}{
\int\prod_{w\in N(v)\cap K}p^w(x_0,x^w_1)
\nu^{v}_{x_0}(dx^{v}_0)
}.
\end{eqnarray*}
In particular, $\rho^i_x=\nu^{v}_{x_0,x_1}$, so
$C_{ij}\le C_{vv'}^\nu$ if $k'=0$. Moreover, as
\[
\rho^i_x(A) \ge\varepsilon^2
\frac{
\int\mathbf{1}_A(x^v_0)\prod_{w\in N(v)\setminus\{v'\}}p^w(x_0,x^w_1)
\nu^{v}_{x_0}(dx^{v}_0)
}{
\int\prod_{w\in N(v)\setminus\{v'\}}p^w(x_0,x^w_1)
\nu^{v}_{x_0}(dx^{v}_0)
},
\]
we have $C_{ij}\le1-\varepsilon^2$ if $k'=1$
and $v'\in N(v)$ by Lemma~\ref{lem:minorize}, and $C_{ij}=0$ otherwise.
We therefore immediately obtain the estimate
\[
\sum_{(k',v')\in I}e^{\beta k'}e^{\beta d(v,v')}
C_{(0,v)(k',v')} \le\corr(\nu,\beta) + \bigl(1-\varepsilon^2
\bigr)e^{\beta(r+1)}\Delta.
\]
On the other hand, note that $\rho^i_x=\tilde\rho^i_x$
if $N(v)\subseteq K$, and that we have
$\rho^i_x\ge\varepsilon^{2\Delta}\nu^{v}_{x_0}$ and
$\tilde\rho^i_x\ge\varepsilon^{2\Delta}\nu^{v}_{x_0}$.
Therefore, we obtain by Lemma~\ref{lem:minorize}
\[
b_i = \sup_{x\in\bbS} \bigl\|\rho^i_x-
\tilde\rho^i_x\bigr\| \le%
\cases{ 0, &\quad $\mbox{for $v
\in K\setminus\partial K$}$,\vspace*{2pt}
\cr
2 \bigl(1-\varepsilon^{2\Delta}
\bigr), &\quad $\mbox{otherwise}$.} %
\]

\textit{Case} $k=1$. In this case, we have
\[
\rho^i_x(A) = \frac{
\int\mathbf{1}_A(x^v_1)
p^v(x_0,x^v_1) g^v(x^v_1,Y_n^v) \psi^v(dx^v_1)
}{
\int
p^v(x_0,x^v_1) g^v(x^v_1,Y_n^v) \psi^v(dx^v_1)
},
\]
while $\tilde\rho_i^x=\rho_i^x$ if $v\in K$ and
\[
\tilde\rho^i_x(A) = \frac{
\int\mathbf{1}_A(x^v_1)
g^v(x^v_1,Y_n^v) \psi^v(dx^v_1)
}{
\int
g^v(x^v_1,Y_n^v) \psi^v(dx^v_1)
},
\]
otherwise. Thus, we obtain from Lemma~\ref{lem:minorize}
\[
b_i = \sup_{x\in\bbS}\bigl \|\rho^i_x-
\tilde\rho^i_x\bigr\| \le%
\cases{ 0, &\quad $\mbox{for $v
\in K$}$,\vspace*{2pt}
\cr
2 \bigl(1-\varepsilon^2 \bigr), &\quad $
\mbox{otherwise}$.} %
\]
On the other hand, we can readily estimate as above
\[
\sum_{(k',v')\in I}e^{\beta|1-k'|}e^{\beta d(v,v')}
C_{(1,v)(k',v')} \le \bigl(1-\varepsilon^2 \bigr)e^{\beta(r+1)}
\Delta.
\]

Combining the above two cases and the assumption of the proposition yields
\[
\max_{(k,v)\in I}\sum_{(k',v')\in I}
e^{\beta\{|k-k'|+d(v,v')\}} C_{(k,v)(k',v')}\le\frac{1}{2}.
\]
Applying Theorem~\ref{thmm:dobrushin} and
Lemma~\ref{lem:expdecay} gives
\begin{eqnarray*}
\|\mathsf{F}_n\nu-\mathsf{\tilde F}_n\nu
\|_{J} &= &\|\rho-\tilde\rho\|_{\{1\}\times J}
\\
& \le&2 \bigl(1-\varepsilon^{2\Delta} \bigr) \sum
_{v\in J} \biggl\{ \sum_{v'\in(V\setminus K)\cup\partial K}
D_{(1,v)(0,v')} + \sum_{v'\in V\setminus K} D_{(1,v)(1,v')}
\biggr\}
\\
& \le& 4e^{-\beta} \bigl(1-\varepsilon^{2\Delta} \bigr)
e^{-\beta
d(J,\partial K)} \card J
\end{eqnarray*}
for every $J\subseteq K$.
\end{pf}

\subsection{Decay of correlations of the block filter}
\label{sec:decay}

To idea behind the block filter $\tilde\pi_n^\mu$ is that the error should
decay exponentially in the block size by virtue of the decay of
correlations property. While we have developed above the two ingredients
(filter stability and one-step error bound) required to obtain a
time-uniform error bound between $\pi_n^\mu$ and $\tilde\pi_n^\mu
$, we
have done this by imposing the decay of correlations property as an
assumption. Thus, perhaps the crucial point remains to be proved: we must
show that decay of correlations does indeed hold, that is,
$\corr(\tilde\pi_n^\mu,\beta)$ can be controlled uniformly in
time. This
is the goal of the present section.

Unfortunately, $\corr(\tilde\pi_n^\mu,\beta)$ is not
straightforward to
control directly. We therefore introduce an alternative
measure of correlation decay that will be easier to control.
For any probability measure $\mu$ on $\bbX$ and $x,z\in\bbX$, $v\in V$,
$K\in\mathcal{K}$, let
\begin{eqnarray*}
\mu_{x,z}^{v,K}(A) &:=  &\mathbf{P}^\mu
\bigl[X_0^v\in A| X_0^{V\setminus\{v\}}=x^{V\setminus\{v\}},
X_1^K=z^K \bigr]
\\
&=  &\frac{
\int\mathbf{1}_A(x^v)\prod_{w\in N(v)\cap K}p^w(x,z^w) \mu^v_x(dx^v)
}{\int\prod_{w\in N(v)\cap K}p^w(x,z^w) \mu^v_x(dx^v)}.
\end{eqnarray*}
We now define
\[
\tilde C^\mu_{vv'} := \frac{1}{2} \max
_{K\in\mathcal{K}} \sup_{z\in\bbX} \sup_{x,\tilde x\in\bbX
:x^{V\setminus\{v'\}}=
\tilde x^{V\setminus\{v'\}}}
\bigl\|\mu^{v,K}_{x,z}-\mu^{v,K}_{\tilde x,z}\bigr\|
\]
for $v,v'\in V$. The quantity
\[
\tcorr(\mu,\beta) := \max_{v\in V}\sum
_{v'\in V} e^{\beta d(v,v')}\tilde C^\mu_{vv'}
\]
is a measure of correlation decay that is well adapted to the block
filter. In order for this quantity to be useful, we must first show
that it controls $\corr(\mu,\beta)$.

%
\begin{lem}
\label{lem:corrcompare}
For any probability measure $\mu$ and $\beta>0$, we have
\[
\corr(\mu,\beta)\le \bigl(1-\varepsilon^{2\Delta} \bigr)e^{2\beta r}
\Delta^2 + 2\varepsilon^{-2\Delta}\tcorr(\mu,\beta).
\]
\end{lem}

\begin{pf}
By definition
\[
\mu_{x,z}^{v}(A) = \frac{
\int\mathbf{1}_A(x^v)\prod_{w\in N(v)\setminus K}p^w(x,z^w) \mu
_{x,z}^{v,K}(dx^v)
}{
\int\prod_{w\in N(v)\setminus K}p^w(x,z^w) \mu_{x,z}^{v,K}(dx^v)}.
\]
Let $x,\tilde x\in\bbX$ be such that $x^{V\setminus\{v'\}}=
\tilde x^{V\setminus\{v'\}}$. If
$v'\notin\bigcup_{w\in N(v)}N(w)$, then
\[
\bigl\|\mu_{x,z}^v-\mu_{\tilde x,z}^v\bigr\|\le2
\varepsilon^{-2\Delta}\bigl \|\mu_{x,z}^{v,K}-
\mu_{\tilde x,z}^{v,K}\bigr\|
\]
by Lemma~\ref{lem:weight}. On the other hand, note that
\[
\mu_{x,z}^{v}(A) \ge\varepsilon^{2\Delta}
\mu_{x,z}^{v,K}(A), \qquad\mu_{\tilde x,z}^{v}(A) \ge
\varepsilon^{2\Delta}\mu_{\tilde x,z}^{v,K}(A).
\]
We can therefore estimate using Lemma~\ref{lem:minorize}
for $v'\in\bigcup_{w\in N(v)}N(w)$
\[
\bigl\|\mu_{x,z}^{v}-\mu_{\tilde x,z}^{v}\bigr\| \le2
\bigl(1-\varepsilon^{2\Delta} \bigr)+ \varepsilon^{2\Delta}\bigl\|
\mu_{x,z}^{v,K}- \mu_{\tilde x,z}^{v,K}\bigr\|.
\]
Thus, we obtain
\begin{eqnarray*}
\corr(\mu,\beta) &\le& \bigl(1-\varepsilon^{2\Delta} \bigr) \max
_{v\in V}\sum_{v'\in\bigcup_{w\in N(v)}N(w)} e^{\beta d(v,v')}
+ 2\varepsilon^{-2\Delta}\tcorr(\mu,\beta)
\\
&\le &\bigl(1-\varepsilon^{2\Delta} \bigr)e^{2\beta r}
\Delta^2 + 2\varepsilon^{-2\Delta}\tcorr(\mu,\beta).
\end{eqnarray*}
As $\mu$ and $\beta$ were arbitrary, the proof is complete.
\end{pf}

We now aim to establish a time-uniform bound on
$\tcorr(\tilde\pi_n^\mu,\beta)$. To this end, we first prove
a one-step bound which will subsequently be iterated.

%
\begin{prop}
\label{prop:tcorr}
Suppose there exists $\varepsilon>0$ such that
\[
\varepsilon\le p^v \bigl(x,z^v \bigr)\le
\varepsilon^{-1} \qquad\mbox{for all }v\in V, x,z\in\bbX.
\]
Let $\nu$ be a probability measure on $\bbX$, and suppose that
\[
\tcorr(\nu,\beta) + \bigl(1-\varepsilon^2 \bigr)e^{\beta(r+1)}
\Delta\le\tfrac{1}{2}
\]
for a sufficiently small constant $\beta>0$. Then we have
\[
\tcorr(\mathsf{\tilde F}_s\nu,\beta) \le2 \bigl(1-
\varepsilon^{2\Delta} \bigr)e^{2\beta r}\Delta^2
\]
for any $s\in\mathbb{N}$.
\end{prop}

\begin{pf}
Let $K,K'\in\mathcal{K}$, $v\in K$,
$v'\in V$ ($v'\ne v$), and let $z,x,\tilde x\in\bbX$ such that
$x^{V\setminus\{v'\}}=\tilde x^{V\setminus\{v'\}}$.
These choices will be fixed until further notice.

Define $I=(\{0\}\times V)\cup(1,v)$ and $\bbS=\bbX\times\bbX^v$, and
let
{\fontsize{10.1}{12.1}\selectfont
\begin{eqnarray*}
&&\hspace*{-4pt}\rho(A)
\\
&&\hspace*{-4pt}\qquad= \frac{
\int\mathbf{1}_A(x_0,x^v) g^v(x^v,Y_s^v)
\prod_{w\in K}p^w(x_0,x^w)
\prod_{u\in N(v)\cap K'}p^u(x,z^u)
\nu(dx_0)
\psi^v(dx^v)
}{
\int g^v(x^v,Y_s^v)
\prod_{w\in K}p^w(x_0,x^w)
\prod_{u\in N(v)\cap K'}p^u(x,z^u)
\nu(dx_0)
\psi^v(dx^v)
},
\\
&&\hspace*{-4pt}\tilde\rho(A)
\\
&&\hspace*{-4pt}\qquad =\frac{
\int\mathbf{1}_A(x_0,\tilde x^v) g^v(\tilde x^v,Y_s^v)
\prod_{w\in K}p^w(x_0,\tilde x^w)
\prod_{u\in N(v)\cap K'}p^u(\tilde x,z^u)
\nu(dx_0)
\psi^v(d\tilde x^v)
}{
\int g^v(\tilde x^v,Y_s^v)
\prod_{w\in K}p^w(x_0,\tilde x^w)
\prod_{u\in N(v)\cap K'}p^u(\tilde x,z^u)
\nu(dx_0)
\psi^v(d\tilde x^v)
}.
\end{eqnarray*}}
\hspace*{-3pt}Then we have by construction
\[
\bigl\|(\mathsf{\tilde F}_s\nu)^{v,K'}_{x,z}- (
\mathsf{\tilde F}_s\nu)^{v,K'}_{\tilde x,z} \bigr\|= \|\rho-
\tilde\rho\|_{(1,v)}.
\]
We will apply Theorem~\ref{thmm:dobrushin} to bound
$\|\rho-\tilde\rho\|_{(1,v)}$. To this end, we must bound $C_{ij}$
and $b_i$ with $i=(k,t)$ and $j=(k',t')$. We distinguish two cases.

\textit{Case} $k=0$. In this case, we have
\begin{eqnarray*}
\rho^i_{(x_0,x^v)}(A) &=& \frac{
\int\mathbf{1}_A(x^{t}_0)\prod_{w\in N(t)\cap K}p^w(x_0,x^w)
\nu^{t}_{x_0}(dx^{t}_0)
}{\int\prod_{w\in N(t)\cap K}p^w(x_0,x^w) \nu^{t}_{x_0}(dx^{t}_0)},
\\
\tilde\rho^i_{(x_0,\tilde x^v)}(A) &=& \frac{
\int\mathbf{1}_A(x^{t}_0)\prod_{w\in N(t)\cap K}p^w(x_0,\tilde x^w)
\nu^{t}_{x_0}(dx^{t}_0)
}{\int\prod_{w\in N(t)\cap K}p^w(x_0,\tilde x^w) \nu^{t}_{x_0}(dx^{t}_0)}.
\end{eqnarray*}
Note that $\rho^i_{(x_0,x^v)}=\nu^{t,K}_{x_0,x}$. We therefore have
$C_{ij} \le\tilde C^\nu_{tt'}$ when $k'=0$.
Moreover,
\[
\rho^i_{(x_0,x^v)}(A) \ge\varepsilon^2
\frac{
\int\mathbf{1}_A(x^{t}_0)\prod_{w\in N(t)\cap(K\setminus\{v\})}
p^w(x_0,x^w)
\nu^{t}_{x_0}(dx^{t}_0)
}{\int\prod_{w\in N(t)\cap(K\setminus\{v\})}
p^w(x_0,x^w) \nu^{t}_{x_0}(dx^{t}_0)}
\]
implies $C_{ij}\le1-\varepsilon^2$ if $k'=1$
and $v\in N(t)$ by Lemma~\ref{lem:minorize}, and $C_{ij}=0$ otherwise.
On the other hand, note that as $x^{V\setminus\{v'\}}=
\tilde x^{V\setminus\{v'\}}$ we have
$\rho^i_{(x_0,x^v)}=\tilde\rho^i_{(x_0,\tilde x^v)}$
if $v'\notin N(t)\cap K$, while both $\rho^i_{(x_0,x^v)}(A)$ and
$\tilde\rho^i_{(x_0,\tilde x^v)}(A)$ dominate
\[
\varepsilon^2 \frac{
\int\mathbf{1}_A(x^{t}_0)\prod_{w\in N(t)\cap(K\setminus\{v'\})}
p^w(x_0,x^w)
\nu^{t}_{x_0}(dx^{t}_0)
}{\int\prod_{w\in N(t)\cap(K\setminus\{v'\})}
p^w(x_0,x^w) \nu^{t}_{x_0}(dx^{t}_0)}.
\]
Therefore, by Lemma~\ref{lem:minorize}
\[
b_{(0,t)} \le%
\cases{ 0, &\quad  $\mbox{for $v'\notin
N(t)\cap K$}$,\vspace*{2pt}
\cr
2 \bigl(1-\varepsilon^{2} \bigr), &\quad $\mbox{otherwise}$.} %
\]

\textit{Case} $k=1$. In this case, we have
\begin{eqnarray*}
\rho^i_{(x_0,x^v)}(A) &= &\frac{
\int\mathbf{1}_A(x^v) g^v(x^v,Y_s^v) p^v(x_0,x^v)
\prod_{u\in N(v)\cap K'}p^u(x,z^u)
\psi^v(dx^v)
}{
\int g^v(x^v,Y_s^v) p^v(x_0,x^v)
\prod_{u\in N(v)\cap K'}p^u(x,z^u)
\psi^v(dx^v)
},
\\
\tilde\rho^i_{(x_0,\tilde x^v)}(A) &=& \frac{
\int\mathbf{1}_A(\tilde x^v) g^v(\tilde x^v,Y_s^v)
p^v(x_0,\tilde x^v)
\prod_{u\in N(v)\cap K'}p^u(\tilde x,z^u)
\psi^v(d\tilde x^v)
}{
\int g^v(\tilde x^v,Y_s^v)
p^v(x_0,\tilde x^v)
\prod_{u\in N(v)\cap K'}p^u(\tilde x,z^u)
\psi^v(d\tilde x^v)
}.
\end{eqnarray*}
Estimating as above, we obtain $C_{ij}\le
1-\varepsilon^2$ whenever $k'=0$ and $t'\in N(v)$, and $C_{ij}=0$
otherwise. Similarly, arguing again as above, we obtain
\[
b_{(1,v)} \le%
\cases{ 0, &\quad $\mbox{for $v'\notin
\displaystyle\bigcup_{w\in N(v)\cap K'}N(w)$}$,\vspace*{2pt}
\cr
2 \bigl(1-
\varepsilon^{2\Delta} \bigr), & \quad$\mbox{otherwise}$.} %
\]

Define the matrix $\{C_{ij}(v)\}_{i,j\in I}$ with the following entries:
\begin{eqnarray*}
C_{(0,t)(0,t')}(v) &=& \tilde C^\nu_{tt'},
\\
C_{(0,t)(1,v)}(v) &=& C_{(1,v)(0,t)}(v) = \bigl(1-\varepsilon^2
\bigr)\mathbf{1}_{t\in N(v)},
\\
C_{(1,v)(1,v)}(v) &=& 0.
\end{eqnarray*}
Combining the above two cases yields $C_{ij}\le C_{ij}(v)$,
and we readily compute
\[
\sum_{(k',t')\in I}e^{\beta\{|k-k'|+d(t,t')\}} C_{(k,t)(k',t')}(v)
\le\tcorr(\nu,\beta) + \bigl(1-\varepsilon^2 \bigr)e^{\beta(r+1)}
\Delta\le\frac{1}{2}
\]
where we have used the assumption of the proposition.
By Theorem~\ref{thmm:dobrushin}
\begin{eqnarray*}
\bigl\|(\mathsf{\tilde F}_s\nu)^{v,K'}_{x,z}- (
\mathsf{\tilde F}_s\nu)^{v,K'}_{\tilde x,z} \bigr\| &=& \|\rho-
\tilde\rho\|_{(1,v)}
\\
& \le&2 \bigl(1-\varepsilon^2 \bigr) \mathbf{1}_{v'\in K} \sum
_{t'\in N(v')} D_{(1,v)(0,t')}(v)
\\
&&{} + 2 \bigl(1-\varepsilon^{2\Delta} \bigr) \mathbf{1}_{v'\in\bigcup
_{w\in
N(v)\cap K'}N(w)}
D_{(1,v)(1,v)}(v),
\end{eqnarray*}
where $D(v):=\sum_{n\ge0}C(v)^n$. But note that the right-hand
side does not depend on $K'$ or $z,x,\tilde x$ (provided
$x^{V\setminus\{v'\}}=
\tilde x^{V\setminus\{v'\}}$).
We therefore obtain
\begin{eqnarray*}
\tilde C_{vv'}^{\mathsf{\tilde F}_s\nu} &\le& \bigl(1-\varepsilon^2
\bigr) \mathbf{1}_{v'\in K} \sum_{t'\in N(v')}
D_{(1,v)(0,t')}(v)
\\
&&{} + \bigl(1-\varepsilon^{2\Delta} \bigr) \mathbf{1}_{v'\in\bigcup
_{w\in
N(v)\cap K'}N(w)}
D_{(1,v)(1,v)}(v)
\end{eqnarray*}
for every $K\in\mathcal{K}$, $v\in K$ and $v'\in V$.
In particular, we have
\begin{eqnarray*}
\sum_{v'\in V}e^{\beta d(v,v')} \tilde
C_{vv'}^{\mathsf{\tilde F}_s\nu} &\le& \bigl(1-\varepsilon^2 \bigr)
\sum_{v'\in K}e^{\beta d(v,v')} \sum
_{t'\in N(v')} D_{(1,v)(0,t')}(v)
\\
&&{} + \bigl(1-\varepsilon^{2\Delta} \bigr) D_{(1,v)(1,v)}(v) \sum
_{v'\in\bigcup_{w\in
N(v)\cap K'}N(w)}e^{\beta d(v,v')}.
\end{eqnarray*}
To proceed, we note that
\[
\sum_{v'\in K}e^{\beta d(v,v')} \sum
_{t'\in N(v')} D_{(1,v)(0,t')}(v) \le e^{\beta r}\Delta\sum
_{v'\in V}e^{\beta d(v,v')}D_{(1,v)(0,v')}(v),
\]
where we have used that $d(v,v')\le d(v,t')+r$ for $t'\in N(v')$.
Similarly, we have
\[
\sum_{v'\in\bigcup_{w\in
N(v)\cap K'}N(w)}e^{\beta d(v,v')} \le e^{2\beta r}
\Delta^2.
\]
We can therefore estimate
\[
\sum_{v'\in V}e^{\beta d(v,v')} \tilde
C_{vv'}^{\mathsf{\tilde F}_s\nu} \le \bigl(1-\varepsilon^{2\Delta}
\bigr)e^{2\beta r}\Delta^2 \sum_{(k',v')\in I}e^{\beta\{
|1-k'|+d(v,v')\}}
D_{(1,v)(k',v')}(v).
\]
Applying Lemma~\ref{lem:expdecay} to $C(v)$ yields the result.
\end{pf}

We now iterate the above result.

%
\begin{cor}
\label{cor:tcorr}
Suppose there exists $\varepsilon>0$ such that
\[
\varepsilon\le p^v \bigl(x,z^v \bigr)\le
\varepsilon^{-1}\qquad \mbox{for all }v\in V, x,z\in\bbX,
\]
and such that
\[
\varepsilon> \varepsilon_0 = \biggl(1-\frac{1}{16\Delta^2}
\biggr)^{1/2\Delta}.
\]
Let $\mu$ be a probability measure on $\bbX$ such that
\[
\tcorr(\mu,\beta) \le\tfrac{1}{8},
\]
where $\beta= -(2r)^{-1}\log16\Delta^2(1-\varepsilon^{2\Delta})>0$.
Then
\[
\tcorr \bigl(\tilde\pi_n^\mu,\beta \bigr) \le
\tfrac{1}{8} \qquad\mbox{for all }n\ge0.
\]
In particular, the latter holds whenever $\mu=\delta_x$ for any
$x\in\bbX$.
\end{cor}

\begin{pf}
The assumption $\varepsilon>\varepsilon_0$ implies $\beta>0$ and
\[
\bigl(1-\varepsilon^2 \bigr)e^{\beta(r+1)}\Delta\le
\tfrac{1}{16}.
\]
Therefore, if $\tcorr(\nu,\beta)\le1/8$, then
Proposition~\ref{prop:tcorr} yields
\[
\tcorr(\mathsf{\tilde F}_s\nu,\beta) \le2 \bigl(1-
\varepsilon^{2\Delta} \bigr)e^{2\beta r}\Delta^2 \le
\tfrac{1}{8}.
\]
Thus, if $\tcorr(\mu,\beta) \le1/8$, then
$\tcorr(\tilde\pi_n^\mu,\beta)\le1/8$ for all $n\ge0$.
Moreover, as $\tcorr(\delta_x,\beta)=0$, the result holds automatically
for $\mu=\delta_x$.
\end{pf}

We finally obtain the requisite bound on $\corr(\tilde\pi_n^\mu
,\beta)$
using Lemma~\ref{lem:corrcompare}.

%
\begin{cor}[(Decay of correlations)]
\label{cor:decay}
Suppose there exists $\varepsilon>0$ with
\[
\varepsilon\le p^v \bigl(x,z^v \bigr)\le
\varepsilon^{-1} \qquad\mbox{for all }v\in V, x,z\in\bbX,
\]
such that
\[
\varepsilon> \varepsilon_0 = \biggl(1-\frac{1}{16\Delta^2}
\biggr)^{1/2\Delta}.
\]
Let $\beta= -(2r)^{-1}\log16\Delta^2(1-\varepsilon^{2\Delta})>0$.
Then
\[
\corr \bigl(\tilde\pi_n^x,\beta \bigr) \le
\tfrac{1}{3}
\]
for every $n\ge0$ and $x\in\bbX$.
\end{cor}

\begin{pf}
By Corollary~\ref{cor:tcorr} and Lemma~\ref{lem:corrcompare}, we can
estimate
\[
\corr \bigl(\tilde\pi_n^x,\beta \bigr)\le
\tfrac{1}{16} + \tfrac{1}{4} \varepsilon^{-2\Delta} \le
\tfrac{1}{3},
\]
where we used that $\varepsilon^{2\Delta}\ge1-1/16$.
\end{pf}

\subsection{Bounding the bias}
\label{sec:thmbias}

In the previous sections, we have proved a local filter stability bound
(Proposition~\ref{prop:lfstab}), a local one-step error bound
(Propositions \ref{prop:1steps} and \ref{prop:1stepn}), and decay of
correlations of the block filter (Corollary~\ref{cor:decay}). We can now
combine these results to obtain a time-uniform error bound between the
filter and the block filter; this controls the bias of the block particle
filtering algorithm.

%
\begin{thmm}[(Bias term)]
\label{thmm:bias}
Suppose there exists $\varepsilon>0$ such that
\[
\varepsilon\le p^v \bigl(x,z^v \bigr)\le
\varepsilon^{-1} \qquad\mbox{for all }v\in V, x,z\in\bbX,
\]
and such that
\[
\varepsilon> \varepsilon_0 = \biggl(1-\frac{1}{18\Delta^2}
\biggr)^{1/2\Delta}.
\]
Let $\beta= -(2r)^{-1}\log18\Delta^2(1-\varepsilon^{2\Delta})>0$.
Then
\[
\bigl\|\pi_n^x - \tilde\pi_n^x
\bigr\|_J \le\frac{8e^{-\beta}}{1-e^{-\beta}} \bigl(1-\varepsilon ^{2\Delta} \bigr)
\card J e^{-\beta d(J,\partial K)}
\]
for every $n\ge0$, $x\in\bbX$, $K\in\mathcal{K}$ and $J\subseteq K$.
\end{thmm}

\begin{pf}
We begin with the elementary error decomposition
\[
\bigl\|\pi_n^x - \tilde\pi_n^x
\bigr\|_J \le\sum_{s=1}^n \bigl\|
\mathsf{F}_n\cdots\mathsf{F}_{s+1}\mathsf{F}_s
\tilde\pi_{s-1}^x- \mathsf{F}_n\cdots
\mathsf{F}_{s+1}\mathsf{\tilde F}_s \tilde
\pi_{s-1}^x\bigr\|_J.
\]
We will bound each term in the sum.

\textit{Case} $s=n$. To bound this term, note that
\[
\corr \bigl(\tilde\pi_{n-1}^x,\beta \bigr) + \bigl(1-
\varepsilon^2 \bigr)e^{\beta(r+1)}\Delta\le\tfrac{1}{3} +
\tfrac{1}{18} \le\tfrac{1}{2}
\]
by Corollary~\ref{cor:decay}. Therefore, applying
Proposition~\ref{prop:1stepn} with $\nu=\tilde\pi_{n-1}^x$, we obtain
\[
\bigl\|\mathsf{F}_n\tilde\pi_{n-1}^x -\mathsf{
\tilde F}_n\tilde\pi_{n-1}^x\bigr\|_{J}
\le4e^{-\beta} \bigl(1-\varepsilon^{2\Delta} \bigr) e^{-\beta
d(J,\partial K)}
\card J.
\]

\textit{Case} $s<n$. To bound this term, note that by
Corollary~\ref{cor:decay}
\[
\corr \bigl(\tilde\pi_s^x,\beta \bigr) + 3 \bigl(1-
\varepsilon^{2\Delta} \bigr)e^{2\beta r}\Delta^2 \le
\tfrac{1}{3} + \tfrac{1}{6} = \tfrac{1}{2}.
\]
Applying Proposition~\ref{prop:lfstab} with $\mu=\tilde\pi_s^x$
and $\nu=\mathsf{F}_s\tilde\pi_{s-1}^x$ yields
\begin{eqnarray*}
&&\bigl \|\mathsf{F}_n\cdots\mathsf{F}_{s+1}
\mathsf{F}_s\tilde\pi_{s-1}^x-
\mathsf{F}_n\cdots\mathsf{F}_{s+1}\mathsf{\tilde
F}_s\tilde\pi_{s-1}^x\bigr\|_J
\\
&&\qquad \le2e^{-\beta(n-s)} \sum_{v\in J}\max
_{v'\in V} e^{-\beta d(v,v')} \sup_{x,z\in\bbX}\bigl\| \bigl(
\mathsf{F}_s\tilde\pi_{s-1}^x
\bigr)^{v'}_{x,z} - \bigl(\mathsf{\tilde F}_s
\tilde\pi_{s-1}^x \bigr)^{v'}_{x,z}\bigr\|.
\end{eqnarray*}
On the other hand, as by Corollary~\ref{cor:decay}
\[
\corr \bigl(\tilde\pi_{s-1}^x,\beta \bigr) + \bigl(1-
\varepsilon^2 \bigr)e^{\beta(r+1)}\Delta\le\tfrac{1}{3} +
\tfrac{1}{18} \le\tfrac{1}{2},
\]
we have by Proposition~\ref{prop:1steps} with $\nu=\tilde\pi_{s-1}^x$
\[
\sup_{x,z\in\bbX} \bigl\| \bigl(\mathsf{F}_s\tilde
\pi_{s-1}^x \bigr)^{v'}_{x,z} - \bigl(
\mathsf{\tilde F}_s\tilde\pi_{s-1}^x
\bigr)^{v'}_{x,z}\bigr\| \le4e^{-\beta} \bigl(1-
\varepsilon^{2\Delta} \bigr) e^{-\beta d(v',\partial K)}.
\]
We therefore obtain the estimate
\begin{eqnarray*}
&&\bigl \|\mathsf{F}_n\cdots\mathsf{F}_{s+1}
\mathsf{F}_s\tilde\pi_{s-1}^x-
\mathsf{F}_n\cdots\mathsf{F}_{s+1}\mathsf{\tilde
F}_s\tilde\pi_{s-1}^x\bigr\|_J
\\
&&\qquad \le8e^{-\beta} \bigl(1-\varepsilon^{2\Delta} \bigr)
e^{-\beta(n-s)} e^{-\beta d(J,\partial K)} \card J,
\end{eqnarray*}
where we have used
$d(v,v') + d(v',\partial K) \ge d(v,\partial K)$.

Substituting the above two cases into the error decomposition and
summing the geometric series yields the statement of the theorem.
\end{pf}

\subsection{Local stability of the block filter}
\label{sec:blockfstab}

As was explained in Section~\ref{sec:variance}, the chief difficulty in
obtaining a time-uniform bound on the variance term is to establish
stability of the block filter. This will be done in the present section.

We first establish a stability bound for nonrandom initial conditions.

%
\begin{prop}
\label{prop:preblockfstab}
Suppose there exists $\varepsilon>0$ such that
\[
\varepsilon\le p^v \bigl(x,z^v \bigr)\le
\varepsilon^{-1}\qquad \mbox{for all }v\in V, x,z\in\bbX,
\]
and such that
\[
\varepsilon>\varepsilon_0= \biggl(1-\frac{1}{6\Delta^2}
\biggr)^{1/2\Delta}.
\]
Let $\beta=-\log6\Delta^2(1-\varepsilon^{2\Delta})>0$. Then
\[
\bigl\|\mathsf{\tilde F}_n\cdots\mathsf{\tilde F}_{s+1}
\delta_z- \mathsf{\tilde F}_n\cdots\mathsf{\tilde
F}_{s+1}\delta_{z'}\bigr\|_J \le4\card J
e^{-\beta(n-s)}
\]
for every $s<n$, $z,z'\in\bbX$, $K\in\mathcal{K}$, and $J\subseteq K$.
\end{prop}

\begin{pf}
Fix throughout the proof $n>0$, $K\in\mathcal{K}$, and $J\subseteq K$.
We will also assume throughout the proof for notational simplicity that
$s=0$ (the ultimate conclusion will extend to any $s<n$ as in the proof of
Proposition~\ref{prop:lfstab}).

We begin by constructing the computation tree as explained in
Section~\ref{sec:variance}. For future reference, let us work first in
the more
general setting where the initial distributions
$\mu=\bigotimes_{K'\in\mathcal{K}}\mu^{K'}$ and
$\nu=\bigotimes_{K'\in\mathcal{K}}\nu^{K'}$ are independent across the
blocks (rather than the special case of point masses $\delta_x$ and
$\delta_{x'}$). Define for $K'\in\mathcal{K}$
\[
N \bigl(K' \bigr)= \bigl\{K''\in
\mathcal{K}\dvtx d \bigl(K',K'' \bigr)
\le r \bigr\},
\]
that is, $N(K')$ is the collection of blocks that interact with block $K'$
in one step of the dynamics [recall that $\card N(K')\le\Delta
_{\mathcal{K}}$].
Then we can evidently write
\[
\mathsf{B}^{K'}\mathsf{\tilde F}_s\mu=
\mathsf{C}_s^{K'}\mathsf{P}^{K'} \bigotimes
_{K''\in N(K')}\mu^{K''},
\]
where we have defined for any probability $\eta$ on $\bbX^{K'}$
\[
\bigl(\mathsf{C}_s^{K'}\eta \bigr) (A) :=
\frac{
\int\mathbf{1}_A(x^{K'})\prod_{v\in K'}g^v(x^v,Y_s^v) \eta(dx^{K'})
}{
\int\prod_{v\in K'}g^v(x^v,Y_s^v) \eta(dx^{K'})
},
\]
and for any probability $\eta$ on $\bbX^{\bigcup_{K''\in N(K')}K''}$
\[
\bigl(\mathsf{P}^{K'}\eta \bigr) (A) := \int\mathbf{1}_A
\bigl(x^{K'} \bigr)\prod_{v\in K'}p^v
\bigl(z,x^v \bigr) \psi^v \bigl(dx^v \bigr)
\eta(dz).
\]
We therefore have
\begin{eqnarray*}
&&\mathsf{B}^K\mathsf{\tilde F}_n\cdots\mathsf{\tilde
F}_1\mu\\
&&\qquad=
\mathsf{C}_n^K\mathsf{P}^K \bigotimes
_{K_{n-1}\in N(K)} \biggl[ \mathsf{C}_{n-1}^{K_{n-1}}
\mathsf{P}^{K_{n-1}} \\
&&\hspace*{112pt}{}\bigotimes_{K_{n-2}\in N(K_{n-1})} \biggl[
\mathsf{C}_{n-2}^{K_{n-2}}\mathsf{P}^{K_{n-2}} \cdots\\
&&\hspace*{177pt}
\bigotimes_{K_1\in N(K_2)} \biggl[ \mathsf{C}_1^{K_1}
\mathsf{P}^{K_1} \bigotimes_{K_0\in N(K_1)}
\mu^{K_0} \biggr] \cdots \biggr] \biggr].
\end{eqnarray*}
The structure of the computation tree is now readily visible in this
expression. To formalize the construction, we introduce the tree index
set
\[
T := \bigl\{[K_u\cdots K_{n-1}]\dvtx0\le u<n,
K_s\in N(K_{s+1})\mbox{ for }u\le s<n \bigr\}\cup \bigl
\{[ \varnothing] \bigr\},
\]
where we write $K_n:=K$ for simplicity (recall that $K$ and $n$ are
fixed throughout). The root of the tree $[\varnothing]$
represents the block $K$ at time $n$, while
$[K_u\cdots K_{n-1}]$ represents the duplicate of block $K_u$ at time $u$
that affects block $K$ at time $n$ along the branch
$K_u\to K_{u+1}\to\cdots\to K_{n-1}\to K$ (cf. Figure~\ref
{fig:comptree} for a simple illustration). The vertex set
corresponding to the computation tree is defined as
\[
I = \bigl\{[K_u\cdots K_{n-1}]v\dvtx[K_u
\cdots K_{n-1}]\in T, v\in K_u \bigr\} \cup \bigl\{[
\varnothing]v\dvtx v\in K \bigr\},
\]
and the corresponding state space is given by
\[
\bbS= \prod_{i\in I}\bbX^i,\qquad
\bbX^{[t]v}=\bbX^v \qquad\mbox{for } [t]v\in I.
\]
It will be convenient in the sequel to introduce some additional notation.
First, we will specify the children $c(i)$ of an index $i\in I$ as
follows:
\[
c \bigl([K_u\cdots K_{n-1}]v \bigr) := \bigl
\{[K_{u-1}\cdots K_{n-1}]v'\dvtx
K_{u-1}\in N(K_u), v'\in N(v) \bigr\},
\]
and similarly for $c([\varnothing]v)$.
Denote the depth $d(i)$ and location $v(i)$ of $i\in I$ as
\[
d \bigl([K_u\cdots K_{n-1}]v \bigr) := u, \qquad \bigl([
\varnothing]v \bigr) := n,\qquad v \bigl([t]v \bigr) := v.
\]
We define the index set of nonleaf vertices in $I$ as
\[
I_+ := \bigl\{i\in I\dvtx0<d(i)\le n \bigr\},
\]
and the set of leaves of the tree $T$ as
\[
T_0 := \bigl\{[K_0\cdots K_{n-1}]\dvtx
K_s\in N(K_{s+1})\mbox{ for } 0\le s<n \bigr\}.
\]
Finally, it will be natural to identify $[t]\in T$ with the
corresponding subset of $I$:
\[
[K_u\cdots K_{n-1}] = \bigl\{[K_u\cdots
K_{n-1}]v\dvtx v\in K_u \bigr\},
\]
together with the analogous identification for $[\varnothing]$.

We now define the probability measures $\rho,\tilde\rho$ on $\bbS$ as
follows:
\begin{eqnarray*}
&&\hspace*{-4pt}\rho(A) \\
&&\hspace*{-5pt}\qquad=
 \frac{
\int\mathbf{1}_A(x)
\prod_{i\in I_+}
p^{v(i)}(x^{c(i)},x^{i})
g^{v(i)}(x^{i},Y_{d(i)}^{v(i)})
\psi^{v(i)}(dx^{i})
\prod_{[t]\in T_0}\mu^{[t]}(dx^{[t]})
}{
\int
\prod_{i\in I_+}
p^{v(i)}(x^{c(i)},x^{i})
g^{v(i)}(x^{i},Y_{d(i)}^{v(i)})
\psi^{v(i)}(dx^{i})
\prod_{[t]\in T_0}\mu^{[t]}(dx^{[t]})
},
\\
&&\hspace*{-4pt}\tilde\rho(A)\\
&&\hspace*{-4pt}\qquad =\frac{
\int\mathbf{1}_A(x)
\prod_{i\in I_+}
p^{v(i)}(x^{c(i)},x^{i})
g^{v(i)}(x^{i},Y_{d(i)}^{v(i)})
\psi^{v(i)}(dx^{i})
\prod_{[t]\in T_0}\nu^{[t]}(dx^{[t]})
}{
\int
\prod_{i\in I_+}
p^{v(i)}(x^{c(i)},x^{i})
g^{v(i)}(x^{i},Y_{d(i)}^{v(i)})
\psi^{v(i)}(dx^{i})
\prod_{[t]\in T_0}\nu^{[t]}(dx^{[t]})
},
\end{eqnarray*}
where we write $\mu^{[K_0\cdots K_{n-1}]}:=\mu^{K_0}$ and $\nu
^{[K_0\cdots
K_{n-1}]}:=\nu^{K_0}$ for simplicity. Then, by construction, the measure
$\mathsf{B}^K\mathsf{\tilde F}_n\cdots\mathsf{\tilde F}_1\mu$ coincides
with the marginal of $\rho$ on the root of the computation tree, while
$\mathsf{B}^K\mathsf{\tilde F}_n\cdots\mathsf{\tilde F}_1\nu$ coincides
with the marginal of $\tilde\rho$ on the root of the computation tree.
In particular, we obtain
\[
\|\mathsf{\tilde F}_n\cdots\mathsf{\tilde F}_1\mu-
\mathsf{\tilde F}_n\cdots\mathsf{\tilde F}_1\nu
\|_J = \|\rho-\tilde\rho\|_{[\varnothing]J}.
\]
We will use Theorem~\ref{thmm:dobrushin} to obtain a bound on
this expression.

Throughout the remainder of the proof, we specialize to the case that
$\mu=\delta_z$ and $\nu=\delta_{z'}$. To apply Theorem~\ref
{thmm:dobrushin}, we must bound the quantities $C_{ij}$ and $b_i$ with
$i=[K_u\cdots K_{n-1}]v$ and $j=[K_{u'}'\cdots K_{n-1}']v'$. We
distinguish three cases.

\textit{Case} $u=0$. As $\mu=\delta_z$ is nonrandom we evidently
have $\rho^i_x = \delta_{z^v}$, so that $C_{ij}=0$. On the other hand,
as $\tilde\rho^i_x = \delta_{z^{\prime v}}$, we cannot do better than
$b_i\le2$.

\textit{Case} $0<u<n$. Now we have
\begin{eqnarray*}
\rho^i_x(A) &=& \tilde\rho^i_x(A)\\
&=&
\frac{
\int\mathbf{1}_A(x^i) g^v(x^i,Y_u^v)
p^v(x^{c(i)},x^i)\prod_{\ell\in I_+\dvtx i\in c(\ell)}p^{v(\ell)}(
x^{c(\ell)},x^\ell)
\psi^v(dx^i)
}{
\int g^v(x^i,Y_u^v)
p^v(x^{c(i)},x^i)\prod_{\ell\in I_+\dvtx i\in c(\ell)}p^{v(\ell)}(
x^{c(\ell)},x^\ell)
\psi^v(dx^i)
}.
\end{eqnarray*}
Thus, $b_i=0$. Moreover, by inspection, $\rho^i_x$ does not depend on
$x^j$ except in the following cases: $j\in c(i)$; $i\in c(j)$;
$j\in c(\ell)$ for some $\ell\in I_+$ such that $i\in c(\ell)$.
As $\card c(\ell)\le\Delta$ for every $\ell\in I_+$, we estimate
using Lemma~\ref{lem:minorize}
\[
C_{ij} \le%
\cases{ 1-\varepsilon^2, &\quad $\mbox{if
$j\in c(i)$}$,\vspace*{2pt}
\cr
1-\varepsilon^2, & \quad $\mbox{if $i\in
c(j)$}$,\vspace*{2pt}
\cr
1-\varepsilon^{2\Delta}, &\quad  $\mbox{if $\displaystyle j\in\bigcup
_{\ell\in I_+\dvtx i\in c(\ell)}c(\ell)$}$,\vspace*{2pt}
\cr
0, &\quad $
\mbox{otherwise}$.} %
\]
This yields
\[
\sum_{j\in I} e^{\beta|d(i)-d(j)|}C_{ij} \le2
\bigl(1-\varepsilon^2 \bigr)e^{\beta}\Delta+ \bigl(1-
\varepsilon^{2\Delta} \bigr)\Delta^2 \le3 \bigl(1-
\varepsilon^{2\Delta} \bigr)e^{\beta}\Delta^2,
\]
where we have used that $\beta>0$ and $\Delta\ge1$ in the last
inequality.

\textit{Case} $u=n$. Now $i=[\varnothing]v$, so we have
\[
\rho^i_x(A) = \tilde\rho^i_x(A)
= \frac{
\int\mathbf{1}_A(x^i) g^v(x^i,Y_n^v)
p^v(x^{c(i)},x^i)
\psi^v(dx^i)
}{
\int g^v(x^i,Y_n^v)
p^v(x^{c(i)},x^i)
\psi^v(dx^i)
}.
\]
Arguing precisely as above, we obtain $b_i=0$ and
\[
\sum_{j\in I} e^{\beta|d(i)-d(j)|}C_{ij} \le
\bigl(1-\varepsilon^2 \bigr)e^{\beta}\Delta.
\]

Combining the above three cases, we obtain
\[
\max_{i\in I}\sum_{j\in I}
e^{\beta|d(i)-d(j)|}C_{ij} \le3 \bigl(1-\varepsilon^{2\Delta}
\bigr)e^{\beta}\Delta^2=\frac{1}{2}
\]
by the assumption of the proposition. Thus, by Theorem~\ref{thmm:dobrushin}
\[
\|\mathsf{\tilde F}_n\cdots\mathsf{\tilde F}_1
\delta_z- \mathsf{\tilde F}_n\cdots\mathsf{\tilde
F}_1\delta_{z'}\|_J = \|\rho-\tilde\rho
\|_{[\varnothing]J} \le4\card J e^{-\beta n},
\]
where we have used Lemma~\ref{lem:expdecay} with
$m(i,j)=\beta|d(i)-d(j)|$. The proof is completed by extending
to general $s<n$ as in the proof of Proposition~\ref{prop:lfstab}.
\end{pf}

The proof of Proposition~\ref{prop:preblockfstab} was simplified by the
fact that the resulting bound holds uniformly for all point mass initial
conditions (this could be used to obtain a uniform bound for all initial
measures along the same lines as the proof of Corollary~\ref{cor:fstab}).
To obtain a bound on the variance term, however, we require a more precise
stability bound for the block filter that provides explicit control in
terms of the initial conditions. We will shortly deduce such a bound from
Proposition~\ref{prop:preblockfstab}. Before we can do so, however, we
must prove a refinement of Lemma~\ref{lem:weight}.

%
\begin{lem}
\label{lem:specialweight}
Let $\mu=\mu^1\otimes\cdots\otimes\mu^d$ and $\nu=\nu^1\otimes
\cdots
\otimes\nu^d$ be product probability measures on
$\bbS=\bbS^1\times\cdots\times\bbS^d$,
and let $\Lambda\dvtx\bbS\to\mathbb{R}$ be a bounded and strictly positive
measurable function. Define the probability measures
\[
\mu_\Lambda(A) := \frac{\int\mathbf{1}_A(x)\Lambda(x)\mu(dx)}{
\int\Lambda(x)\mu(dx)},\qquad \nu_\Lambda(A) :=
\frac{\int\mathbf{1}_A(x)\Lambda(x)\nu(dx)}{
\int\Lambda(x)\nu(dx)}.
\]
Suppose that there exists a constant $\varepsilon>0$ such that the
following holds: for every $i=1,\ldots,d$, there is a
measurable function $\Lambda^i\dvtx\bbS\to\mathbb{R}$ such that
\[
\varepsilon\Lambda^i(x)\le\Lambda(x)\le\varepsilon^{-1}
\Lambda^i(x) \qquad\mbox{for all }x\in\bbS
\]
and such that $\Lambda^i(x)=\Lambda^i(\tilde x)$ whenever
$x^{\{1,\ldots,d\}\setminus\{i\}}=
\tilde x^{\{1,\ldots,d\}\setminus\{i\}}$.
Then
\[
\|\mu_\Lambda-\nu_\Lambda\| \le\frac{2}{\varepsilon^2} \sum
_{i=1}^d\bigl\|\mu^i-\nu^i\bigr\|.
\]
\end{lem}

\begin{pf}
Define for $i=0,\ldots,d$ the measures
\[
\rho_i := \nu^1\otimes\cdots\otimes\nu^i
\otimes\mu^{i+1}\otimes\cdots\otimes\mu^d,\qquad
\rho_{i,\Lambda}(A) := \frac{\int\mathbf{1}_A(x)\Lambda(x)\rho_i(dx)}{
\int\Lambda(x)\rho_i(x)}
\]
(by convention, $\rho_0=\mu$ and $\rho_d=\nu$). Then we can estimate
\[
\|\mu_\Lambda-\nu_\Lambda\| \le\sum
_{i=1}^d\|\rho_{i,\Lambda}-\rho_{i-1,\Lambda}
\|.
\]
Now note that we can estimate for $|f|\le1$
\[
\bigl|\rho_{i,\Lambda}(f)-\rho_{i-1,\Lambda}(f)\bigr| \le\frac
{1}{\varepsilon\rho_i(\Lambda^i)} \bigl[\bigl|
\rho_i(f\Lambda)-\rho_{i-1}(f\Lambda)\bigr| + \bigl|\rho_i(
\Lambda)-\rho_{i-1}(\Lambda)\bigr| \bigr]
\]
as in the proof of Lemma~\ref{lem:weight}. Moreover, we can write
\begin{eqnarray*}
\bigl|\rho_i(f\Lambda)-\rho_{i-1}(f\Lambda)\bigr| &=&
\frac{\rho_i(\Lambda^i)}{\varepsilon} \biggl|\int f^i(x)\nu^i \bigl(dx^i
\bigr)- \int f^i(x)\mu^i \bigl(dx^i \bigr)\biggr |,
\\
\bigl|\rho_i(\Lambda)-\rho_{i-1}(\Lambda)\bigr| &= &\frac{\rho_i(\Lambda
^i)}{\varepsilon}
\biggl|\int g^i(x)\nu^i \bigl(dx^i \bigr)- \int
g^i(x)\mu^i \bigl(dx^i \bigr)\biggr |,
\end{eqnarray*}
where $f^i$ and $g^i$ are functions on $\bbS^i$ defined by
\begin{eqnarray*}
f^i \bigl(x^i \bigr) &:=& \frac{\varepsilon}{\rho_i(\Lambda^i)} \int f(x)
\Lambda(x) \nu^1 \bigl(dx^1 \bigr)\cdots
\nu^{i-1} \bigl(dx^{i-1} \bigr) \mu^{i+1}
\bigl(dx^{i+1} \bigr) \cdots\mu^d \bigl(dx^d
\bigr),
\\
g^i \bigl(x^i \bigr) &:= &\frac{\varepsilon}{\rho_i(\Lambda^i)} \int
\Lambda(x) \nu^1 \bigl(dx^1 \bigr)\cdots
\nu^{i-1} \bigl(dx^{i-1} \bigr) \mu^{i+1}
\bigl(dx^{i+1} \bigr) \cdots\mu^d \bigl(dx^d
\bigr).
\end{eqnarray*}
Evidently $|f^i|\le1$ and $|g^i|\le1$, and the proof follows directly.
\end{pf}

We can now obtain a stability bound with control on the initial
conditions.

%
\begin{prop}
\label{prop:0blockfstab}
Suppose there exists $\varepsilon>0$ with
\[
\varepsilon\le p^v \bigl(x,z^v \bigr)\le
\varepsilon^{-1} \qquad\mbox{for all }v\in V, x,z\in\bbX
\]
such that
\[
\varepsilon>\varepsilon_0= \biggl(1-\frac{1}{6\Delta^2}
\biggr)^{1/2\Delta}.
\]
Let $\beta=-\log6\Delta^2(1-\varepsilon^{2\Delta})>0$. Then for any
product probability measures
\[
\mu=\bigotimes_{K\in\mathcal{K}}\mu^K, \qquad \nu=
\bigotimes_{K\in\mathcal{K}}\nu^K,
\]
we have
\[
\|\mathsf{\tilde F}_n\cdots\mathsf{\tilde F}_{s+1}\mu-
\mathsf{\tilde F}_n\cdots\mathsf{\tilde F}_{s+1}\nu
\|_J \le\frac{4}{\varepsilon^{2|\mathcal{K}|_\infty}} \card J e^{-\beta(n-s)}\sum
_{K\in\mathcal{K}}\alpha_K \bigl\|\mu^K-
\nu^K\bigr\|
\]
for every $s<n$, $K\in\mathcal{K}$, and $J\subseteq K$. Here,
$(\alpha_K)_{K\in\mathcal{K}}$ are nonnegative integers, depending
on $J$
and $n-s$ only, such
that $\sum_{K\in\mathcal{K}}\alpha_K\le\Delta_\mathcal{K}^{n-s}$.
\end{prop}

\begin{pf}
We fix $s=0$, $n>0$, $K\in\mathcal{K}$, $J\subseteq K$ as in the
proof of Proposition~\ref{prop:preblockfstab}, and adopt the notation
used there. Define the functions
\begin{eqnarray*}
h_A \bigl(x^{T_0} \bigr) & :=& \int\mathbf{1}_A
\bigl(x^{[\varnothing]J} \bigr) \prod_{i\in I_+}
p^{v(i)} \bigl(x^{c(i)},x^{i} \bigr)
g^{v(i)} \bigl(x^{i},Y_{d(i)}^{v(i)} \bigr)
\psi^{v(i)} \bigl(dx^{i} \bigr),
\\
h \bigl(x^{T_0} \bigr) & := &\int\prod_{i\in I_+}
p^{v(i)} \bigl(x^{c(i)},x^{i} \bigr)
g^{v(i)} \bigl(x^{i},Y_{d(i)}^{v(i)} \bigr)
\psi^{v(i)} \bigl(dx^{i} \bigr)
\end{eqnarray*}
on the leaves $T_0$ of the computation tree,
for every measurable $A\subseteq\bbX^J$. Then
\[
(\mathsf{\tilde F}_n\cdots\mathsf{\tilde F}_{1}\mu) (A)
= \frac{
\int h_A(x^{T_0})\prod_{[t]\in T_0}\mu^{[t]}(dx^{[t]})
}{
\int h(x^{T_0})\prod_{[t]\in T_0}\mu^{[t]}(dx^{[t]})
} = \int\frac{h_A(x^{T_0})}{h(x^{T_0})} \tilde\mu \bigl(dx^{T_0}
\bigr),
\]
where we define the measure
\[
\tilde\mu(A) := \frac{\int\mathbf{1}_A(x^{T_0}) h(x^{T_0})
\prod_{[t]\in T_0}\mu^{[t]}(dx^{[t]})
}{
\int h(x^{T_0})
\prod_{[t]\in T_0}\mu^{[t]}(dx^{[t]})}.
\]
The measure $\tilde\nu$ is defined analogously, and we have
\[
\|\mathsf{\tilde F}_n\cdots\mathsf{\tilde F}_{1}\mu-
\mathsf{\tilde F}_n\cdots\mathsf{\tilde F}_{1}\nu
\|_J = 2\sup_{A\subseteq\bbX^J} \biggl|\int\frac{h_A}{h} \,d
\tilde \mu- \int\frac{h_A}{h} \,d\tilde\nu\biggr|,
\]
where the supremum is taken only over measurable sets. But
note that $h_A/h$ is precisely the filter obtained when the initial
condition is a point mass on the leaves of the computation tree
(albeit not with the special duplication pattern induced by the
unravelling of the original model; however, this was not used in
the proof of Proposition~\ref{prop:preblockfstab}). Therefore,
the proof of Proposition~\ref{prop:preblockfstab} yields
\[
2\sup_{z,\tilde z\in\bbX^{T_0}}\sup_{A\subseteq\bbX^J}\biggl |
\frac{h_A(z)}{h(z)}- \frac{h_A(\tilde z)}{h(\tilde z)} \biggr|\le4\card J e^{-\beta n}.
\]
In particular, using the identity $|\mu(f)-\nu(f)|\le\frac
{1}{2}\osc
f
\|\mu-\nu\|$, we obtain
\[
\|\mathsf{\tilde F}_n\cdots\mathsf{\tilde F}_{1}\mu-
\mathsf{\tilde F}_n\cdots\mathsf{\tilde F}_{1}\nu
\|_J \le2\card J e^{-\beta n} \|\tilde\mu-\tilde\nu\|.
\]
We now aim to apply Lemma~\ref{lem:specialweight} to estimate
$\|\tilde\mu-\tilde\nu\|$.

To this end, consider a block $[t]\in T_0$. The integrand in the
definition of $h(x^{T_0})$ depends only on $x^{[t]}$ through the
terms $p^{v(i)}(x^{c(i)},x^i)$ with $c(i)\cap[t]\ne\varnothing$.
If we write $[t]=[K_0\cdots K_{n-1}]$, then $c(i)\cap[t]\ne
\varnothing$
requires at least $i\in[K_1\cdots K_{n-1}]$ and therefore $\card\{
i\in
I_+\dvtx c(i)\cap[t]\ne\varnothing\}\le\card K_1\le|\mathcal
{K}|_\infty$.
Thus, we have
\[
\varepsilon^{|\mathcal{K}|_\infty}h^{[t]}(z) \le h(z) \le
\varepsilon^{-|\mathcal{K}|_\infty}h^{[t]}(z)
\]
for every $z\in\bbX^{T_0}$ and $[t]\in T_0$, where
\[
h^{[t]} \bigl(x^{T_0} \bigr) := \int\prod
_{i\in I_+: c(i)\cap[t]=\varnothing} p^{v(i)} \bigl(x^{c(i)},x^{i}
\bigr) \prod_{i\in I_+} g^{v(i)}
\bigl(x^{i},Y_{d(i)}^{v(i)} \bigr)
\psi^{v(i)} \bigl(dx^{i} \bigr)
\]
does not depend on $x^{[t]}$. By Lemma~\ref{lem:specialweight}, we
obtain
\[
\|\tilde\mu-\tilde\nu\| \le\frac{2}{\varepsilon^{2|\mathcal
{K}|_\infty}} \sum_{[t]\in T_0}
\bigl\|\mu^{[t]}-\nu^{[t]}\bigr\| = \frac{2}{\varepsilon^{2|\mathcal{K}|_\infty}} \sum
_{K'\in\mathcal
{K}}\alpha_{K'} \bigl\|\mu^{K'}-
\nu^{K'}\bigr\|,
\]
where we define $\alpha_{K'}=\card\{[K_0\cdots K_{n-1}]\in T_0\dvtx
K_0=K'\}$.
As the computation tree has a branching factor of
at most $\Delta_\mathcal{K}$, we evidently have $\sum_{K\in\mathcal{K}}
\alpha_K=\card T_0\le\Delta_\mathcal{K}^n$. The result therefore follows
directly
for the case $s=0$, and the general case $s<n$ is immediate
as in the proof of Proposition~\ref{prop:lfstab}.
\end{pf}

We finally state the block filter stability bound in its
most useful form.

%
\begin{cor}[(Block filter stability)]
\label{cor:blockfstab}
Suppose there exists $\varepsilon>0$ with
\[
\varepsilon\le p^v \bigl(x,z^v \bigr)\le
\varepsilon^{-1} \qquad\mbox{for all }v\in V, x,z\in\bbX
\]
such that
\[
\varepsilon>\varepsilon_0= \biggl(1-\frac{1}{6\Delta_\mathcal
{K}\Delta^2}
\biggr)^{1/2\Delta}.
\]
Let $\beta=-\log6\Delta_\mathcal{K}\Delta^2(1-\varepsilon
^{2\Delta})>0$.

Then for any
(possibly random) product probability measures
\[
\mu=\bigotimes_{K\in\mathcal{K}}\mu^K, \qquad\nu=
\bigotimes_{K\in\mathcal{K}}\nu^K,
\]
we have
\begin{eqnarray*}
&&\mathbf{E} \bigl[ \bigl\|\mathsf{\tilde F}_n\cdots\mathsf{\tilde
F}_{s+1}\mu- \mathsf{\tilde F}_n\cdots\mathsf{\tilde
F}_{s+1}\nu\bigr\|_J^2 \bigr]^{1/2}
\\
&&\qquad \le\frac{4}{\varepsilon^{2|\mathcal{K}|_\infty}} \card J e^{-\beta(n-s)} \max_{K\in\mathcal{K}}
\mathbf{E} \bigl[\bigl\|\mu^K-\nu^K\bigr\|^2
\bigr]^{1/2}
\end{eqnarray*}
for every $s<n$, $K\in\mathcal{K}$, and $J\subseteq K$.
\end{cor}

\begin{pf}
The result follows readily from Proposition~\ref{prop:0blockfstab}
(note that we have now absorbed the branching factor
$\Delta_\mathcal{K}^{n-s}$ in the definition of $\beta$).
\end{pf}

\subsection{Bounding the variance}
\label{sec:thmvariance}

To complete the proof of Theorem~\ref{thmm:main}, it now remains to bound
the variance term ${|\!|\!| \tilde\pi_n-\hat\pi_n |\!|\!|}_J$
uniformly in time.
This is the goal of the present section. We will first obtain bounds on
the one-step error, and then combine these with the block filter stability
bound of Corollary~\ref{cor:blockfstab} to obtain time-uniform control
of the error. The main remaining difficulty is to properly account for
the fact that Corollary~\ref{cor:blockfstab} is phrased in terms of the
total variation norm $\|\cdot\|_J$, which is too strong to control
the sampling error (we do not know how to prove an analogous result to
Corollary~\ref{cor:blockfstab} in the weaker ${|\!|\!| \cdot |\!|\!
|}_J$-norm).
To this end, we retain one time step of the block filter dynamics in the
one-step error (we control
$\|\mathsf{\tilde F}_{s+1}\mathsf{\tilde F}_s\hat\pi_{s-1}^x-
\mathsf{\tilde F}_{s+1}\mathsf{\hat F}_s\hat\pi_{s-1}^x\|_K$
rather than
${|\!|\!| \mathsf{\tilde F}_s\hat\pi_{s-1}^x- \mathsf{\hat F}_s\hat
\pi_{s-1}^x |\!|\!|}_K$), which allows us to
exploit the fact that the dynamics $\mathsf{P}$ has a density.

Let us begin with the most trivial result: a one-step bound in the
${|\!|\!| \cdot |\!|\!|}_J$-norm. This estimate will be used to bound
the error in
the last time step $s=n$.

%
\begin{lem}[(Sampling error, $s=n$)]
\label{lem:var1stepn}
Suppose there exists $\kappa>0$ such that
\[
\kappa\le g^v \bigl(x^v,y^v \bigr)\le
\kappa^{-1} \qquad\mbox{for all }v\in V, x\in\bbX, y\in\bbY.
\]
Then
\[
\max_{K\in\mathcal{K}} {\bigl|\!\bigl|\!\bigl| \mathsf{\tilde F}_n\hat \pi
_{n-1}^\mu- \mathsf{\hat F}_n\hat
\pi_{n-1}^\mu \bigr|\!\bigr|\!\bigr|}_K \le
\frac{2\kappa^{-2|\mathcal{K}|_\infty}}{\sqrt{N}}.
\]
\end{lem}

\begin{pf}
Note that
\begin{eqnarray*}
{\bigl|\!\bigl|\!\bigl| \mathsf{\tilde F}_n\hat\pi_{n-1}^\mu-
\mathsf{\hat F}_n\hat\pi_{n-1}^\mu \bigr|\!\bigr|\!\bigr|}_K &=& {\bigl|\!\bigl|\!\bigl| \mathsf{C}_n^K
\mathsf{B}^K \mathsf{P} \hat\pi _{n-1}^\mu-
\mathsf{C}_n^K \mathsf{B}^K
\mathsf{S}^N\mathsf{P} \hat \pi_{n-1}^\mu \bigr|\!\bigr|\!\bigr|}
\\
&\le& 2\kappa^{-2\card K}{\bigl|\!\bigl|\!\bigl| \mathsf{P}\hat\pi_{n-1}^\mu-
\mathsf{S}^N\mathsf{P}\hat\pi_{n-1}^\mu \bigr|\!\bigr|\!\bigr|} \le \frac{2\kappa^{-2\card K}}{\sqrt{N}},
\end{eqnarray*}
where the first inequality is Lemma~\ref{lem:weight} and the
second inequality follows from the simple estimate
${\bigl|\!\bigl|\!\bigl| \mu-\mathsf{S}^N\mu |\!|\!|}\le1/\sqrt{N}$ that holds
for any
probability $\mu$.
\end{pf}

For the error in steps $s<n$, the requisite one-step bound
(Proposition~\ref{prop:var1steps}) is more involved.
Before we prove it, we must first introduce an elementary lemma about
products of empirical measures that will be needed below.

%
\begin{lem}
\label{lem:vstatistic}
For any probability measure $\mu$, we have
\[
{\bigl|\!\bigl|\!\bigl| \mu^{\otimes d}-\hat\mu^{\otimes d} \bigr|\!\bigr|\!\bigr|} \le
\frac
{4d}{\sqrt{N}},
\]
where $\hat\mu=\frac{1}{N}\sum_{k=1}^N\delta_{X_k}$
and $X_1,\ldots,X_N$ are i.i.d. $\sim\mu$.
\end{lem}

\begin{pf}
We assume throughout that $N\ge d^2$ without loss of generality (otherwise
the bound is trivial). Let $|f|\le1$ be a measurable function. Then
\[
\hat\mu^{\otimes d}(f)= \frac{1}{N^d}\sum_{k_1,\ldots,k_d=1}^N
f(X_{k_1},\ldots,X_{k_d}).
\]
We begin by bounding
\[
\mathrm{Var} \bigl[\hat\mu^{\otimes d}(f) \bigr] = \frac{1}{N^{2d}} \sum
_{k_1,\ldots,k_d=1}^N \sum
_{k_1',\ldots,k_d'=1}^N \mathbf{E}[F_{k_1,\ldots,k_d}F_{k_1',\ldots,k_d'}],
\]
where
\[
F_{k_1,\ldots,k_d}:= f(X_{k_1},\ldots,X_{k_d})- \mathbf{E}
\bigl[f(X_{k_1},\ldots,X_{k_d}) \bigr].
\]
Note that $\mathbf{E}[F_{k_1,\ldots,k_d}F_{k_1',\ldots,k_d'}]=0$
when $\{k_1,\ldots,k_d\}\cap\{k_1',\ldots,k_d'\}=\varnothing$. Thus
\[
\mathrm{Var} \bigl[\hat\mu^{\otimes d}(f) \bigr] \le\frac
{4}{N^{2d}} \sum
_{k_1,\ldots,k_d=1}^N \sum
_{k_1',\ldots,k_d'=1}^N \mathbf{1}_{\{k_1,\ldots,k_d\}\cap\{
k_1',\ldots,k_d'\}\ne\varnothing},
\]
where we use $|F_{k_1,\ldots,k_d}|\le2$. But for each choice of
$k_1,\ldots,k_d$, there are at least $(N-d)^d$ choices of
$k_1',\ldots,k_d'$ such that
$\{k_1,\ldots,k_d\}\cap\{k_1',\ldots,k_d'\}=\varnothing$, so
\[
\mathrm{Var} \bigl[\hat\mu^{\otimes d}(f) \bigr] \le4 \biggl(1-
\frac{N^d(N-d)^d}{N^{2d}} \biggr) = 4 \biggl(1- \biggl(1-\frac{d}{N}
\biggr)^d \biggr) \le\frac{4d^2}{N}.
\]
We can therefore estimate
\begin{eqnarray*}
{\bigl|\!\bigl|\!\bigl| \mu^{\otimes d}-\hat\mu^{\otimes d} \bigr|\!\bigr|\!\bigr|} &\le&\bigl\|\mu
^{\otimes d}- \mathbf{E} \bigl[\hat\mu^{\otimes d} \bigr]\bigr\|+ {\bigl|\!\bigl|\!\bigl|
\mathbf {E} \bigl[\hat \mu^{\otimes d} \bigr]-\hat\mu^{\otimes d} \bigr|\!\bigr|\!\bigr|}
\\
&\le&\bigl\|\mu^{\otimes d}-\mathbf{E} \bigl[\hat\mu^{\otimes d} \bigr]\bigr\|+
\frac{2d}{\sqrt{N}}.
\end{eqnarray*}
It remains to estimate the first term. To this end, note that
$\mathbf{E}[f(X_{k_1},\ldots,X_{k_d})]=\mu^{\otimes d}(f)$ whenever
$k_1\ne\cdots\ne k_d$. Therefore, we evidently have
\begin{eqnarray*}
\bigl|\mathbf{E} \bigl[\hat\mu^{\otimes d}(f) \bigr]-\mu^{\otimes d}(f)\bigr| &\le&
\frac{1}{N^d}\sum_{k_1,\ldots,k_d=1}^N\bigl |\mathbf{E}
\bigl[f(X_{k_1},\ldots,X_{k_d}) \bigr]-\mu^{\otimes d}(f)\bigr|
\\
&\le&2 \biggl(1-\frac{1}{N^d}\frac{N!}{(N-d)!} \biggr) \le2 \biggl(1-
\biggl(1-\frac{d}{N} \biggr)^d \biggr) \le\frac{2d^2}{N}.
\end{eqnarray*}
But as $N\ge d^2$, we have $d^2/N\le d/\sqrt{N}$. The result follows.
\end{pf}

This result will be used in the following form.

%
\begin{cor}
\label{cor:blocksampling}
For any subset of blocks $\mathcal{L}\subseteq\mathcal{K}$, we have
\[
{\biggl|\!\biggl|\!\biggl| \bigotimes_{K\in\mathcal{L}}
 \mathsf{B}^{K}\mu-
\bigotimes_{K\in\mathcal{L}} \mathsf{B}^{K}
\mathsf{S}^N\mu \biggr|\!\biggr|\!\biggr|} \le\frac{4\card\mathcal{L}}{\sqrt{N}}
\]
for every probability measure $\mu$ on $\bbX$ and $s\ge1$.
\end{cor}

\begin{pf}
Write $\hat\mu:=\mathsf{S}^N\mu$ and $d=\card\mathcal{L}$, and
let us
enumerate the blocks
$\mathcal{L}=\{K_1,\ldots,K_d\}$. Then for any bounded function
$f\dvtx\bbX^{\cup\mathcal{L}}\to\mathbb{R}$, we can write
\begin{eqnarray*}
\biggl( \bigotimes_{K\in\mathcal{L}}\mathsf{B}^K\mu \biggr)
(f) &=& \int f \bigl(x^{K_1}_1,\ldots,x^{K_d}_d
\bigr) \mu(dx_1)\cdots\mu(dx_d),
\\
\biggl( \bigotimes_{K\in\mathcal{L}}\mathsf{B}^K
\mathsf{S}^N\mu \biggr) (f) &=& \int f \bigl(x^{K_1}_1,
\ldots,x^{K_d}_d \bigr) \hat\mu(dx_1)\cdots
\hat \mu(dx_d).
\end{eqnarray*}
Thus, evidently
\[
{\biggl|\!\biggl|\!\biggl| \bigotimes_{K\in\mathcal{L}} \mathsf{B}^{K}\mu-
\bigotimes_{K\in\mathcal{L}} \mathsf{B}^{K}
\mathsf{S}^N\mu \biggr|\!\biggr|\!\biggr|} \le{\bigl|\!\bigl|\!\bigl| \mu^{\otimes d}-\hat
\mu^{\otimes d} \bigr|\!\bigr|\!\bigr|},
\]
and the result follows from Lemma~\ref{lem:vstatistic}.
\end{pf}

We now proceed to prove a one-step error bound for time steps $s<n$.

%
\begin{prop}[(Sampling error, $s<n$)]
\label{prop:var1steps}
Suppose there exist $\varepsilon,\kappa>0$ with
\[
\varepsilon\le p^v \bigl(x,z^v \bigr) \le
\varepsilon^{-1},\qquad \kappa\le g^v \bigl(x^v,y^v
\bigr) \le\kappa^{-1} \qquad\forall v\in V, x,z\in\bbX, y\in\bbY.
\]
Then
\[
\max_{K\in\mathcal{K}} \mathbf{E} \bigl[\bigl \|\mathsf{\tilde
F}_{s+1} \mathsf{\tilde F}_s\hat\pi_{s-1}^\mu-
\mathsf{\tilde F}_{s+1} \mathsf{\hat F}_s\hat
\pi_{s-1}^\mu\bigr\|_K^2
\bigr]^{1/2} \le\frac{
16\Delta_\mathcal{K} \varepsilon^{-2|\mathcal{K}|_\infty}
\kappa^{-4|\mathcal{K}|_\infty\Delta_\mathcal{K}}
}{\sqrt{N}}
\]
for every $0<s<n$.
\end{prop}

\begin{pf}
We begin by bounding using Lemma~\ref{lem:weight}
\begin{eqnarray*}
\bigl\|\mathsf{\tilde F}_{s+1} \mathsf{\tilde F}_s\hat
\pi_{s-1}^\mu- \mathsf{\tilde F}_{s+1} \mathsf{
\hat F}_s\hat\pi_{s-1}^\mu\bigr\|_K &=&
\bigl\| \mathsf{C}_{s+1}^K\mathsf{B}^K\mathsf{P}
\mathsf{\tilde F}_s\hat\pi_{s-1}^\mu-
\mathsf{C}_{s+1}^K\mathsf{B}^K\mathsf{P}
\mathsf{\hat F}_s\hat\pi_{s-1}^\mu\bigr\|
\\
&\le&2\kappa^{-2|\mathcal{K}|_\infty} \bigl\| \mathsf{B}^K\mathsf{P} \mathsf{
\tilde F}_s\hat\pi_{s-1}^\mu-
\mathsf{B}^K\mathsf{P} \mathsf{\hat F}_s\hat
\pi_{s-1}^\mu\bigr\|.
\end{eqnarray*}
Now note that
\begin{eqnarray*}
&&\frac{(\mathsf{B}^K\mathsf{P}
\mathsf{\tilde F}_s\hat\pi_{s-1}^\mu)(dx^K)}{
\psi^K(dx^K)
}\\
&&\qquad =
 \frac{
\int\prod_{v\in K} p^v(z,x^v)
\prod_{K'\in N(K)}\prod_{v'\in K'}g^{v'}(z^{v'},Y_s^{v'})
(\mathsf{B}^{K'}\mathsf{P}\hat\pi_{s-1}^\mu)(dz^{K'})}{
\int
\prod_{K'\in N(K)}\prod_{v'\in K'}g^{v'}(z^{v'},Y_s^{v'})
(\mathsf{B}^{K'}\mathsf{P}\hat\pi_{s-1}^\mu)(dz^{K'})},
\\
&&\frac{(\mathsf{B}^K\mathsf{P}
\mathsf{\hat F}_s\hat\pi_{s-1}^\mu)(dx^K)}{
\psi^K(dx^K)
} \\
&&\qquad= \frac{
\int\prod_{v\in K} p^v(z,x^v)
\prod_{K'\in N(K)}\prod_{v'\in K'}g^{v'}(z^{v'},Y_s^{v'})
(\mathsf{B}^{K'}\mathsf{S}^N\mathsf{P}\hat\pi_{s-1}^\mu)(dz^{K'})}{
\int
\prod_{K'\in N(K)}\prod_{v'\in K'}g^{v'}(z^{v'},Y_s^{v'})
(\mathsf{B}^{K'}\mathsf{S}^N\mathsf{P}\hat\pi_{s-1}^\mu)(dz^{K'})},
\end{eqnarray*}
where $\psi^K(dx^K):=\prod_{v\in K}\psi^v(dx^v)$, and we can write
\begin{eqnarray*}
&&\bigl\| \mathsf{B}^K\mathsf{P} \mathsf{\tilde F}_s\hat
\pi_{s-1}^\mu- \mathsf{B}^K\mathsf{P} \mathsf{
\hat F}_s\hat\pi_{s-1}^\mu\bigr\| \\
&&\qquad=
 \int\biggl| \frac{(\mathsf{B}^K\mathsf{P}
\mathsf{\tilde F}_s\hat\pi_{s-1}^\mu)(dx^K)}{
\psi^K(dx^K)
}- \frac{(\mathsf{B}^K\mathsf{P}
\mathsf{\hat F}_s\hat\pi_{s-1}^\mu)(dx^K)}{
\psi^K(dx^K)
} \biggr| \psi^K
\bigl(dx^K \bigr).
\end{eqnarray*}
We therefore have by Minkowski's integral inequality
\begin{eqnarray*}
&& \mathbf{E} \bigl[ \bigl\| \mathsf{B}^K\mathsf{P} \mathsf{\tilde
F}_s\hat\pi_{s-1}^\mu- \mathsf{B}^K
\mathsf{P} \mathsf{\hat F}_s\hat\pi_{s-1}^\mu
\bigr\|^2 \bigr]^{1/2}
\\
&&\qquad \le\int\mathbf{E} \biggl[\biggl | \frac{(\mathsf{B}^K\mathsf{P}
\mathsf{\tilde F}_s\hat\pi_{s-1}^\mu)(dx^K)}{
\psi^K(dx^K)
}- \frac{(\mathsf{B}^K\mathsf{P}
\mathsf{\hat F}_s\hat\pi_{s-1}^\mu)(dx^K)}{
\psi^K(dx^K)
}
\biggr|^2 \biggr]^{1/2}\psi^K \bigl(dx^K
\bigr)
\\
&&\qquad \le\psi^K \bigl(\bbX^K \bigr)\sup
_{x^K\in\bbX^K} \mathbf{E} \biggl[\biggl | \frac{(\mathsf{B}^K\mathsf{P}
\mathsf{\tilde F}_s\hat\pi_{s-1}^\mu)(dx^K)}{
\psi^K(dx^K)
}-
\frac{(\mathsf{B}^K\mathsf{P}
\mathsf{\hat F}_s\hat\pi_{s-1}^\mu)(dx^K)}{
\psi^K(dx^K)
} \biggr|^2 \biggr]^{1/2}.
\end{eqnarray*}
As we have
\[
\varepsilon\psi^v \bigl(\bbX^v \bigr)\le\int
p^v \bigl(x,z^v \bigr)\psi^v
\bigl(dz^v \bigr)=1, \qquad\prod_{v\in K}
p^v \bigl(z,x^v \bigr)\le\varepsilon^{-|\mathcal{K}|_\infty}
\]
and
\[
\kappa^{|\mathcal{K}|_\infty\Delta_{\mathcal{K}}} \le\prod_{K'\in N(K)}\prod
_{v'\in K'}g^{v'} \bigl(z^{v'},Y_s^{v'}
\bigr) \le\kappa^{-|\mathcal{K}|_\infty\Delta_{\mathcal{K}}},
\]
we can apply Lemma~\ref{lem:weight} to estimate
\begin{eqnarray*}
&&\mathbf{E} \bigl[ \bigl\| \mathsf{B}^K\mathsf{P} \mathsf{\tilde
F}_s\hat\pi_{s-1}^\mu- \mathsf{B}^K
\mathsf{P} \mathsf{\hat F}_s\hat\pi_{s-1}^\mu
\bigr\|^2 \bigr]^{1/2}
\\
&&\qquad \le2\varepsilon^{-2|\mathcal{K}|_\infty} \kappa^{-2|\mathcal
{K}|_\infty\Delta_{\mathcal{K}}} {\biggl|\!\biggl|\!\biggl| \bigotimes
_{K'\in
N(K)} \mathsf{B}^{K'}\mathsf{P}\hat
\pi_{s-1}^\mu- {\bigotimes }_{K'\in N(K)}
\mathsf{B}^{K'}\mathsf{S}^N\mathsf{P}\hat \pi
_{s-1}^\mu \biggr|\!\biggr|\!\biggr|}.
\end{eqnarray*}
By Corollary~\ref{cor:blocksampling} (applied conditionally given
$\hat\pi_{s-1}^\mu$), we obtain
\[
\mathbf{E} \bigl[ \bigl\| \mathsf{B}^K\mathsf{P} \mathsf{\tilde
F}_s\hat\pi_{s-1}^\mu- \mathsf{B}^K
\mathsf{P} \mathsf{\hat F}_s\hat\pi_{s-1}^\mu
\bigr\|^2 \bigr]^{1/2} \le\frac{
8\Delta_\mathcal{K} \varepsilon^{-2|\mathcal{K}|_\infty}
\kappa^{-2|\mathcal{K}|_\infty\Delta_{\mathcal{K}}}
}{\sqrt{N}}.
\]
The result follows immediately.
\end{pf}

We finally put everything together.

%
\begin{thmm}[(Variance term)]
\label{thmm:variance}
Suppose there exist $\varepsilon,\kappa>0$ with
\[
\varepsilon\le p^v \bigl(x,z^v \bigr) \le
\varepsilon^{-1},\qquad \kappa\le g^v \bigl(x^v,y^v
\bigr) \le\kappa^{-1}\qquad \forall v\in V, x,z\in\bbX, y\in\bbY
\]
such that
\[
\varepsilon>\varepsilon_0= \biggl(1-\frac{1}{6\Delta_\mathcal
{K}\Delta^2}
\biggr)^{1/2\Delta}.
\]
Let $\beta=-\log6\Delta_\mathcal{K}\Delta^2(1-\varepsilon
^{2\Delta})>0$.
Then
\[
{\bigl|\!\bigl|\!\bigl| \tilde\pi_n^x-\hat\pi_n^x
\bigr|\!\bigr|\!\bigr|}_J \le\card J \frac{64\Delta_\mathcal{K} e^\beta}{1-e^{-\beta}} \frac{
\varepsilon^{-4|\mathcal{K}|_\infty}
\kappa^{-4|\mathcal{K}|_\infty\Delta_\mathcal{K}}
}{\sqrt{N}}
\]
for every $n\ge0$, $x\in\bbX$, $K\in\mathcal{K}$ and $J\subseteq K$.
\end{thmm}

\begin{pf}
We begin with the elementary error decomposition
\[
{\bigl|\!\bigl|\!\bigl| \tilde\pi_n^x - \hat\pi_n^x
\bigr|\!\bigr|\!\bigr|}_J \le\sum_{s=1}^n
{\bigl|\!\bigl|\!\bigl| \mathsf{\tilde F}_n\cdots\mathsf{\tilde F}_{s+1}
\mathsf{\tilde F}_s \hat\pi_{s-1}^x- \mathsf{
\tilde F}_n\cdots\mathsf{\tilde F}_{s+1}\mathsf{\hat
F}_s \hat\pi _{s-1}^x \bigr|\!\bigr|\!\bigr|}_J.
\]
The term $s=n$ in this sum is bounded in Lemma~\ref{lem:var1stepn}:
\[
{\bigl|\!\bigl|\!\bigl| \mathsf{\tilde F}_n\hat\pi_{n-1}^x-
\mathsf{\hat F}_n\hat \pi_{n-1}^x \bigr|\!\bigr|\!\bigr|}_J \le\frac{2\kappa^{-2|\mathcal{K}|_\infty}}{\sqrt{N}}.
\]
The term $s=n-1$ is bounded in Proposition~\ref{prop:var1steps}:
\[
{\bigl|\!\bigl|\!\bigl| \mathsf{\tilde F}_n\mathsf{\tilde F}_{n-1} \hat
\pi _{s-1}^x- \mathsf{\tilde F}_n\mathsf{\hat
F}_{n-1} \hat\pi_{s-1}^x \bigr|\!\bigr|\!\bigr|}_J
\le \frac{
16\Delta_\mathcal{K} \varepsilon^{-2|\mathcal{K}|_\infty}
\kappa^{-4|\mathcal{K}|_\infty\Delta_\mathcal{K}}
}{\sqrt{N}}.
\]
Now suppose $s<n-1$. Then we can estimate using Corollary~\ref{cor:blockfstab}
\begin{eqnarray*}
&& {\bigl|\!\bigl|\!\bigl| \mathsf{\tilde F}_n\cdots\mathsf{\tilde
F}_{s+1} \mathsf {\tilde F}_s \hat\pi_{s-1}^x-
\mathsf{ \tilde F}_n\cdots\mathsf {\tilde F}_{s+1}\mathsf{
\hat F}_s \hat\pi_{s-1}^x \bigr|\!\bigr|\!\bigr|}_J
\\
&&\qquad \le\frac{4}{\varepsilon^{2|\mathcal{K}|_\infty}} \card J e^{-\beta(n-s-1)} \max_{K\in\mathcal{K}}
\mathbf{E} \bigl[\bigl\| \mathsf{\tilde F}_{s+1}\mathsf{\tilde
F}_s\hat\pi_{s-1}^x- \mathsf{\tilde
F}_{s+1}\mathsf{\hat F}_s\hat\pi_{s-1}^x
\bigr\|_K^2 \bigr]^{1/2}.
\end{eqnarray*}
Applying Proposition~\ref{prop:var1steps} yields
\begin{eqnarray*}
&& {\bigl|\!\bigl|\!\bigl| \mathsf{\tilde F}_n\cdots\mathsf{\tilde
F}_{s+1} \mathsf {\tilde F}_s \hat\pi_{s-1}^x-
\mathsf{ \tilde F}_n\cdots\mathsf {\tilde F}_{s+1}\mathsf{
\hat F}_s \hat\pi_{s-1}^x \bigr|\!\bigr|\!\bigr|}_J
\\
&&\qquad \le\card J e^{-\beta(n-s-1)} \frac{
64\Delta_\mathcal{K} \varepsilon^{-4|\mathcal{K}|_\infty}
\kappa^{-4|\mathcal{K}|_\infty\Delta_\mathcal{K}}
}{\sqrt{N}}.
\end{eqnarray*}
Substituting the above three cases into the error decomposition and
summing the geometric series yields the statement of the theorem.
\end{pf}

Theorems \ref{thmm:bias} and \ref{thmm:variance} now immediately
yield Theorem~\ref{thmm:main}.



%





\printaddresses
\end{document}